\documentclass[a4paper,11pt]{article}
\usepackage{amsmath,amssymb,amsthm,siunitx}
\usepackage{cases}
\usepackage{a4wide}
\usepackage[pdftex]{graphicx}
\usepackage[hang,small,bf]{caption}
\usepackage[subrefformat=parens]{subcaption}
\usepackage{here}
\usepackage{comment}
\usepackage{ulem}
\usepackage[pdftex]{color}

\newtheorem{prop}{Proposition}
\newtheorem{thm}{Theorem}
\newtheorem{lem}{Lemma}
\newtheorem{rem}{Remark}
\newtheorem{cor}{Corollary}

\newcommand{\Add}[1]{\textcolor{black}{#1}}
\newcommand{\Erase}[1]{\if0{#1}\fi}
\newcommand{\cpoint}[4]{\mathbf{#1}^{#2}_{#3,#4}}

\title{\textbf{A rapid numerical method for the Mullins-Sekerka flow with application to contact angle problems}}
\author{Tokuhiro Eto\footnote{Email-address:tokuhiro\_eto@yahoo.co.jp}}
\date{\today}
\begin{document}
\maketitle
\begin{abstract}
    The Mullins-Sekerka problem is numerically solved in $\mathbb{R}^2$ with the aid of the charge simulation method.
    This is an expansion of the numerical scheme by which Sakakibara and Yazaki computed the Hele-Shaw flow.
    We investigate a sufficient condition for the number of collocation points 
    to ensure that the length of the generated approximate polygonal curves gradually decreases.
    \Add{We propose a new benchmark function for the Mullins-Sekerka flow to confirm that the scheme works well.}
    Moreover, by changing the fundamental solutions of the charge simulation method, we are successful to establish a numerical scheme 
    that can be used to treat the Mullins-Sekerka problem with the contact angle condition.
\end{abstract}

\section{Introduction}\label{sec:intro}

The Mullins-Sekerka problem is a quasi-stationary Stefan problem with surface tension.
Its solution describes the time evolution of the interface separating two-phases filled with 
different materials. At a time $t$, one material is filled in a smooth bounded region $\Omega_t$ and another material is filled
in the outer region $\mathbb{R}^2\backslash\overline{\Omega_t}$. The boundary $\Gamma_t = \partial\Omega_t$ represents the interface.
The temperature of the material is denoted by $u = u(t,x)$ for $x\in\mathbb{R}^2\backslash\Gamma_t$ and $t\geq 0$.
Then, the Mullins-Sekerka problem asks us to find $(u,\Gamma_t)$ that satisfies
\begin{equation}\label{eq:Mullins_Sekerka_intro}
    \begin{cases}
        \Delta u = 0\ \ \mbox{in}\ \ \mathbb{R}^2\backslash\Gamma_t, \\
        u = \kappa\ \ \mbox{on}\ \ \Gamma_t, \\
        \nabla u(x) = O\left(\frac{1}{|x|^2}\right)\ \ \mbox{as}\ \ |x|\rightarrow\infty, \\
        V = -\left[\frac{\partial u}{\partial\mathbf n}\right]\ \ \mbox{on}\ \ \Gamma_t,
    \end{cases}
\end{equation}
where $\mathbf{n}$ denotes the normal vector to $\Gamma_t$ outgoing from $\Omega_t$, and $\kappa$ denotes the curvature of $\Gamma_t$.
The normal velocity of $\Gamma_t$ is denoted by $V$. More precisely, the motion of $\Gamma_t$ is governed by
\begin{equation}\label{eq:ContinuousEvolutionLaw}
  \frac{d}{dt}\mathbf{X}(t,s) = V(t,\mathbf{X}(t,s))\mathbf{n}(t,\mathbf{X}(t,s))\ \ (t\geq 0, s\in [0,2\pi]),
\end{equation}
where the interface $\Gamma_t$ is parameterized as $\mathbf{X}(t,s)$ for $s\in[0,2\pi]$.
$V$ describes the speed of $\Gamma_t$ in the direction to $\mathbf{n}$.
Let $\left[\varphi\right]$ be the jump in the normal direction of a quantity $\varphi$ across $\Gamma_t$; namely,
$\left[\varphi\right](x) := \lim_{\varepsilon\rightarrow 0}\{\varphi(x - \varepsilon\mathbf{n}) - \varphi(x + \varepsilon\mathbf{n})\}$ for each $x\in\Gamma_t$.\newline

The Mullins-Sekerka equation is a limit of a phase-field model, 
which is called the Cahn-Hilliard equation where two phases are separated by a transition layer instead of a \Erase{shape}\Add{sharpe} interface.
The Cahn-Hilliard equation is important to understand spinodal decomposition, which explains a phenomenon that
compounded solutes and solids are stable at high temperatures but become unstable at low temperatures and eventually separated by sharp interfaces.
The convergence of the Cahn-Hilliard equation to the Mullins-Sekerka equation was first formally shown by Pego \cite{Pego1989}. 
Stoth \cite{Stoth} gave a rigorous proof of the convergence in the case where the domain under consideration is a ball in $\mathbb{R}^3$, and the initial data and 
the boundary values are all radially symmetric. In a general dimension, Alikakos et al. \cite{AlikakosBatesChen1994} proved that a family of smooth solutions to the Cahn-Hilliard equation 
tends to a smooth solution to the Mullins-Sekerka equation provided that the latter exists.\newline

The Mullins-Sekerka problem has been well studied from an analytical point of view.
Chen et al. \cite{ChenHongYi1996} proved the existence of a classical solution to the Mullins-Sekerka problem local in time in the
two dimensional case, and Escher and Simonett \cite{EScherSimonett} proved it in general dimensional cases.
In the literature on weak solutions, Luckhaus and Sturzenhecker \cite{LuckhausSturzenhecker} proposed
a weak notion of the solution to the Mullins-Sekerka problem and gave its global time existence result whenever
the sum of the surface area measure does not change discontinuously over time. 
\Add{Under the same assumption, Bronsard et al. \cite{BronsardGarckeStoth1998} established a weak solution of a multi-phase Mullins-Sekerka problem
with a triple junction and showed its existence.} As indicated in our experiment, this setting cannot be applied
when multi particles exist and are very close to each other (see Section \ref{sec:CoP}).
R{\"o}ger \cite{Rger2005ExistenceOW} was successful to remove this assumption in terms of
geometric measure theory, \Add{although his result could be applied only to the two-phase case and excluded the multi-phase case and contact angle case.}
\Add{Recently, Hensel and Stinson \cite{HenselStinson202206} proposed a varifold solution to
the Mullins-Sekerka problem based on the energy dissipation property and included a fixed contact angle condition.
Julin et al. \cite{JulinEtal2022} revealed an asymptotic behaviour of the Mullins-Sekerka flow in the torus $\mathbb{T}^2\subset\mathbb{R}^2$.
They proved that a flat flow solution to the Mullins-Sekerka problem exponentially converges to the finite union of disks
whenever the perimeter of the initial data is less than $2$ (see Theorem 1.3 \cite{JulinEtal2022}).}\newline

There are several works that treated the Hele-Shaw problem or the Mullins-Sekerka problem numerically.
Though a typical Mullins-Sekerka problem is considered in a smooth bounded domain, 
Bates et al. \cite{Bates1995ANS} considered the problem $\eqref{eq:Mullins_Sekerka_intro}$ in $\mathbb{R}^2$
and translated the original problem $\eqref{eq:Mullins_Sekerka_intro}$ into a corresponding boundary integral equation.
Eventually, they did not have to care about the boundary of the container;
they split an initial curve into several segments and regarded as a part of circles.
In this way, they could calculate a discrete version of the curvatures and derive a linear system of equations 
whose unknown variables are normal velocities at each vertex by means of the Gibbs-Thomson law (this is the second condition of $\eqref{eq:Mullins_Sekerka_intro}$). 
Recall that our unknown variables in the linear system are coefficients of fundamental solutions in contrast to their scheme (see $\eqref{eq:internal_problem}$ and $\eqref{eq:external_problem}$).
Moreover, they adopted the semi-implicit method to stabilize their scheme with a small time step when solving the linear system.
For other studies using the boundary integral method, we refer the reader to \cite{ZhuChenHou}, \cite{Mayer2000}, and \cite{BatesBrown}.
Barrett et al. \cite{BarrettGarckeRobert} proposed a parametric finite element scheme for the Stefan problem with surface tension and 
proved the well-posedness and stability of their scheme. In the course of the discussion, they also revealed that
the scheme applies to the Mullins-Sekerka problem. Feng and Prohl \cite{FengProhl2005} showed that their numerical scheme, 
the so-called the fully discrete mixed finite element scheme to construct discrete solutions to the Cahn-Hilliard equation, tends to the solution of the 
Mullins-Sekerka equation, provided that a global-in-time classical solution exists.
\Add{Recently, N{\"u}rnberg \cite{Nurunberg202203} introduced a front tracking method for the Mullins-Sekerka flow which is based on 
the finite element method. He proved its unconditional stability and volume conservation law of the scheme. Its accuracy was confirmed in terms of two concentric circles.}\newline

The purpose of this paper is to propose a discrete scheme to solve the Mullins-Sekerka problem numerically,
revealing that our scheme has some desiable properties. To this end, we follow the scheme proposed by Sakakibara and Yazaki \cite{Sakakibara_Yazaki}.
Moreover, we confirm that the proposed scheme possesses curve-shortening property (CS) and area-preserving (AP) properties.
These facts are predictable because the original scheme for the Hele-Shaw problem also has such properties.
However, we focus on CS and derive a discrete variant of the estimation related to the length of the curve.
This outcome is reported in Corollary \ref{cor:time_step_optimization}.\newline

At this stage, we provide a brief explanation of our scheme.
Suppose that a smooth curve $\Gamma_t$ is given for some $t>0$.
Then, $\Gamma_t$ is approximated by an $N$ polygon $\Gamma_t^N$ where $N\geq 3$ is a positive integer. 
The interior domain of $\Gamma_t^N$ is designated as $\Omega_t^N$.
The vertices of $\Gamma_t^N$ are denoted by $\mathbf{X}_i(1\leq i \leq N)$.
For convenience, we adopt a periodic rule for the indices of $\mathbf{X}_i$, such as $\mathbf{X}_0 = \mathbf{X}_N$ and $\mathbf{X}_{N+1} = \mathbf{X}_1$.
Each edge $[\mathbf{X}_{i-1},\mathbf{X}_i]$ is expected to possess a discrete version of the curvature
$\kappa_i = \kappa_i(t)(1\leq i \leq N)$. In addition, normal vectors $\mathbf{N}_i=\mathbf{N}_i(t)$ and tangential vectors $\mathbf{T}_i = \mathbf{T}_i(t)$
at $\mathbf{X}_i$ are suitably defined.
Among the standard numerical methods, we adopt the charge simulation method (CSM)
that was originally developed to approximate a solution to the Laplace equation in a bounded domain with Dirichlet boundary conditions.
CSM is a variant of the method of fundamental solutions (MFS), in which 
approximate solutions are expressed as a linear combination of fundamental solutions to partial differential equations under consideration.
CSM requires choosing proper charge points $y_1^+,\cdots,y_N^+$ from $\mathbb{R}^2\backslash \overline{\Omega^N_t}$ and
collocation points $\mathbf{X}_1,\cdots,\mathbf{X}_N$ on $\Gamma_t$ and expresses an approximate solution as follows:
\begin{equation}\label{eq:combination}
  U^+(x) = \sum_{i=1}^N Q_i^+E(x - y_i^+)
\end{equation}
where $E$ is the fundamental solution of the Laplace equation, that is $E(x):= \frac{1}{2\pi}\log{|x|}$ for $x\in\mathbb{R}^2\backslash\{0\}$.
Since $y_i^+$ s are outside $\Omega^N$, $U^+$ is harmonic at all points in $\Omega^N$.
The coefficients $Q_i$ s should be determined by Dirichlet boundary conditions
$U(\mathbf{X}_i) = \kappa(\mathbf{X}_i)$ for $1\leq i \leq N$.
This is the basic idea of CSM. For more detail about CMS, see Katsurada and Okamoto \cite{KatsuradaOkamoto}.
The above conventional scheme was modified by Murota \cite{MurotaKazuo} to make the scheme possess 
invariance properties that original continuous problems also have. Concretely, he alternatively used
the following combination.
\begin{equation*}
  U^+(x) = Q_0^+ + \sum_{i=1}^N Q_i^+E(x-y_i^+).
\end{equation*}
In this scheme, we have to add an equality to a linear system because we have one more value $Q^+_0$ to find.
For instance, Murota assumed that the sum of $Q_i^+$ s equals zero.
In solving the Hele-Shaw problem numerically, Sakakibara and Yazaki \cite{Sakakibara_Yazaki} improved Murota's invariant scheme to 
make this additional assumption more natural. They defined dummy singular points $z_i^+(1\leq i\leq N)$
and replaced the combination of fundamental solutions by
\begin{equation}\label{eq:internal_problem_intro}
  U^+(x) = Q_0^+ + \sum_{i=1}^N Q_i^+\{E(x-y_i^+)-E(x-z_i^+)\}.
\end{equation}
It can be observed that the above function $U^+$ is invariant 
under translation and scaling without any additional assumptions.
Hence, it is possible to impose an area-preserving requirement that seems more natural
than the zero-average assumption.
Taking singular points $y_i^-$ and $z_i^-$ from $\Omega^N_t$, we also have a function $U^-$ being harmonic in $\mathbb{R}^2\backslash\overline{\Omega^N_t}$
that satisfies the Dirichlet boundary condition.
We should solve an external potential problem to find such a $U^-$. However, it is impossible by either the finite difference method or the finite element method 
owing to the unboundedness of the domain where the problem is considered.
As imposed in the fourth equality of $\eqref{eq:Mullins_Sekerka_intro}$, each point $x$ on $\Gamma^N_t$ is required to move at the speed that is equal to the jump of the normal derivatives of $U^+$ and $U^-$ across $\Gamma^N_t$.
Once we obtain such $U^+$ and $U^-$, the direct differentiation of $U^+$ and $U^-$ yields the representative normal velocity $V_i(t) = V_i^+(t) + V_i^-(t)$ at $\mathbf{X}_i$.
Consequently, $\mathbf{X}_i$ should fulfill the evolution equation
\begin{equation}\label{eq:evolution_equation}
  \frac{d}{dt}{\mathbf{X}_i(t)} = V_i(t)\mathbf{N}_i(t) + W_i(t)\mathbf{T}_i(t)\ \ \mbox{for}\ \ 1\leq i\leq N,\ t > 0.
\end{equation}
Tangential velocity $W_i$ and its vector $\mathbf{T}_i$ are required to stabilize the scheme, although
they have no effect on the geometry of the curve (see Proposition 2.4 \cite{Epstein} for instance).
Normal velocity $V_i$ and its vector $\mathbf{N}_i$ definitely control the motion of a curve.
Finally, we discretize the time variable as $t = n\Delta t (0\leq n\leq N)$ and rearrange the evolution equation $\eqref{eq:evolution_equation}$ as follows:
\begin{equation}\label{eq:evolution_equation_discrete}
  \mathbf{X}^{n+1}_i = \mathbf{X}^n_i + \Delta t(V_i(t_n)\mathbf{N}_i(t_n) + W_i(t_n)\mathbf{T}_i(t_n))\ \ \mbox{for}\ \ 1\leq i\leq N,\ n = 0,1,\cdots.
\end{equation}

A particular novelty of this study is treating a boundary contact case of the Mullins-Sekerka problem in the half plane $\mathbb{R}_+^2$.
To this end, we replace the fundamental solutions of the combination by the Green function on $\mathbb{R}_+^2$.
Since the curves under consideration are open, we modify the structure of the proposed scheme.
Well-posedness of the Mullins-Sekerka problem with ninety contact angle condition was established by Abels et al. \cite{AbelsMaxWilke}.
\Add{After that, Garcke and Rauchecker \cite{GarckeRauchecker2022} show stability and instability results of stationary solutions to the linearized problem that is either flat or a part of a circle.}\newline

The reminder of this paper is organized as follows. In Section \ref{sec:NumeticalScheme}, we rigorously state how to implement the proposed scheme.
Section \ref{sec:PropertiesOfTheScheme} is devoted to list our main results without the proofs. In Section \ref{sec:NumericalExamples}, we give several examples of implementation of the scheme. 
Moreover, the accuracy of our scheme is confirmed in terms of an annulus-like domain \Add{which consists of three concentric circles} \Erase{as an initial data.} 
\Add{and a continuous function being harmonic except on the circles.}
\Add{To our best knowledge, this is a new feature in the literature of the Mullins-Sekerka problem as a benckmark function
that can describe the two-phase motion.} We shall extend the scheme to the boundary contact cases in Section \ref{sec:ExtensionToBoundaryContactCases}.
In Section \ref{sec:SequenceOfTheProofs}, we collect all proofs of Theorems and Propositions whose justification has been postponed.

\section{Numerical scheme}\label{sec:NumeticalScheme}

In this section, we present a concrete procedure to construct approximate polygons.
This is a natural extension of the scheme proposed in \cite{Sakakibara_Yazaki}.

\subsection{Polygonal approximation of the interface}
Let $N\geq 3$ be a positive number, and $\Omega^N$ be an $N$ polygonal domain in $\mathbb{R}^2$ with vertices $\mathbf{X}_1,\cdots,\mathbf{X}_N\in\mathbb{R}^2$.
Set $\mathbf{X}_0 := \mathbf{X}_N, \mathbf{X}_{N+1} := \mathbf{X}_1$. This periodic rule is always applied unless otherwise stated explicitly. 
The symbol $[\mathbf{X}_{i-1},\mathbf{X}_i]$ denotes the line segment connecting $\mathbf{X}_{i-1}$ and $\mathbf{X}_i$.
This is called the edge in the sequel. Then, the boundary $\Gamma^N := \partial\Omega^N$ readily designs an $N$ polygon 
and is expressed as $\Gamma^N = \cup_{i=1}^N[\mathbf{X}_{i-1},\mathbf{X}_i]$.
For each $1\leq i\leq N$, we define 
\begin{equation}\label{eq:rtn}
r_i:=|\mathbf{X}_i - \mathbf{X}_{i-1}|,\ \mathbf{t}_i := \frac{\mathbf{X}_i - \mathbf{X}_{i-1}}{r_i},\ \mathbf{n}_i := \mathbf{t}_i^\perp.
\end{equation}
Here we have \Erase{use}\Add{used} the notation $(a,b)^\perp := (b, -a)$.
Moreover, the midpoint of the edge $[\mathbf{X}_{i-1}, \mathbf{X}_i]$ is denoted by $\mathbf{X}_i^*$, namely $\mathbf{X}_i^* := \frac{\mathbf{X}_{i-1} + \mathbf{X}_i}{2}$.
The outer angle $\varphi_i$ of $\Gamma^N$ at each vertex $\mathbf{X}_i$ is expressed as follows:
\begin{equation*}
    \varphi_i := \operatorname{sgn}{(\mathbf{t}_i, \mathbf{t}_{i+1})} \arccos{(\mathbf{t}_i\cdot\mathbf{t}_{i+1})}.
\end{equation*}
where the function $\operatorname{sgn}$ is defined by
\begin{equation*}
    \operatorname{sgn}{(\mathbf{a},\mathbf{b})} :=
    \begin{cases}
        1\ \ \mbox{if}\ \ \mathbf{a}\cdot\mathbf{b}^\perp > 0,\\
        0\ \ \mbox{if}\ \ \mathbf{a}\cdot\mathbf{b}^\perp = 0,\\
        -1\ \ \mbox{otherwise},
    \end{cases}
\end{equation*}
for each $\mathbf{a},\mathbf{b}\in\mathbb{R}^2$.
By using $\varphi_i$, we set
\begin{equation*}
    \cos_i := \cos{\left(\frac{\varphi_i}{2}\right)},\ \sin_i := \sin{\left(\frac{\varphi_i}{2}\right)},\ \tan_i := \frac{\sin_i}{\cos_i}.
\end{equation*}
Then, we define the discrete curvature of $\Gamma^N$ at $\mathbf{X}_i^*$ as follows:
\begin{equation}\label{eq:DiscreteCurvature}
    \kappa_i := \frac{\tan_i + \tan_{i-1}}{r_i}\ \ \mbox{for}\ \ i = 1,\cdots,N.
\end{equation}
Moreover, we define the normal vectors $\mathbf{N}_i$ and the tangential vectors $\mathbf{T}_i$ at $\mathbf{X}_i$
in terms of the normal vectors $\mathbf{n}_i$ and the tangential vectors $\mathbf{t}_i$ of the edge $[\mathbf{X}_{i-1},\mathbf{X}_i]$ as follows:
\begin{equation*}
    \mathbf{N}_i := \frac{\mathbf{n}_i + \mathbf{n}_{i+1}}{2\cos_i},\ \mathbf{T}_i := \frac{\mathbf{t}_i + \mathbf{t}_{i+1}}{2\cos_i}\ \ \mbox{for}\ \ i = 1,\cdots,N.
\end{equation*}
See the figure below to take a look at our setting.

\begin{figure}[H]
    \centering
    \includegraphics[keepaspectratio, scale=0.5]{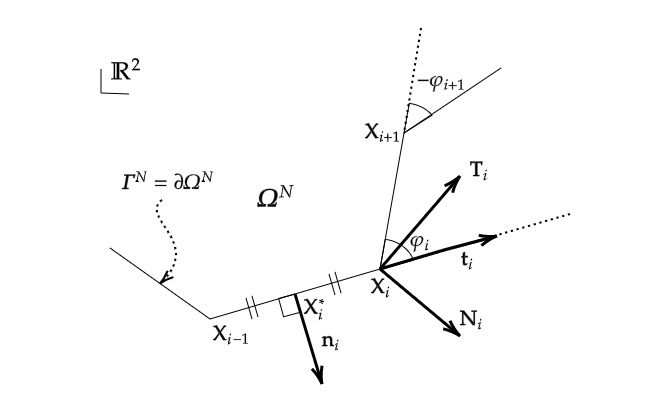}
    \caption{An $N$ polygon mimicking a curve.}\label{fig:N_polygon}
\end{figure}

\subsection{Approximation of the normal velocity}
  \begin{description}
    \item[Solve the internal problem.] 
    In this step, we solve the Laplace equation with a Dirichlet boundary condition given by
    $\kappa_i$.
    To this end, we must choose singular points $y_1^+,\cdots,y_N^+$ and dummy singular points $z_1^+,\cdots,z_N^+$ from $\mathbb{R}^2\backslash\overline{\Omega\Add{^N}}$.
    \Erase{For small $d>0$, w}\Add{W}e set
    \begin{equation*}
        y_i^+ := \mathbf{X}_i^* + d\mathbf{n}_i,\ z_i^+ := \mathbf{X}_i^* + \frac{d}{2}\mathbf{n}_i
    \end{equation*}

    \Add{where $d := \frac{1}{\sqrt{N}}$.}
    Then, suppose that the solution of the Dirichlet boundary problem is of the form 
    \begin{equation}\label{eq:def_of_U}
        U^+(x) = Q^+_0 + \sum_{i=1}^NQ_i^+\{E(x - y^+_i) - E(x-z^+_i)\}
    \end{equation}
    for some $Q_0^+,Q_1^+,\cdots,Q_N^+\in\mathbb{R}$. $U^+$ is clearly harmonic in $\Omega^N$ due to its structure. 
    To determine these values, we shall solve the following linear system of equations:
    \begin{equation}\label{eq:internal_problem}
        \begin{cases}
            Q_0^+ + \sum_{j=1}^NQ_j^+\mathbb{G}_{i,j} = \kappa_i\ \ \mbox{for}\ \ i = 1,\cdots,N, \\
            \sum_{j=1}^NQ_j^+H_j = 0.
        \end{cases}
    \end{equation}
    where 
    \begin{eqnarray*}
        \mathbb{G}_{i,j} &:=& E(\mathbf{X}_i^* - y_j) - E(\mathbf{X}_i^* - z_j), \\
        H_j &:=& - \sum_{i=1}^N\mathbf{H}_{i,j}\cdot \mathbf{n}_ir_i, \\
        \mathbf{H}_{i,j} &:=& \nabla E(\mathbf{X}_i^* - y^+_j) - \nabla E(\mathbf{X}_i^* - z^+_j).
    \end{eqnarray*}

    \item[Solve the external problem.] As in the previous step, 
    we take singular points $y^-_1,\cdots,y^-_N$ and dummy singular points $z^-_1,\cdots,z^-_N$ in $\Omega^N$ as follows:
    \begin{equation*}
        {y}_j^- := \mathbf{X}_i^* - d\mathbf{n}_i,\ {z}_j^- := \mathbf{X}_i^* - \frac{d}{2}\mathbf{n}_i.
    \end{equation*}
    Then, the solution to the exterior Dirichlet boundary problem is expressed as
    \begin{equation}\label{eq:def_of_U_bar}
        {U}^-(x) = {Q}^-_0 + \sum_{i=1}^N{Q}^-_i\{E(x - y_i^-) - E(x-z_i^-)\}
    \end{equation}
    where $Q_0^-, Q_1^-,\cdots,Q_N^-$ are selected to fulfill
    \begin{equation}
        \begin{cases}
            Q_0^- + \sum_{j=1}^N Q_j^-\mathbb{G}_{i,j} = \kappa_i\ \ \mbox{for}\ \ i = 1,\cdots,N, \\
            -\sum_{j=1}^NQ_j^-H_j = 0. \label{eq:external_problem}
        \end{cases}
    \end{equation}

    \item[Derive the normal velocity.] Once $Q^+_i$ and $Q^-_i$ are determined, we can define a representative normal velocity at $\mathbf{X}_i$.
    By setting $v_i^+ := -\nabla U^+(\mathbf{X}_i^*)\cdot\Add{\mathbf{n}_i}$ and $v_i^- := -\nabla U^-(\mathbf{X}_i^*)\cdot(-\mathbf{n}_i)$, we define the representative normal velocities by
    \begin{equation}\label{eq:normal_velocity}
        V_i^\pm := \frac{v_i^\pm + v_{i+1}^\pm}{2\cos_i}\ \ \mbox{for}\ \ i = 1,\cdots,N.
    \end{equation}

    \item[Derive the tangential velocity.] To reduce instability of the scheme, we consider the tangential vector
    that does not affect the shape of the polygon. To this end, we adopt a uniformly distribute method (UDM).
    Let us only state formulae to obtain the tangential velocities. We refer the readers to Sakakibara and Yazaki \cite{Sakakibara_Yazaki} for derivation of the scheme.

    Tangential velocities $W_i(1\leq i\leq N)$ are defined as follows:
    \begin{equation}\label{eq:W_i_explicit_formula}
        W_i := \frac{\Psi_i + C}{\cos_i}\ \ \mbox{for}\ \ 1\leq i\leq N,\ \ C := -\frac{\sum_{i=1}^N\frac{\Psi_i}{\cos_i}}{\sum_{i=1}^N\frac{1}{\cos_i}}.
    \end{equation}
    Here we have set $\Psi_i := \sum_{j=1}^i\psi_j$ for each $1\leq i\leq N$ where
    $\psi_1 := 0$ and 
    \begin{equation*}
        \psi_i := \frac{1}{N}\sum_{i=1}^N\kappa_i(v^+_i + v^-_i)r_i - (V^+_i+V^-_i)\sin_i - (V^+_{i-1}+V^-_{i-1})\sin_{i-1} + \left(\frac{L}{N} - r_i\right)\omega
    \end{equation*}
    for $2\leq i \leq N$ with $L := \sum_{i=1}^N r_i$ and $\omega := 10N$.
  \end{description}
\subsection{Time evolution of the polygonal interface}
    We are now in the position to state our numerical scheme.
    Given an initial curve $\Gamma_0 = \partial\Omega_0$, we approximate it by an $N$ polygonal curve
    \begin{equation*}
      \Gamma^N_0 = \bigcup_{i=1}^N[\mathbf{X}_{i-1}^{(0)},\mathbf{X}_i^{(0)}]
    \end{equation*}
    where $\mathbf{X}^{(0)}_i(1\leq i\leq N)$ denote vertices of $\Gamma^N_0$.
    The domain surrounded by $\Gamma^N_0$ is denoted by $\Omega^N_0$.
    For $t>0$, suppose that an $N$ polygon $\Gamma^N_t$ is an approximation of $\Gamma_t$ to be found.
    The vertices of $\Gamma^N_t$ are denoted by $\mathbf{X}_i = \mathbf{X}_i(t)(1\leq i\leq N)$.
    Quantities $\kappa_i = \kappa_i(t), \mathbf{N}_i=\mathbf{N}_i(t), \mathbf{T}_i = \mathbf{T}_i(t), V^+_i=V^+_i(t), V^-_i = V^-_i(t), V_i = V_i(t):=V^+_i+V^-_i$  and $W_i=W_i(t)$
    are defined as above. Following the continuous version of the evolution law, we consider
    \begin{equation}\label{eq:DiscreteEvolutionLaw}
      \frac{d}{dt}\mathbf{X}_i(t) = V_i(t)\mathbf{N}_i(t) + W_i(t)\mathbf{T}_i(t)\ \ \mbox{for}\ \ i=1,\cdots,N,\ t>0
    \end{equation}
    with the initial condition $\mathbf{X}_i(0) = \mathbf{X}_i^{(0)}(1\leq i \leq N)$.
    Finally, the time variable is discretized as $t_n = n\Delta t(n = 0,1,\cdots)$ 
    with a given time step $\Delta t>0$ and the forward Euler method is applied to $\eqref{eq:DiscreteEvolutionLaw}$.
    Using the notations $\mathbf{X}^n_i = \mathbf{X}_i(t_n),V_i^n=V_i(t_n),\mathbf{N}_i^n=\mathbf{N}_i(t_n),W_i^n=W_i(t_n)$ and $\mathbf{T}_i^n=\mathbf{T}_i(t_n)$, we provide the fully discrete scheme below.

    \begin{equation}\label{eq:FullyDiscreteScheme}
      \mathbf{X}^{n+1}_i = \mathbf{X}^n_i + \Delta t(V_i^n\mathbf{N}_i^n + W_i^n\mathbf{T}_i^n)\ \ \mbox{for}\ \ i = 1,\cdots,N,\ n=0,1,\cdots
    \end{equation}
    with the initial condition $\mathbf{X}^0_i = \mathbf{X}^{(0)}_i(1\leq i \leq N)$.
    An $N$ polygonal curve that consists of the vertices $\mathbf{X}^n_i(1\leq i \leq N)$ solving $\eqref{eq:FullyDiscreteScheme}$ 
    is denoted by $\Gamma^{N,n}_{\Delta t}$. The interior domain of $\Gamma^{N,n}_{\Delta t}$ is denoted by $\Omega^{N,n}_{\Delta t}$.

\section{Properties of the scheme}\label{sec:PropertiesOfTheScheme}
As mentioned in Introduction, the proposed numerical scheme is expected to have the area-preserving and
the curve-shortening properties since the original continuous problem has such properties.
Here, we assume that $\Gamma_t$ is a simply closed curve that separates $\mathbb{R}^2$ into a
bounded domain $\Omega_t$ and an unbounded domain $\mathbb{R}^2\backslash\overline{\Omega_t}$.
Suppose that $(u,\Gamma_t)$ is a solution to the Mullins-Sekerka problem $\eqref{eq:Mullins_Sekerka_intro}$.
Let us distinguish $u$ inside $\Omega_t$ by writing $u^i$, whereas $u$ outside $\Omega_t$ is presented as $u^e$.
Then, the area $\mathfrak{A}_t$ of $\Omega_t$ never changes. Indeed, we compute
\begin{eqnarray*}
    \frac{d}{dt}\mathfrak{A}_t &=& \int_{\Gamma_t}VdS = -\int_{\Gamma_t}\left[\frac{\partial u}{\partial \mathbf{n}}\right]dS = \int_{\Gamma_t}\left(\frac{\partial u^e}{\partial \mathbf{n}} - \frac{\partial u^i}{\partial \mathbf{n}}\right)dS \\
    &=& -\int_{B(0,R)\backslash\overline{\Omega_t}}\Delta u^edx+ \int_{\partial B(0,R)}\frac{\partial u^e}{\partial\mathbf{n}}dS - \int_{\Omega_t}\Delta u^i dx \\
    &=& 2\pi R\cdot O\left(\frac{1}{R^2}\right)\ \ \mbox{as}\ \ R\rightarrow\infty.
\end{eqnarray*}

Here, $R>0$ is arbitrarily taken so large that $\Omega_t\subset\subset B(0,R)$. Thus, $\frac{d}{dt}\mathfrak{A}_t = 0$.
Moreover, we see that the length $\mathfrak{L}_t$ of $\Gamma_t$ does not increase in time. Indeed,
\begin{eqnarray*}
    \frac{d}{dt}\mathfrak{L}_t &=& \int_{\Gamma_t}\kappa VdS = -\int_{\Gamma_t}u\left[\frac{\partial u}{\partial \mathbf{n}}\right]dS = \int_{\Gamma_t}\left(u\frac{\partial u^e}{\partial \mathbf{n}} - u\frac{\partial u^i}{\partial \mathbf{n}}\right)dS \\
    &=& -\int_{B(0,R)}|\nabla u|^2dx + \int_{\partial B(0,R)}u\frac{\partial u^e}{\partial \mathbf{n}} dS -\int_{B(0,R)\backslash\overline{\Omega}_t}u^e\Delta u^edx -\int_{\Omega_t}u^i\Delta u^idx.
\end{eqnarray*}

Since $\nabla u = O\left(\frac{1}{|x|^2}\right)$ as $|x|\rightarrow\infty$, $u$ is bounded in $\mathbb{R}^2$ (see Lemma \ref{lem:boundedness_nabla_1/x**2}).
Therefore, by letting $R\rightarrow\infty$, we see that the right-hand side of the above identity is not positive. In this section,
we will observe that these fine properties are valid even in the discrete version of the solutions derived by MFS.
Let us prepare several useful formulae to proceed argument.
All proofs of the following results are postponed and stated in Section \ref{sec:SequenceOfTheProofs}.\newline

Let $L = L(t)$ and $A = A(t)$ be the length and the area of a polygon evolving in time, subject to $\eqref{eq:DiscreteEvolutionLaw}$, respectively.
Namely,
\begin{equation}\label{eq:ContinuousLength}
  L := \sum_{i = 1}^N |\mathbf{X}_i - \mathbf{X}_{i-1}|,
\end{equation}
\begin{equation}\label{eq:ContinuousArea}
  A := \frac{1}{2}\sum_{i=1}^N \mathbf{X}_{i-1}^\perp\cdot\mathbf{X}_i.
\end{equation}
The following formulae are used to derive the tangential velocities of each vertex of the polygon.

\begin{prop}\label{prop:1}
    Let $L$ and $A$ be defined by $\eqref{eq:ContinuousLength}$ and $\eqref{eq:ContinuousArea}$, respectively.
     Then, the following formulae are valid:
    \begin{eqnarray}
        \dot{L} &=& \sum_{i=1}^N\kappa_i(v^+_i + v^-_i)r_i,\label{eq:prop1-1} \\
        \dot{A} &=& \sum_{i=1}^N(v^+_i + v^-_i)r_i + \operatorname{err}A\label{eq:prop1-2}
    \end{eqnarray}
    where
    \begin{equation*}
        \operatorname{err}A := \sum_{i=1}^N \left(W_i\sin_i - \frac{v^+_{i+1}-v^+_i}{2} - \frac{v^-_{i+1}-v^-_i}{2}\right)\frac{r_{i+1}-r_i}{2}.
    \end{equation*}
\end{prop}

\begin{prop}[Uniform boundedness of charges]\label{thm:uniform_boundedness_charges}
  Set
  \begin{equation*}
    \mathbb{A}_N := \begin{pmatrix}
      0 & \mathbf{H}^T \\
      \mathbf{1} & \mathbb{G}
    \end{pmatrix}, \mathbf{Q}_N := \begin{pmatrix}
      Q_0 \\
      Q_1 \\
      \vdots \\
      Q_N
    \end{pmatrix},
      \mathbf{K}_N := \begin{pmatrix}
      0 \\
      \kappa_1 \\
      \vdots \\
      \kappa_N
    \end{pmatrix}.
  \end{equation*}
  where $\mathbf{H} := (H_1,\cdots,H_N)^T$ and $\mathbb{G}:=(G_{i,j})$.
  Then, the values $\sup_{N\in\mathbb{N}}{\frac{1}{N}\|\mathbb{A}_N^{-1}\|_1}$, $\sup_{N\in\mathbb{N}}{\frac{1}{N}\|\mathbf{Q}_N\|_1}$
  ,$\sup_{N\in\mathbb{N}}{\frac{1}{N}\|\mathbf{K}_N\|_1}$ are finite whenever $\mathbb{A}_N$ is regular for sufficiently large $N$.
  Especially, each charge $Q_i(0\leq i\leq N)$ has order $O(1)$.
\end{prop}

\begin{thm}[Curve shortening property]\label{thm:CS}
  Assume that $L$ is defined by $\eqref{eq:ContinuousLength}$.
  Then, it holds that
    \begin{equation}\label{eq:thm_CS}
        \dot{L} \leq -\int_{\Omega^N_t}|\nabla U^+|^2 - \int_{B_R\backslash\overline{\Omega^N_t}}|\nabla U^-|^2 + O(N^2\log{N}) + C_N\cdot O\left(\frac{1}{R}\right)\ \ \mbox{as}\ \ N,R\rightarrow\infty
    \end{equation}
    where $C_N$ denotes a positive constant depending on $N$.
\end{thm}

\begin{rem}
  Despite the appearance of the term $O(N^2\log{N})$, the right-hand side of $\eqref{eq:thm_CS}$ becomes nonpositive for $N$ large enough.
  Indeed, since $d\leq|x - y_j|\leq d + \mbox{diam}(\Omega^N)$ for each $1\leq j\leq N$ and $x\in\overline{\Omega^N}$, it holds that $|x - y_j| = O(\frac{1}{\sqrt{N}})$.
  An argument similar to Lemma \ref{lem:estimate_H_G} (see Section \ref{sec:SequenceOfTheProofs}) shows 
  \begin{equation*}
    \left(\frac{x - y_j}{|x - y_j|^2} - \frac{x - z_j}{|x - z_j|^2}\right) = O(\sqrt{N}).
  \end{equation*}
  Hence, a direct calculation gives:
  \begin{eqnarray*}
    |\nabla U^+(x)|^2 &=& \frac{1}{4\pi^2}\sum_{j = 1}^NQ_i^+\left(\frac{x - y_j}{|x - y_j|^2} - \frac{x - z_j}{|x - z_j|^2}\right)\cdot\sum_{j = 1}^NQ_i^+\left(\frac{x - y_j}{|x - y_j|^2} - \frac{x - z_j}{|x - z_j|^2}\right) \\ 
    &=& \sum_{i,j=1}^N O(1)\times O(1)\times O(\sqrt{N})\times O(\sqrt{N}) = O(N^3).
  \end{eqnarray*}
  Here, we have recalled that $Q_j^+ = O(1)$ for each $1\leq j\leq N$ from Proposition \ref{thm:uniform_boundedness_charges}.
  To get a desired property, we fix $N$ so large that $-\int_{\Omega^N}|\nabla U^+|^2 + O(N^2\log{N})$ is negative, and then, 
  send $R$ to infinity. Note that the term $-\int_{B_R\backslash\overline{\Omega^N}}|\nabla{U}^-|^2$ is decreasing with respect to $R$.
  Thanks to Proposition \ref{prop:1}, we can expect that our scheme decreases the length of the polygon step-by-step.
  This expectation turns out to be true subsequently (see Corollary \ref{cor:time_step_optimization}).
\end{rem}

\begin{rem}
  If the domain that we approximate by a polygon is a circle, then the term $O(N^2\log{N})$ appearing in Theorem \ref{thm:CS} is expected to be replaced by $O(\frac{\alpha^N}{N})$ for some $0<\alpha<1$.
  Let us write the approximate solutions of $\eqref{eq:def_of_U}$ and $\eqref{eq:def_of_U_bar}$ as $U^+_{(N)}$ and $U^-_{(N)}$, respectively.
  Moreover, we replace the domain in which we solve the external problem with an annulus.
  Then, from the results of Katsurada (Theorem 2.2 and Theorem 4.1 \cite{KatsuradaOkamotoII}) and Murota (Theorem 2.4 \cite{MurotaKazuo}), we can expect 
  that $U^+_{(N)}$ and $U^-_{(N)}$ converge exponentially to the exact solutions $U^+$ and $U^-$ as $N\rightarrow\infty$.
  We cannot apply these results directly because their selection of collocation points is slightly different from ours.
  Thus, let us assume that the results are available for our study. Since $\Gamma^N$ is a regular polygon,
  we see that $\kappa_i = \kappa$ for every $1\leq i\leq N$. Then, the AP property, which is the second constraints of $\eqref{eq:internal_problem}$ and $\eqref{eq:external_problem}$, guarantees that
  \begin{equation*}
    \sum_{i=1}^N\int_{\Gamma_i^N}S_i(\mathbf{X}_i^*)dS = 0.
  \end{equation*}
  Since $U^+_{(N)}$ and $U^-_{(N)}$ are radially symmetric due to their construction, 
  $\nabla U^+_{(N)}\cdot\mathbf{n}_i - \nabla U^-_{(N)}\cdot\mathbf{n}_i$ has the same sign for all $1\leq i\leq N$.
  Hence, we can assume that this value is nonnegative without loss of generality.
  Then, we can estimate as follows:
  \begin{eqnarray*}
    \left|\sum_{i=1}^N\int_{\Gamma_i^N}S_i(x)dS\right| &=& \left|\sum_{i=1}^N\int_{\Gamma_i}\left(U^+_{(N)}\nabla U^+_{(N)}\cdot\mathbf{n}_i - U^-_{(N)}\nabla U^-_{(N)}\cdot\mathbf{n}_i\right)dS\right| \\
    &\leq& \sum_{i=1}^N \sup_{\Gamma_i^N}|U^+_{(N)} - U^-_{(N)}|\int_{\Gamma_i}\left(\nabla U^+_{(N)}\cdot\mathbf{n}_i - \nabla U^-_{(N)}\cdot\mathbf{n}_i\right)dS \\
    &\leq& \sup_{\Omega}|U^+_{(N)} - U^-_{(N)}|\sum_{i=1}^N \int_{\Gamma_i}\left(\nabla U^+_{(N)}\cdot\mathbf{n}_i - \nabla U^-_{(N)}\cdot\mathbf{n}_i\right)dS. 
  \end{eqnarray*}
  Here, we note that $\Gamma_i^N\subset\overline{\Omega}$ for each $1\leq i\leq N$ to derive the last inequality.
  Moreover, there are neither singular points nor dummy singular points in the domain $\Omega_i^N$ sandwiched between $\Gamma_i^N$ and $\partial\Omega$ for every $1\leq i\leq N$
  for $N\in\mathbb{N}$ large enough. To see this, recall that $|\mathbf{X}_i^* - y_i^{\pm}| = |\mathbf{X}_i^* - z_i^{\pm}| = O(\frac{1}{\sqrt{N}})$
  and the width of $\Omega_i^N$ in the direction $\mathbf{n}_i$ equals $r\left(1-\cos{\frac{\pi}{2N}}\right) = O\left(\frac{1}{N^2}\right)$.
  This obviously implies that $y^{\pm}_i,z^{\pm}_i\notin\Omega_i^N$ for sufficiently large $N\in\mathbb{N}$.
  Hence, we obtain
  \begin{equation*}
    0 = \int_{\Omega_i^N} \Delta U^{\pm}_{(N)}dx = \int_{\partial\Omega_i^N\cap\partial\Omega}\nabla U^{\pm}_{(N)}\cdot\mathbf{n}dS - \int_{\Gamma_i^N}\nabla U^{\pm}_{(N)}\cdot\mathbf{n}_idS.
  \end{equation*}
  From this observation, we can calculate as follows:
  \begin{eqnarray*}
    \left|\sum_{i=1}^N\int_{\Gamma_i^N}S_i(x)dS\right| &\leq & \sup_{\Omega}|U^+_{(N)} - U^-_{(N)}| \sum_{i=1}^N\int_{\partial\Omega_i^N\cap\partial\Omega}\left(\nabla U^+_{(N)}\cdot\mathbf{n}_i - \nabla U^-_{(N)}\cdot \mathbf{n}_i\right)dS \\
    &\leq & C\tau^N \int_{\partial\Omega} \left(\nabla U^+_{(N)}\cdot\mathbf{n}_i - \nabla U^-_{(N)}\cdot\mathbf{n}_i\right)dS = O(1)\tau^N\ \ \mbox{as}\ \ N\rightarrow\infty.
  \end{eqnarray*}
  Admitting Katsurada's result(See Remark4.1 \cite{KatsuradaOkamotoII}), the integration of the right hand side of the above inequality tends to
  $\int_{\partial\Omega}\left[ \frac{\partial U}{\partial\mathbf{n}} \right] dS$ as $N\rightarrow\infty$ and this quantity equals zero
  if $U^+$ and $U^-$ are exact solutions to $\eqref{eq:Mullins_Sekerka_intro}$.
  In this way, we can predict that it is not necessary to take $N$ so large to ensure the CS property holds.
\end{rem}

\begin{thm}[Area preserving property]\label{thm:AP}
    Assume that the vertices $\mathbf{X}_i$ are uniformly distributed, namely, $r_i = r_{i+1}$ for each $1\leq i\leq N$ and
    these are differentiable with respect to the time variable.
    Then, the area $A$ surrounded by $\Gamma$ is theoretically preserved for all $t$; that is $\dot{A} = 0$,
    provided that we adopt the UDM.
\end{thm}

Theorem \ref{thm:CS} indicates that our scheme has the CS property in some sense once we assume the differentiability of $\mathbf{X}_i$ with respect to the time variable.
Our next trial is to derive a time-discrete version of Theorem \ref{thm:CS}.
For simplicity of notation, let us use the symbol $v_i^n$ instead of $v^+_i + v^-_i$.
Moreover, let $r^n_i$ and $\kappa^n_i$ denote the corresponding values of $\Gamma^{N,n}_{\Delta t}$ defined in $\eqref{eq:rtn}$ and $\eqref{eq:DiscreteCurvature}$, respectively.

\begin{thm}[Discrete version of the CS property]\label{thm:discrete_L_formula}
  For each $n\geq 0$, let $L^n := \sum_{i = 1}^N|\mathbf{X}^n_i - \mathbf{X}^n_{i-1}|$ where 
  $\mathbf{X}^n_i(1\leq i\leq N)$ is defined in terms of $\eqref{eq:FullyDiscreteScheme}$.
  Then, it holds that for each $n\geq 0$,
  \begin{equation}\label{eq:discrete_L_formula}
    L^{n+1}-L^n = \Delta t\sum_{j=1}^N\kappa^n_jv^n_jr^n_j + \Delta t^2\left(O\left(\frac{1}{N}\right)\left(\sum_{j=1}^N\kappa^n_jv^n_jr^n_j\right)^2 + O(N^{\frac{3}{2}})\sum_{j=1}^N\kappa^n_jv^n_jr^n_j + O(N^4)\right).
  \end{equation}
\end{thm}
    We deduce from Theorem \ref{thm:CS} and Theorem \ref{thm:discrete_L_formula} that the highest degree with respect to $N$ of negative terms on the right-hand side of $\eqref{eq:discrete_L_formula}$
    equals $\max{\{3-\alpha, \frac{9}{2} - 2\alpha\}}$. On the other hand, one of the positive terms equals $5-2\alpha$.
    Thus, by solving the inequality $\max{\{3-\alpha, \frac{9}{2} - 2\alpha\}} > 5-2\alpha$, we see that the right-hand side of $\eqref{eq:discrete_L_formula}$ should be negative
    for $N\in\mathbb{N}$ large enough whenever $\alpha > 2$.
    Consequently, we obtain the following corollary.
\begin{cor}\label{cor:time_step_optimization}
  For any $\alpha > 2$, suppose that $\Delta t := N^{-\alpha}$. Then, for sufficiently large $N$, $L^{n+1}\leq L^n$ is valid for each $n\geq 0$.
\end{cor}
\begin{rem}
  In \cite{Sakakibara_Yazaki}, several numerical experiments were carried out with $N = 100$ and $\Delta t := 0.1 N^{-2}$.
  This choice seems reasonable since $\Delta t = N^{-\frac{5}{2}}$, and this setting is consistent with the result of Corollary \ref{cor:time_step_optimization}.
\end{rem}
\section{Numerical examples}\label{sec:NumericalExamples}
\subsection{Pi curve}\label{sec:pi_curve}
We borrowed the coordinates of the pi curve as the initial data from https://ja.wolframalpha.com/, as \cite{Sakakibara_Yazaki}. See Figure \ref{fig:pi} for a numerical result.
\begin{figure}[H]
    \begin{minipage}[b]{0.5\linewidth}
      \centering
      \includegraphics[keepaspectratio, scale=0.25]{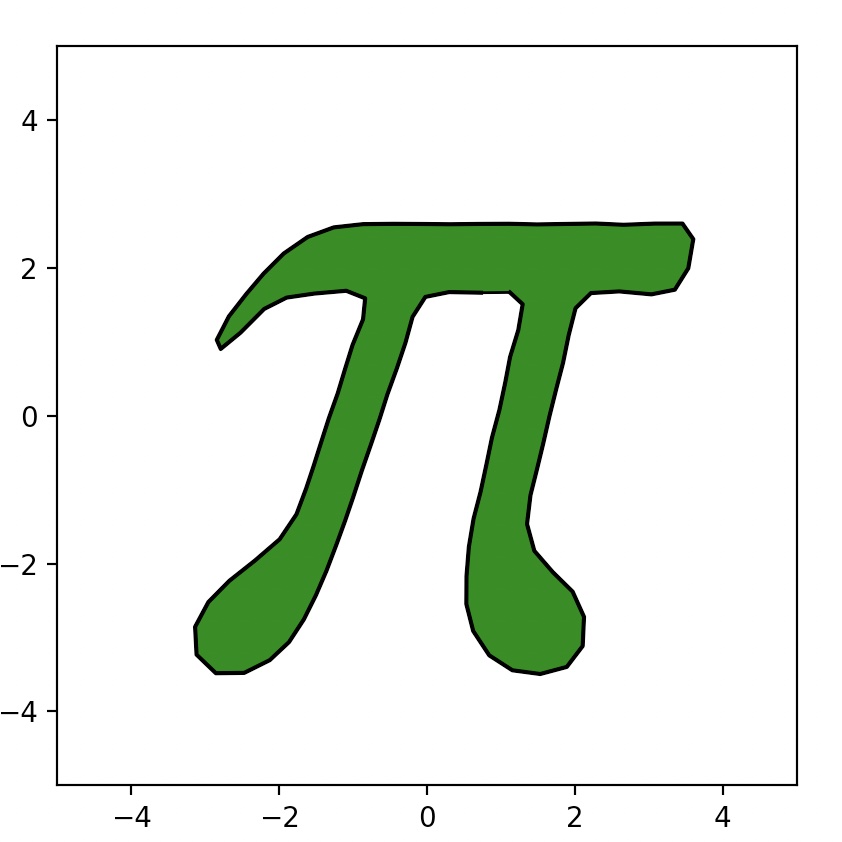}
      \subcaption{t = 0}\label{fig:pi_0_0}
    \end{minipage}
    \begin{minipage}[b]{0.5\linewidth}
      \centering
      \includegraphics[keepaspectratio, scale=0.25]{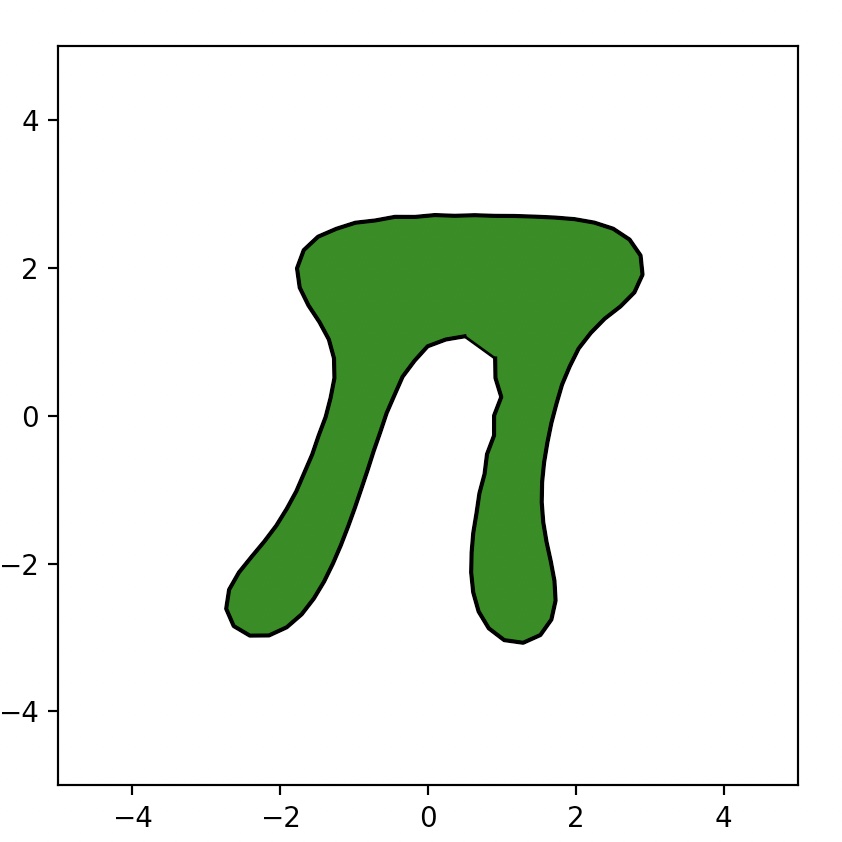}
      \subcaption{t=0.02}\label{fig:pi_0_02}
    \end{minipage} \\
    \begin{minipage}[b]{0.5\linewidth}
      \centering
      \includegraphics[keepaspectratio, scale=0.25]{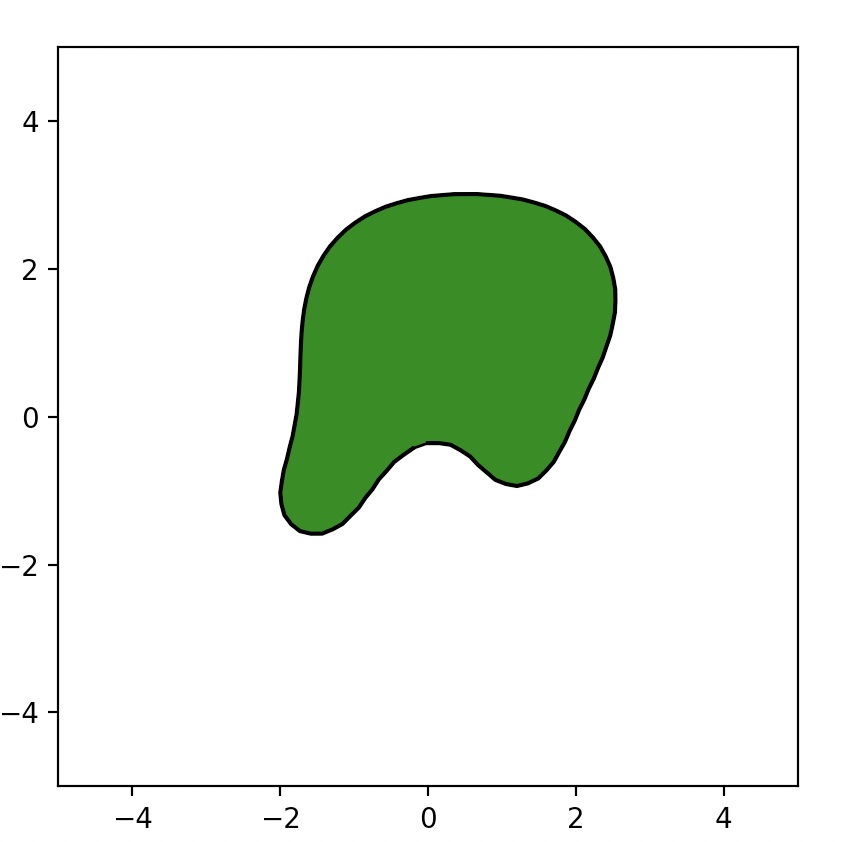}
      \subcaption{t = 0.08}\label{fig:pi_0_08}
    \end{minipage}
    \begin{minipage}[b]{0.5\linewidth}
      \centering
      \includegraphics[keepaspectratio, scale=0.25]{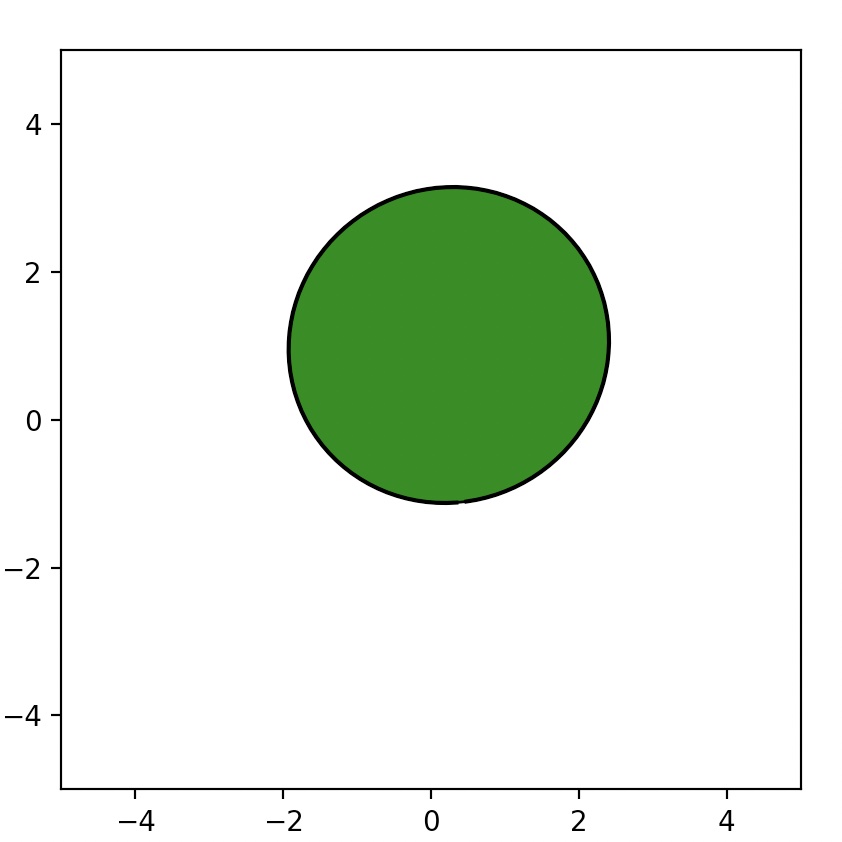}
      \subcaption{\Erase{t = 0.4}\Add{t = 0.7}}\label{fig:pi_0_40}
    \end{minipage}
    \caption{Evolution of Pi shaped curve. $N=100, \Delta t=0.1N^{-2}$.}\label{fig:pi}
  \end{figure}

  \begin{figure}[H]
      \centering
      \includegraphics[keepaspectratio, scale=0.5]{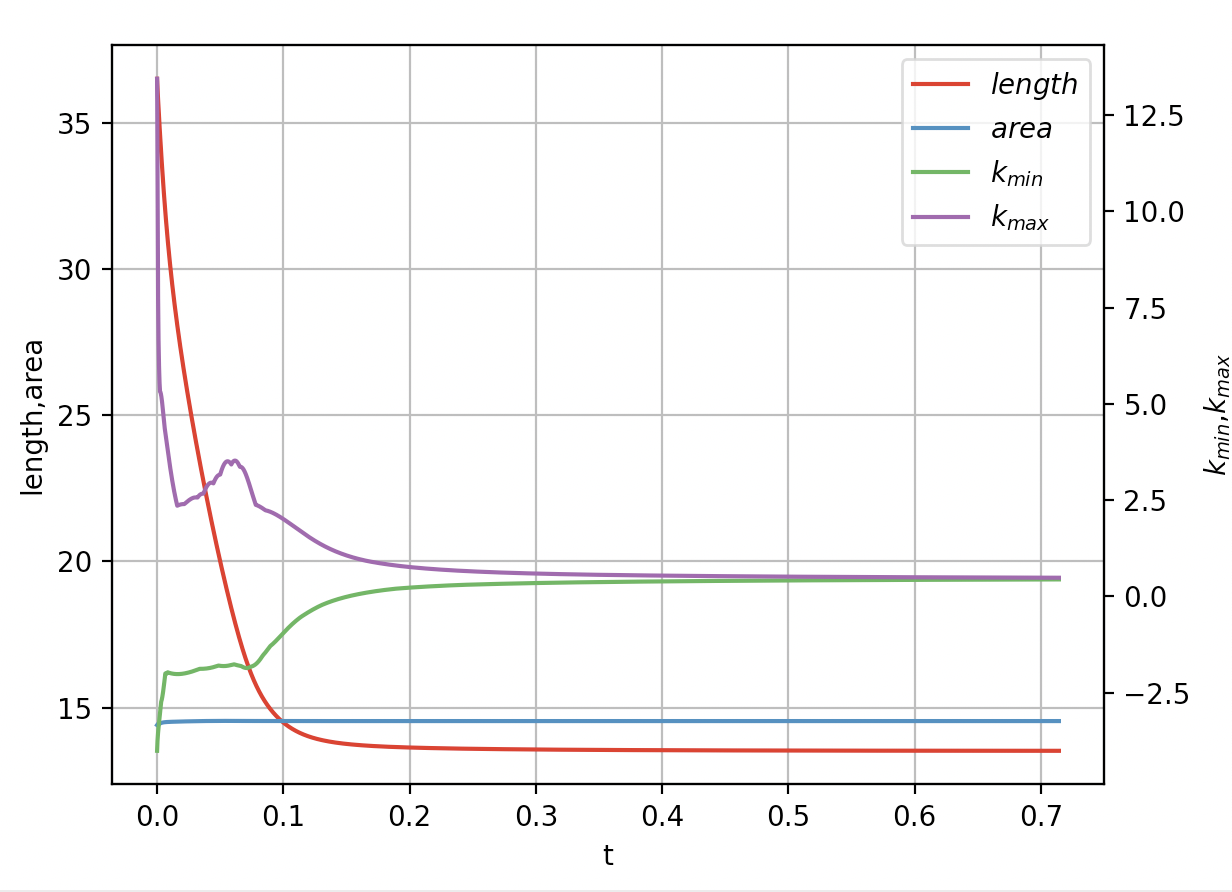}
      \caption{Time evolution of length, area and curvatures for pi-curve.}\label{fig:pi_length_area_k_graph}
  \end{figure}

  To confirm the uniform boundedness of $Q^+_0$, $\frac{1}{N}\sum_{j=1}^N|Q^+_j|$, ${Q}^-_0$, and $\frac{1}{N}\sum_{j=1}^N|{Q}^-_j|$ for a large $N$, 
  we extract the first step of the scheme for each $N$ and draw their values as a graph. We infer from Figure \ref{fig:pi_N_Q_graph} that
  all of them may be dominated by constants and depend only on the initial curve.
  \begin{figure}[H]
      \centering
      \includegraphics[keepaspectratio, scale=0.5]{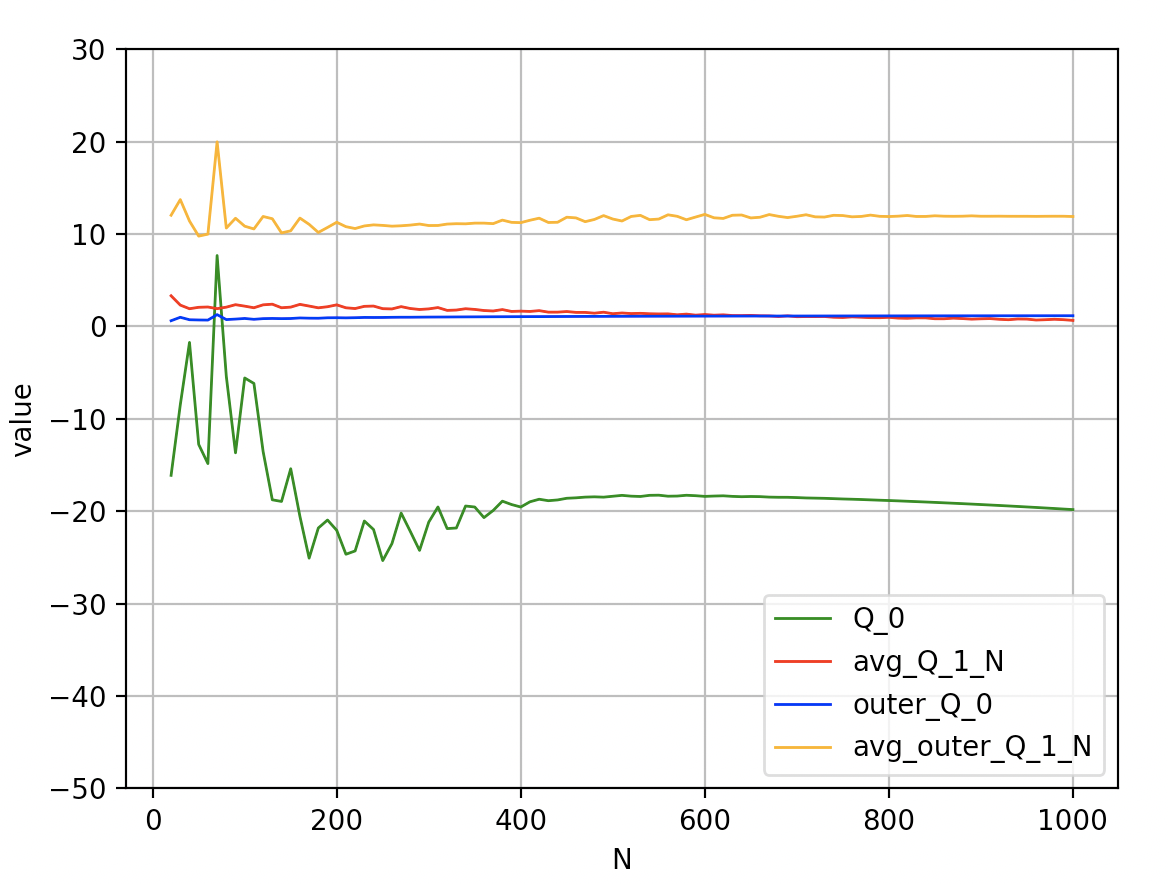}
      \caption{N dependence of coefficients.}\label{fig:pi_N_Q_graph}
  \end{figure}

\subsection{Star}

The star-shaped curve is close to a circle. As expected, this shape is deformed into a circle by our scheme.
See Figure \ref{fig:star} for a numerical result.
\begin{figure}[H]
    \begin{minipage}[b]{0.5\linewidth}
      \centering
      \includegraphics[keepaspectratio, scale=0.25]{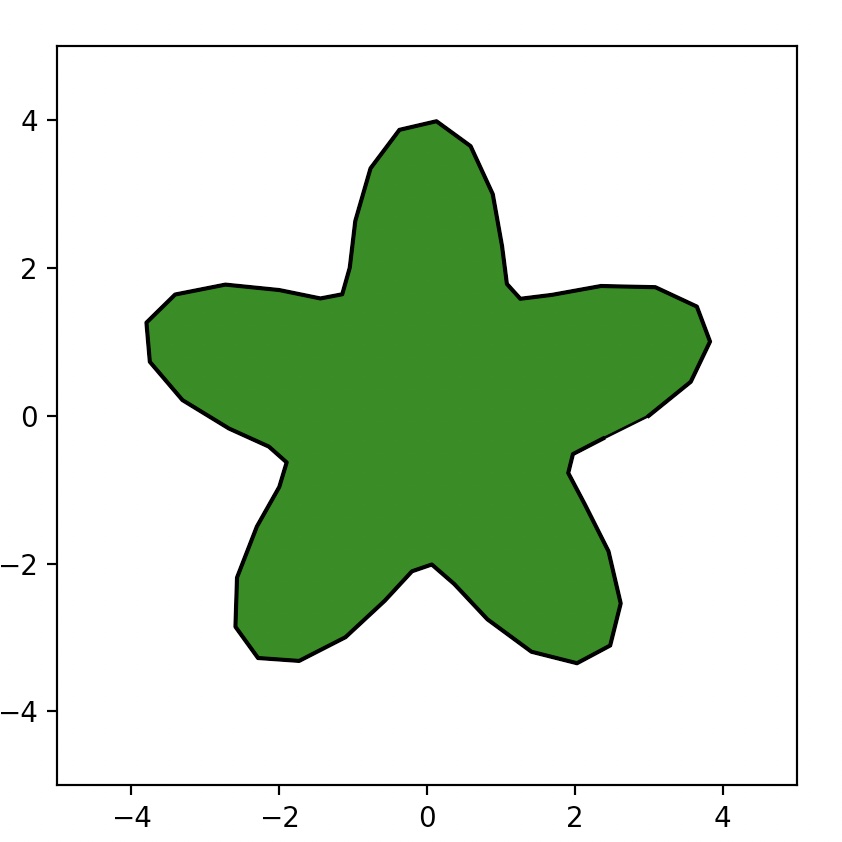}
      \subcaption{t = 0}\label{fig:star_0_0}
    \end{minipage}
    \begin{minipage}[b]{0.5\linewidth}
      \centering
      \includegraphics[keepaspectratio, scale=0.25]{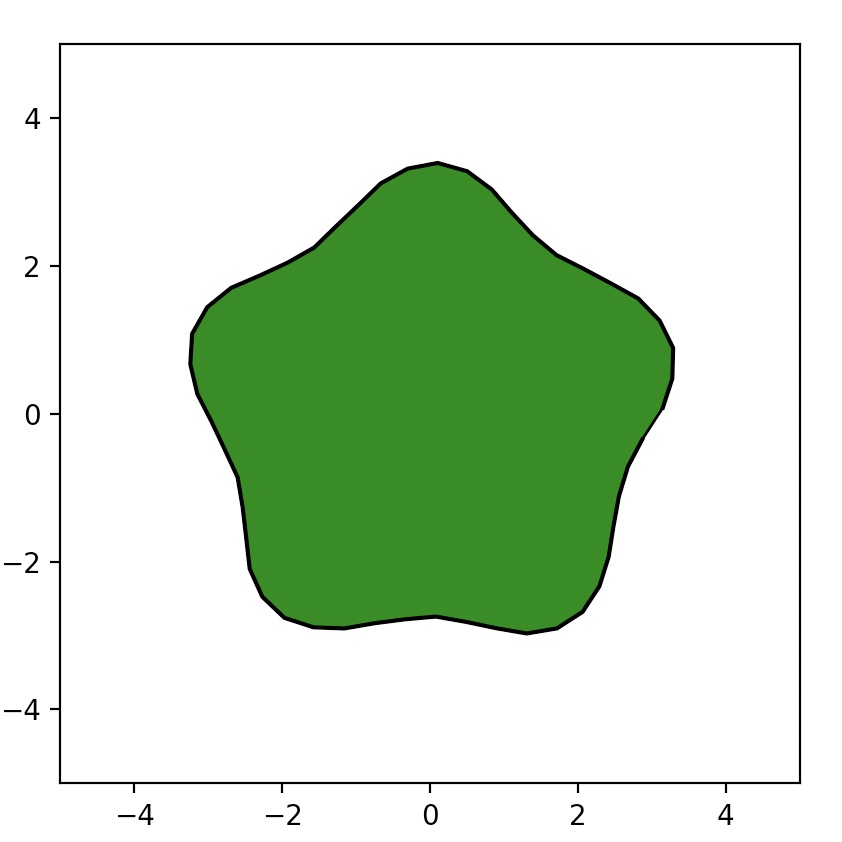}
      \subcaption{\Erase{t=0.01}\Add{t=0.05}}\label{fig:star_1}
    \end{minipage} \\
    \begin{minipage}[b]{0.5\linewidth}
      \centering
      \includegraphics[keepaspectratio, scale=0.25]{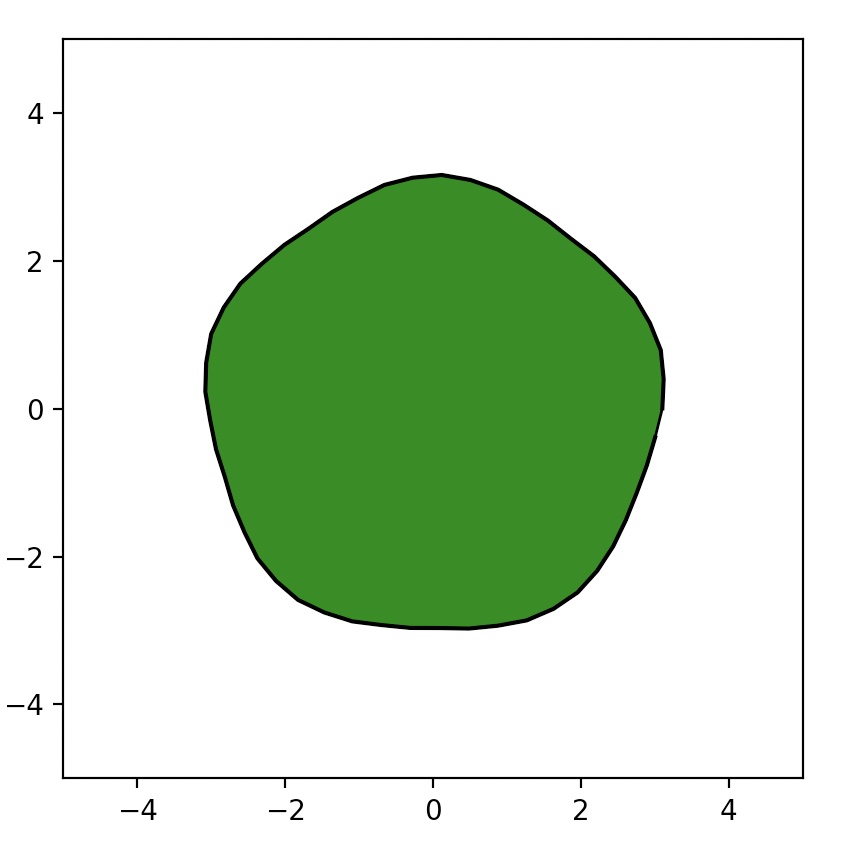}
      \subcaption{\Erase{t=0.03}\Add{t=0.1}}\label{fig:star_3}
    \end{minipage}
    \begin{minipage}[b]{0.5\linewidth}
      \centering
      \includegraphics[keepaspectratio, scale=0.25]{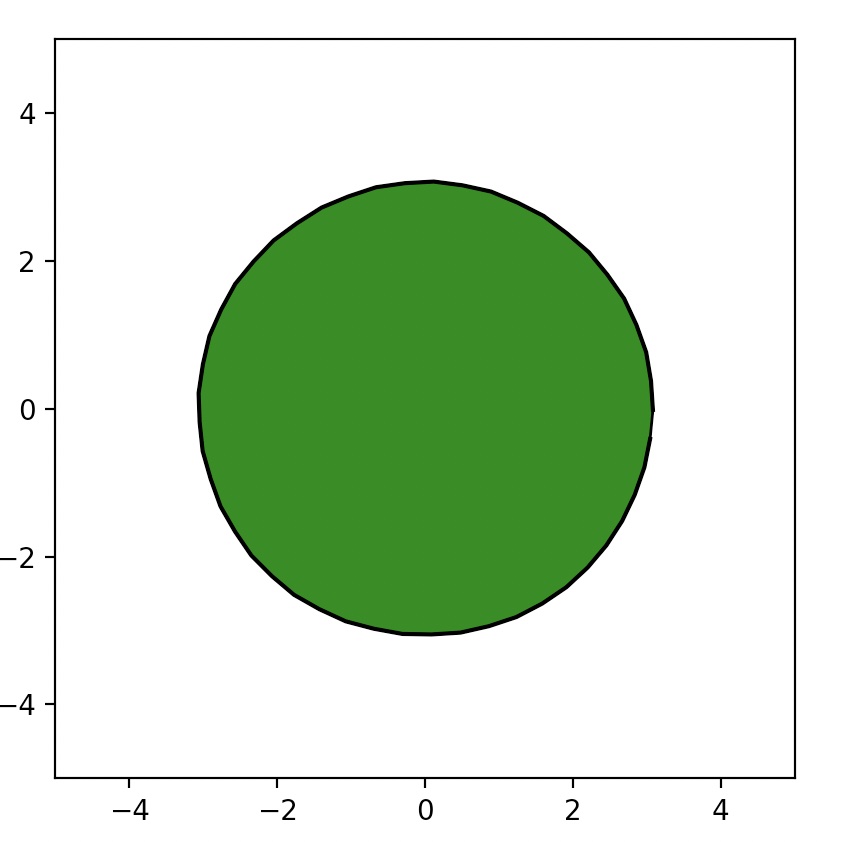}
      \subcaption{\Erase{t=0.05}\Add{t=0.2}}\label{fig:star_5}
    \end{minipage}
    \caption{Evolution of Star shaped curve. $N=50, \Delta t=0.1N^{-2}$.}\label{fig:star}
  \end{figure}
\subsection{Tube}
Mayer \cite{Soni1997TwosidedMF} showed that the one-sided Mullins-Sekerka flow (Hele-Shaw flow) does not necessarily preserve the convexity of the initial curve.
This is a unique feature of the Mullins-Sekerka flow. This fact was also numerically observed by Bates et al. \cite{Bates1995ANS} provided that 
an initial curve is a tube with two circular end caps, even in the 2-phase Mullins-Sekerka flow. We shall confirm that our scheme yields the same result as well.
We set $N = 50, \Delta t = 0.1N^{-2}$. In addition, assume that the length and thickness of the tube are $8.0$ and $1.0$, respectively.
\begin{figure}[H]
    \begin{minipage}[b]{0.5\linewidth}
      \centering
      \includegraphics[keepaspectratio, scale=0.2]{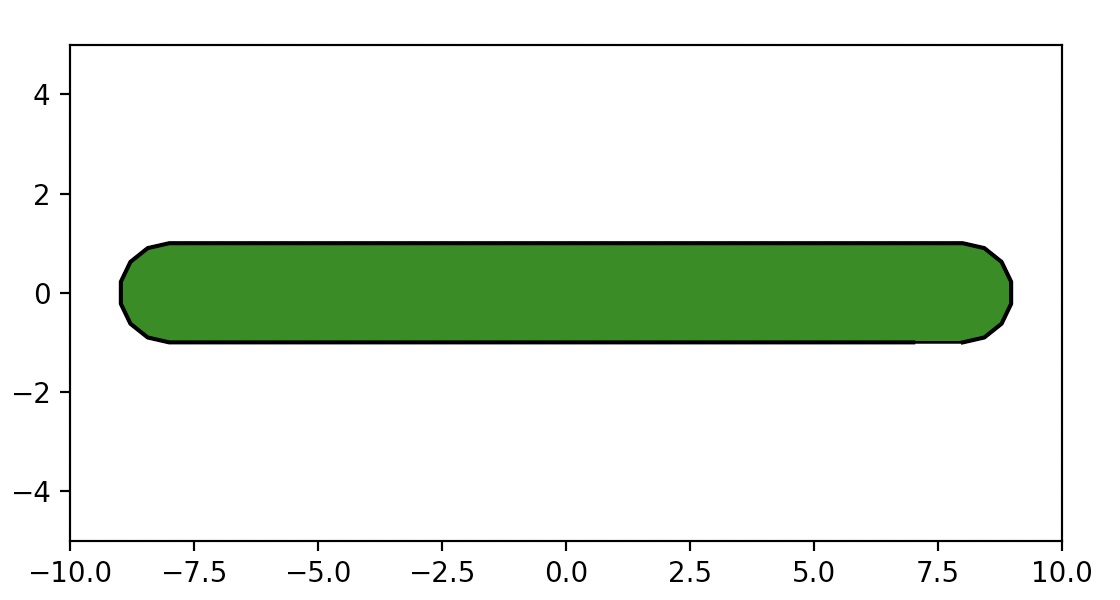}
      \subcaption{t = 0}\label{fig:tube_0}
    \end{minipage}
    \begin{minipage}[b]{0.5\linewidth}
      \centering
      \includegraphics[keepaspectratio, scale=0.2]{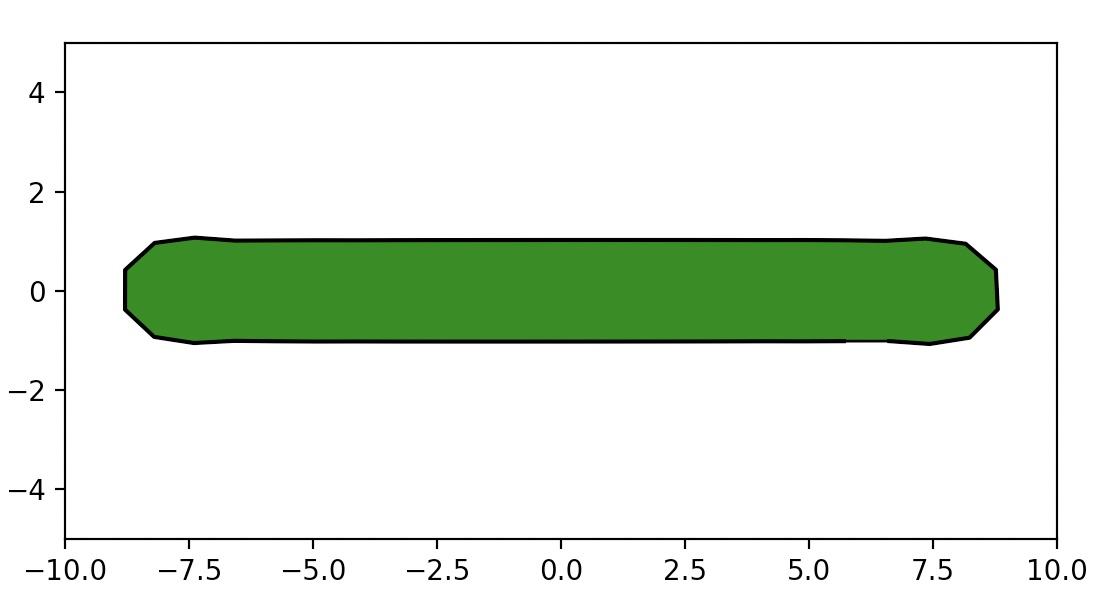}
      \subcaption{t = 0.01}\label{fig:tube_1000}
    \end{minipage} \\
    \begin{minipage}[b]{0.5\linewidth}
      \centering
      \includegraphics[keepaspectratio, scale=0.2]{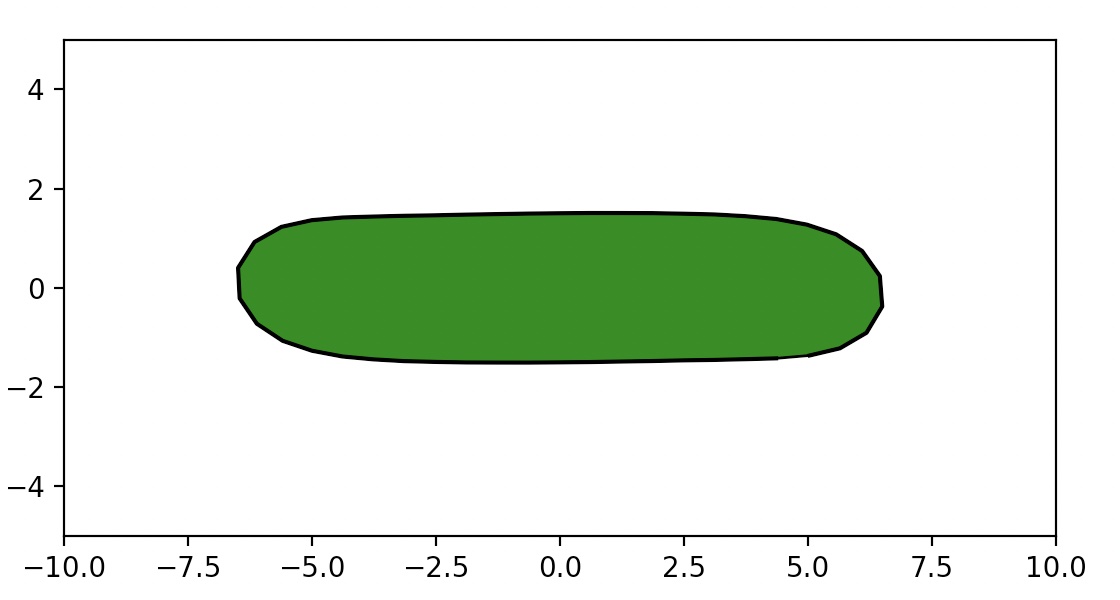}
      \subcaption{\Erase{t=0.025}\Add{t = 0.2}}\label{fig:tube_2500}
    \end{minipage}
    \begin{minipage}[b]{0.5\linewidth}
      \centering
      \includegraphics[keepaspectratio, scale=0.2]{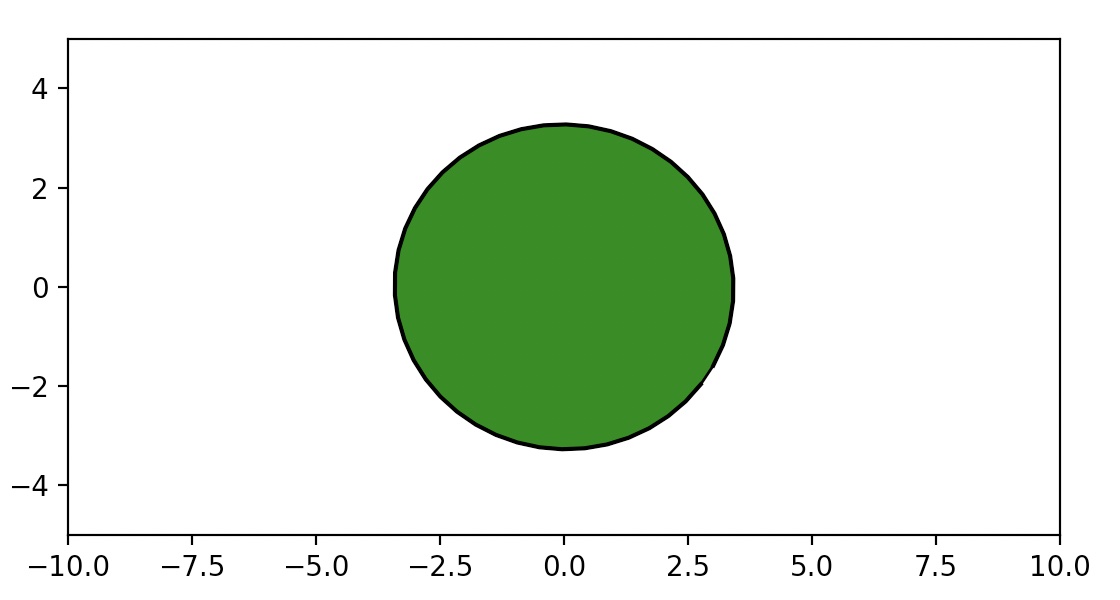}
      \subcaption{\Erase{t=0.07}\Add{t = 1.6}}\label{fig:tube_7000}
    \end{minipage}
    \caption{Evolution of Tube shaped curve.}\label{fig:tube}
  \end{figure}
\subsection{Accuracy of the scheme}
As mentioned in \cite{Bates1995ANS}, two concentric circles and a function $u$ in $\mathbb{R}^2$ are available to comfirm
the accuracy of the numerical scheme. Let us recall the definition of $u$ as follows:
For $0 < R_1 < R_2$ and $x\in\mathbb{R}^2$,
\begin{equation}\label{eq:2ConcentricCircles}
  u(x) := \begin{cases}
    \frac{1}{R_2}\ \ \mbox{if}\ \ |x| \geq R_2,\\
    -\frac{1}{R_1} + \frac{\frac{1}{R_1} + \frac{1}{R_2}}{\log{\frac{R_2}{R_1}}}\log{\frac{|x|}{R_1}}\ \ \mbox{if}\ \ R_1\leq |x|\leq R_2, \\
    -\frac{1}{R_1}\ \ \mbox{if}\ \ |x| \leq R_1.
  \end{cases}
\end{equation}
Then, $u$ is readily harmonic in $\Omega := \{R_1<|x|<R_2\}$ and satisfies the Gibbs-Thomson law on
$\partial B_{R_1}(0)\cup \partial B_{R_2}(0)$. However, we cannot adopt this function as a benchmark for our scheme
because this model can only describe the Hele-Shaw flow, which is one-sided.
Indeed, though the normal derivative of $u$ jumps across $\partial\Omega$,
$u$ is constant in $\mathbb{R}^2\backslash\overline{\Omega}$ so that the jump coincides with $-\nabla u\cdot\nu$ and this case is something special.
Hence, we should prepare a new example of $u$ whose normal derivative does not vanish in both side across the boundary.
To this end, we alternatively consider the following function $u$ and the domains $\Omega_i$ and $\Omega_e$.
Let $0<R_1<R_2<R_3<R$ and we define
\begin{equation}\label{eq:3ConcentricCircles}
  u(x) := \begin{cases}
    \frac{1}{R_3}\ \ \mbox{if}\ \ |x|\geq R_3,\\
    -\frac{1}{R_2} + \frac{\frac{1}{R_2} + \frac{1}{R_3}}{\log{\frac{R_3}{R_2}}}\log{\frac{|x|}{R_2}}\ \ \mbox{if}\ \ R_2\leq |x|\leq R_3,\\
    \frac{1}{R_1} + \frac{-\frac{1}{R_1} - \frac{1}{R_2}}{\log{\frac{R_2}{R_1}}}\log{\frac{|x|}{R_1}}\ \ \mbox{if}\ \ R_1\leq |x|\leq R_2, \\
    \frac{1}{R_1}\ \ \mbox{if}\ \ |x| \leq R_1,
  \end{cases}
\end{equation}
\begin{eqnarray*}
  \Omega_i &:=& \{|x| < R_1\} \cup \{R_2 < |x| < R_3\},\ \ \Omega_e := \mathbb{R}^2\backslash\overline{\Omega_i}.
\end{eqnarray*}
Though the computation will be much more complicated since $\Omega_i$ is not connected similarly to $\Omega_e$, we see that
the normal derivative of $u$ from both sides of across $\partial B_{R_2}(0)\subset\partial\Omega_i$ does not vanish.
Observe that the function $u$ defined by $\eqref{eq:3ConcentricCircles}$ rigorously satisfies 
the equalities $\eqref{eq:StringForm}$ and the radii $R_1, R_2$ and $R_3$ evolve, subject to 
the following system of ordinary differential equations:
\begin{equation}\label{eq:RadiusODESystem}
  \begin{cases}
    \frac{dR_1}{dt} = -\frac{\frac{1}{R_1} + \frac{1}{R_2}}{R_1\log{\frac{R_2}{R_1}}}, \\
    \frac{dR_2}{dt} = -\left(\frac{\frac{1}{R_1} + \frac{1}{R_2}}{R_2\log{\frac{R_2}{R_1}}} + \frac{\frac{1}{R_2} + \frac{1}{R_3}}{R_2\log{\frac{R3}{R_2}}}\right), \\
    \frac{dR_3}{dt} = -\frac{\frac{1}{R_2} + \frac{1}{R_3}}{R_3\log{\frac{R_3}{R_2}}}.
  \end{cases}
\end{equation}
In the sequel, let us confirm that this $u$ can be approximated by the proposed scheme.
First, to estimate the error between the results of the proposed scheme and the rigorous solution, 
we need to derive a numerical solution of $\eqref{eq:RadiusODESystem}$.
To this end, we adopt the fourth-order Runge-Kutta method.
We draw three concentric circles: $C_1$, $C_2$, and $C_3$ 
with radii of $R_1$, $R_2$, and $R_3$, respectively.
Then, we plot $N_1$, $N_2$, and $N_3$ collocation points $\mathbf{X}_{1,1},\cdots,\mathbf{X}_{k,N_k}$ on $C_k$ for $k = 1,2,3$, respectively.
Due to a technical reason, we assume that $N_2 = N_3$. For singular points $y^\pm_{k,j}$ and dummy singular points $z^\pm_{k,j}$, we set:
\begin{multline*}
    y^+_{1,j} := \cpoint{X}{*}{1}{j} + \frac{|\cpoint{X}{}{1}{1}|}{R_1\sqrt{N}}\cpoint{n}{}{1}{j}, \ \ z^+_{1,j} := 1000y_{1,j},
    y^+_{2,j} := \cpoint{X}{*}{2}{j} + \frac{|\cpoint{X}{}{2}{1}|}{R_2\sqrt{N}}\cpoint{n}{}{2}{j}, \ \ z^+_{2,j} := 1000y_{2,j},\\
    y^+_{3,j} := \cpoint{X}{*}{3}{j} + \frac{|\cpoint{X}{}{3}{1}|}{R_3\sqrt{N}}\cpoint{n}{}{3}{j}, \ \ z^+_{3,j} := 1000y_{3,j},
    y^-_{1,j} := \cpoint{X}{*}{1}{j} - \frac{|\cpoint{X}{}{1}{1}|}{R_1\sqrt{N}}\cpoint{n}{}{1}{j}, \ \ y^-_{2,j} := \cpoint{X}{*}{2}{j} - \frac{|\cpoint{X}{}{2}{1}|}{R_2\sqrt{N}}\cpoint{n}{}{2}{j},\\ 
    y^-_{3,j} := \cpoint{X}{*}{3}{j} - \frac{|\cpoint{X}{}{3}{1}|}{R_3\sqrt{N}}\cpoint{n}{}{3}{j}, \ \ z^-_{1,j} := \frac{1}{2}y^-_{1,j},\ \ z^-_{2,j} = z^-_{3,j} := \frac{1}{2}(y^-_{2,j} + y^-_{3,j}).
\end{multline*}

The quantity $d$ that indicates the distance of each charge point $y^\pm_{k,j}$ from the edge $[\cpoint{X}{}{k}{j-1},\cpoint{X}{}{k}{j}]$ is set to $\frac{1}{\sqrt{N}}$ universally.
However, we adjust this quantity by multiplying it with a scaling parameter. 
This scaling will stabilize the position of each charge point, even if the length of the edge $[\cpoint{X}{}{k}{j-1},\cpoint{X}{}{k}{j}]$
becomes quite small. Note that this modification is unnecessary if the domain approximated by a polygon is connected because 
the shape of the domain converges to a circle in time and the length of the edge is bounded below by a positive constant.
In this experiment, the radii are chosen as $R_1 := 1.0$, $R_2 := 5.0$, and $R_3 := 10.0$.
Solving $\eqref{eq:RadiusODESystem}$ numerically for these parameters shows that the circle $C_1$ vanishes at approximately $t = 0.56$, and a topological change occurs.
However, if $N$ is too small, then the error between the approximate solution and the rigorous one is not neglectable, and we cannot 
proceed the scheme until $t=0.56$. This difficulty may stem from observation that the approximate solution $U^+$ should be constant inside $C_1$ although this cannot be expected
due to the structure of the approximate solution. To avoid this problem, we calculate $R_1'(t)$ and $R_3'(t)$ in terms of the system $\eqref{eq:RadiusODESystem}$, whereas 
we approximate $R_2'(t)$ by the proposed scheme. To be more accurate, we apply the fourth-order Runge-Kutta method 
to obtain $R_1(t)$, $R_2(t)$, and $R_3(t)$ for each time-step. To reduce the calculation load, we skip the step of the UDM. 
In other words, we always assume that the tangential velocity equals zero. Being optimistic, it is worth 
to consider the normal velocity because the concentric circles are symmetric with respect to the origin.
In this way, we can proceed with the scheme and measure the error of $R_2$ until a topological change occurs.

\begin{figure}[H]
    \centering
    \includegraphics[keepaspectratio, scale=0.5]{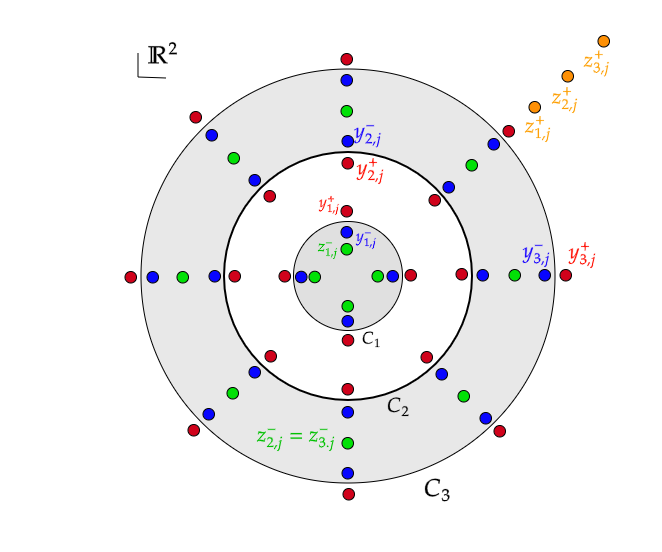}
    \caption{$3$ concentric circles and the position of charge points.}\label{fig:3ConcentricCircles}
\end{figure}
Now, let us explain a method to measure the efficiency of our scheme. 
We fix $N_1 := 10$, change $N := N_2 = N_3$ and compare the results of computation.
Let $\operatorname{Err}(N)$ denote the quantity used to measure the efficiency of our scheme. This quantity should equal zero
when the approximate solution corresponds entirely to a rigorous solution. 
Suppose that the approximate solution obtained by our scheme converges to a rigorous solution
in the sense that $\lim_{N\to\infty}\frac{\operatorname{Err}(N)}{N^r} = C > 0$ for some $r < 0$.
Then, since it also holds that $\lim_{N\to\infty}\frac{\operatorname{Err}(\alpha N)}{(\alpha N)^r} = C$ for every $\alpha > 1$,
we have $\lim_{N\to\infty}\frac{\operatorname{Err}(N)}{\operatorname{Err}(\alpha N)} = \alpha^{-r}$.
Taking the logarithm with a base of $\alpha$ shows
\begin{equation}\label{eq:EOC}
  \lim_{N\to\infty}\log_\alpha{\frac{\operatorname{Err}(\alpha N)}{\operatorname{Err}(N)}} = r.
\end{equation}
This $r<0$ is often called the experimental order of convergence (EOC).
The quantity $\eqref{eq:EOC}$ was utilized in Eq. (3.1) \cite{RiceMu1988} to present an
experimental performance analysis of a PDE-solving system called ELLPACK.

Let us define $\operatorname{Err}(N)$ in our setting. We implement the time-discrete evolution of polygons
until $k\Delta t > 0.5$ holds. Let $M$ be the number of implementations and set $T := M\Delta t$ where $\Delta t := 0.1N^{-2}$ as usual.
Then, we define $\operatorname{Err}(N)$ by
\begin{equation*}
  \operatorname{Err}(N) := \frac{1}{M\Delta t}\sum_{k = 1}^M \frac{|\hat{R}_2(k) - R_2(k)|}{\hat{R}_2(k)}\Delta t \approx \frac{1}{T}\int_0^T\frac{|\hat{R}_2(t) - R_2(t)|}{\hat{R}_2(t)}dt,
\end{equation*}
where $R_2(k) := |\cpoint{X}{k}{2}{1}|$ and $\hat{R}_2(t)$ denotes the rigorous solution of $\eqref{eq:RadiusODESystem}$.
This quantity is nothing but the mean value of the relative error between the radii of the approximate solution and the rigorous one.
In terms of this error, let us summarize the result of the experiment.
\begin{table}[h]
  \caption{EOC with respect to $N$.}
  \label{table:EOC_wrt_N}
  \centering
  \begin{tabular}{ccccc}
    $N$ & $\alpha$ & $\Delta t$ & iteration & EOC \\
    \hline\hline
    $50$ & $1.0$ & $8.2\times 10^{-6}$ & $66599$ & N/A \\
    $60$ & $1.2$ & $5.9\times 10^{-6}$ & $93265$ & -0.7104746918300522\\
    $75$ & $1.25$ & $3.9\times 10^{-6}$ & $141604$ & -0.5791828411968422\\
    $100$ & $1.3$ & $2.2\times 10^{-6}$ & $244574$ & -0.5892795844583173\\
    $130$ & $1.3$ & $1.3\times 10^{-6}$ & $405035$ & -0.6569451784970269\\
  \end{tabular}
\end{table}

\begin{figure}[H]
    \centering
    \includegraphics[keepaspectratio, scale=0.5]{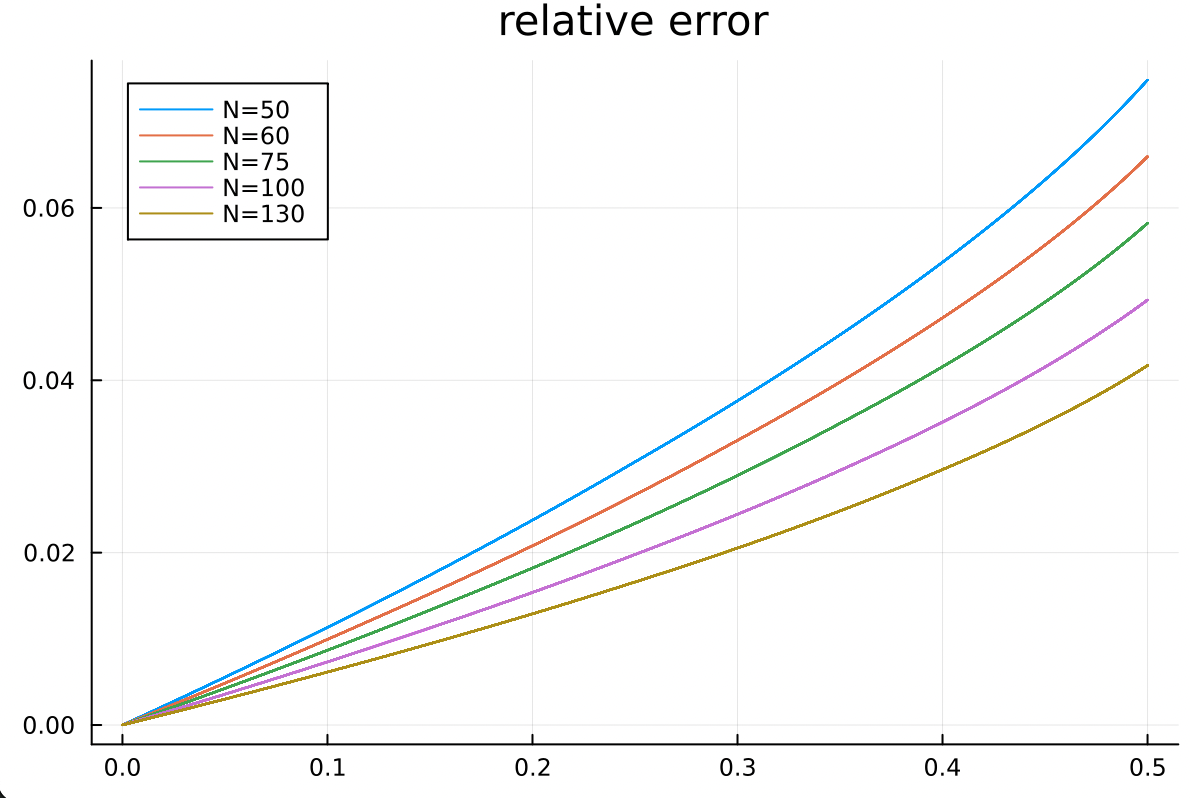}
    \caption{Relative error of $R_2(t)$ for each $N$.}\label{fig:3ConcentricCirclesRelativeError}
\end{figure}

From this result, we can expect the convergence rate of the present scheme to be at least $O\left(\frac{1}{\sqrt{N}}\right)$ as $N\to\infty$.

\subsection{Disappearance of particles}
Bates et al. \cite{Bates1995ANS} treated a case where there are several particles in $\mathbb{R}^2$.
According to \cite{Bates1995ANS}, if the particles are disjoint circles, then the largest circle grows, and 
smaller particles shrink and eventually disappear. See Figure 5 \cite{Garcke2013CurvatureDI} for a 3D case. 
The proposed scheme can also be applied to such cases. In order to observe this phenomenon, 
we prepare four circles in $\mathbb{R}^2$ imitated by $20$ regular polygons, namely, we set $N = 20*4 = 80$.
While evolving, the circles will be removed from the target of the numerical calculation when their area becomes smaller than a prescribed setting. 
\begin{figure}[H]
    \begin{minipage}[b]{0.3\linewidth}
      \centering
      \includegraphics[keepaspectratio, scale=0.3]{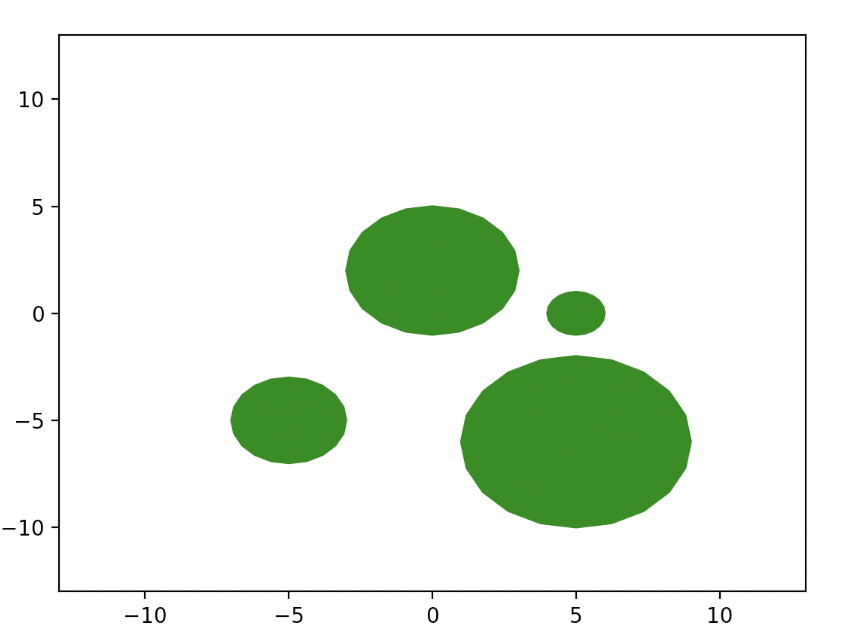}
      \subcaption{step1.}\label{fig:mp_1}
    \end{minipage}
    \begin{minipage}[b]{0.3\linewidth}
      \centering
      \includegraphics[keepaspectratio, scale=0.3]{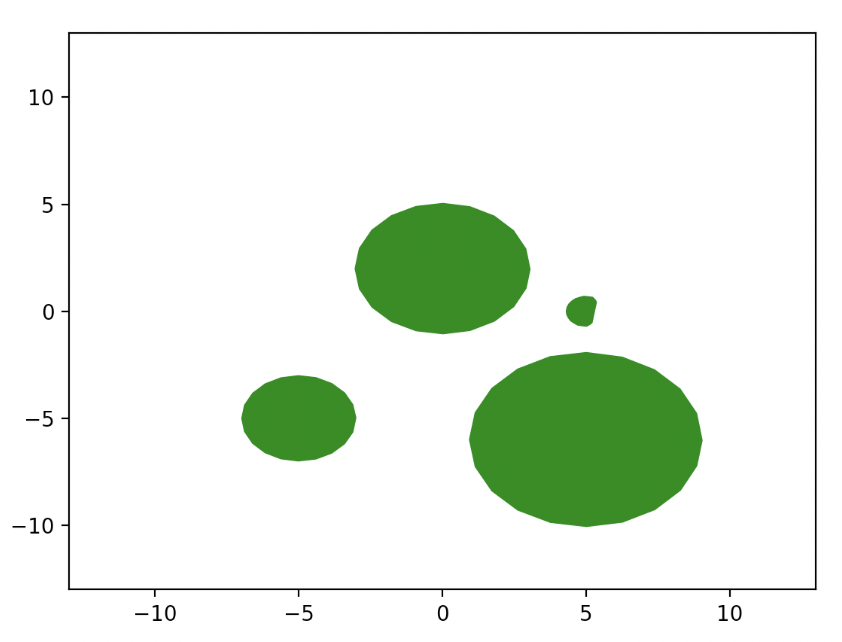}
      \subcaption{step2.}\label{fig:mp_2}
    \end{minipage}
    \begin{minipage}[b]{0.3\linewidth}
      \centering
      \includegraphics[keepaspectratio, scale=0.3]{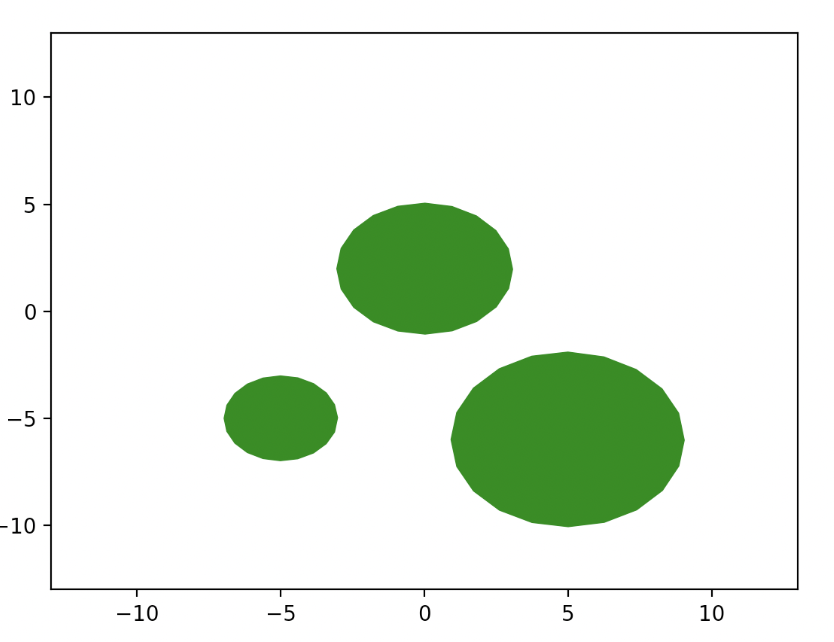}
      \subcaption{step3.}\label{fig:mp_3}
    \end{minipage} \\
    \begin{minipage}[b]{0.3\linewidth}
      \centering
      \includegraphics[keepaspectratio, scale=0.3]{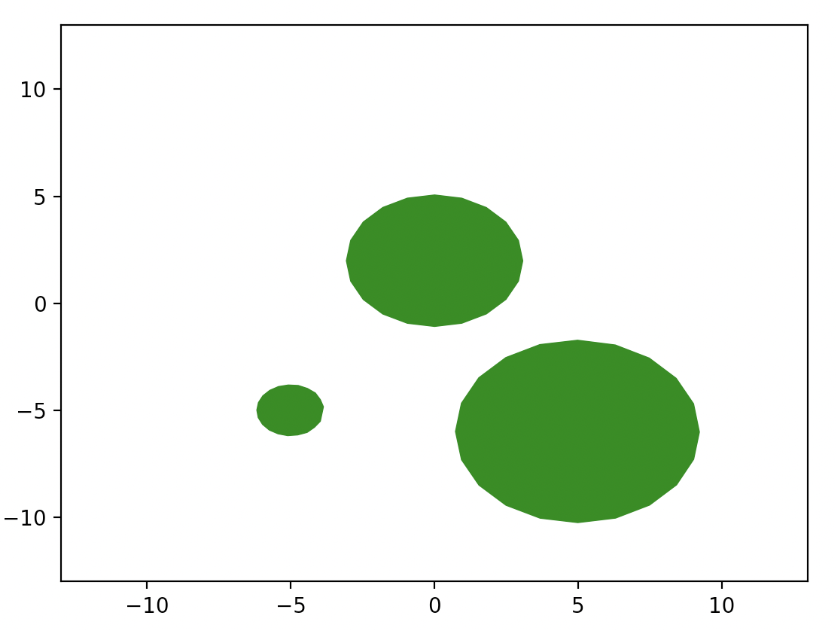}
      \subcaption{step4.}\label{fig:mp_4}
    \end{minipage}
    \begin{minipage}[b]{0.3\linewidth}
      \centering
      \includegraphics[keepaspectratio, scale=0.3]{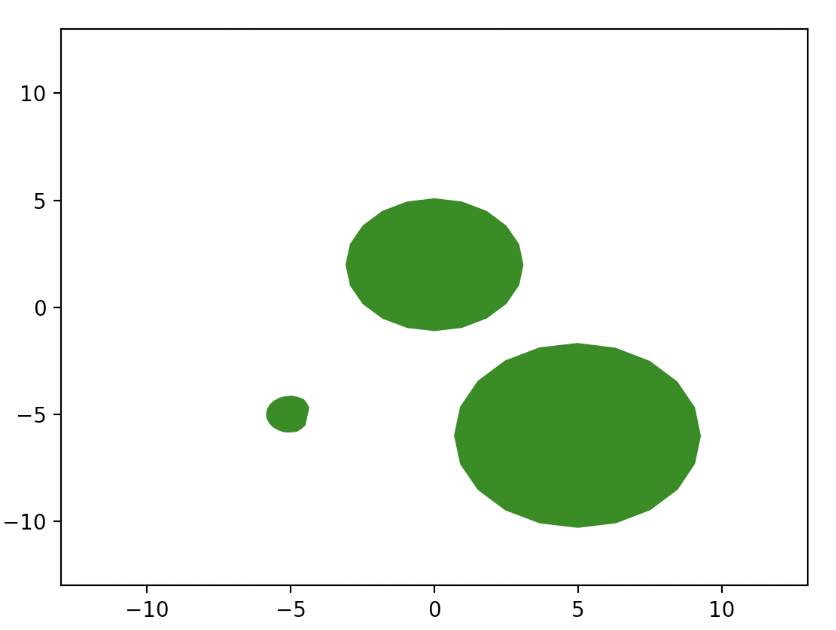}
      \subcaption{step5.}\label{fig:mp_5}
    \end{minipage}
    \begin{minipage}[b]{0.3\linewidth}
      \centering
      \includegraphics[keepaspectratio, scale=0.3]{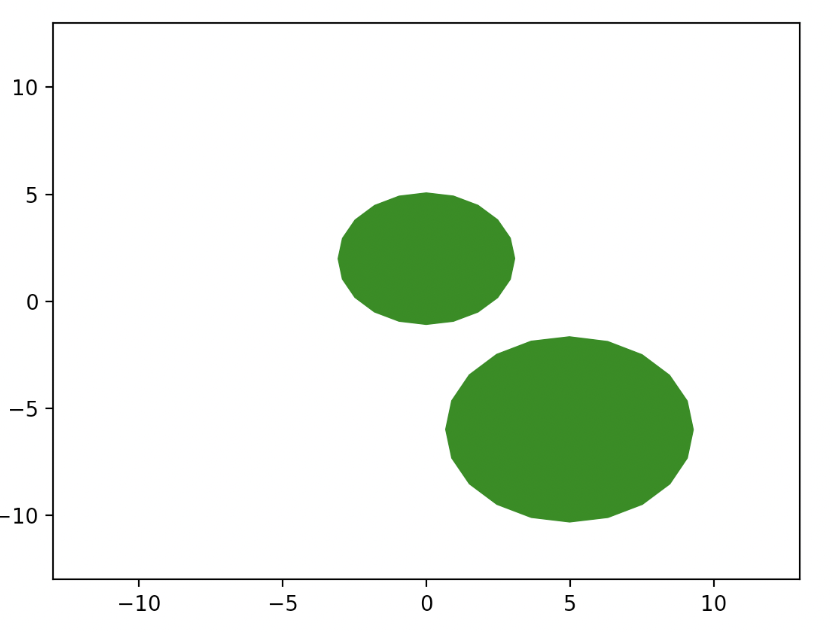}
      \subcaption{step6.}\label{fig:mp_6}
    \end{minipage} \\
    \begin{minipage}[b]{0.3\linewidth}
      \centering
      \includegraphics[keepaspectratio, scale=0.3]{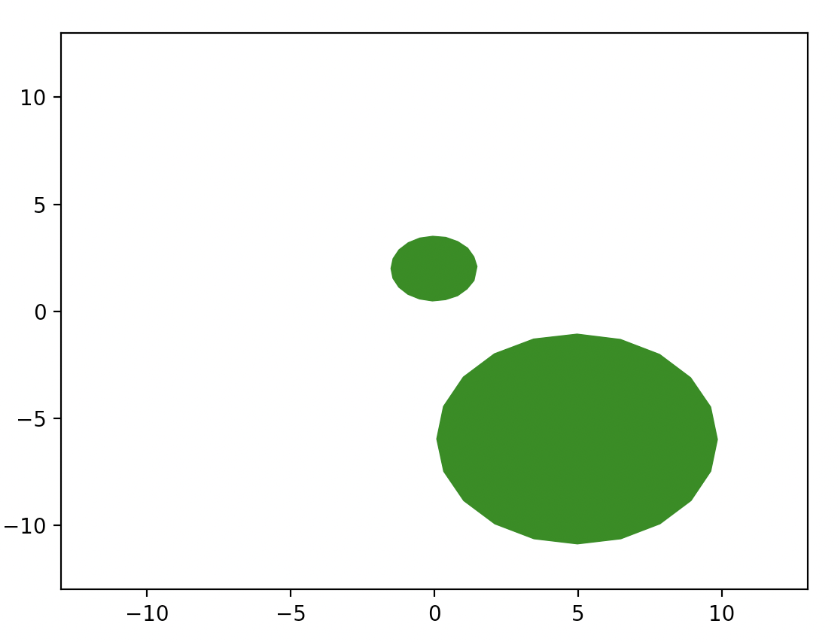}
      \subcaption{step7.}\label{fig:mp_7}
    \end{minipage}
    \begin{minipage}[b]{0.3\linewidth}
      \centering
      \includegraphics[keepaspectratio, scale=0.3]{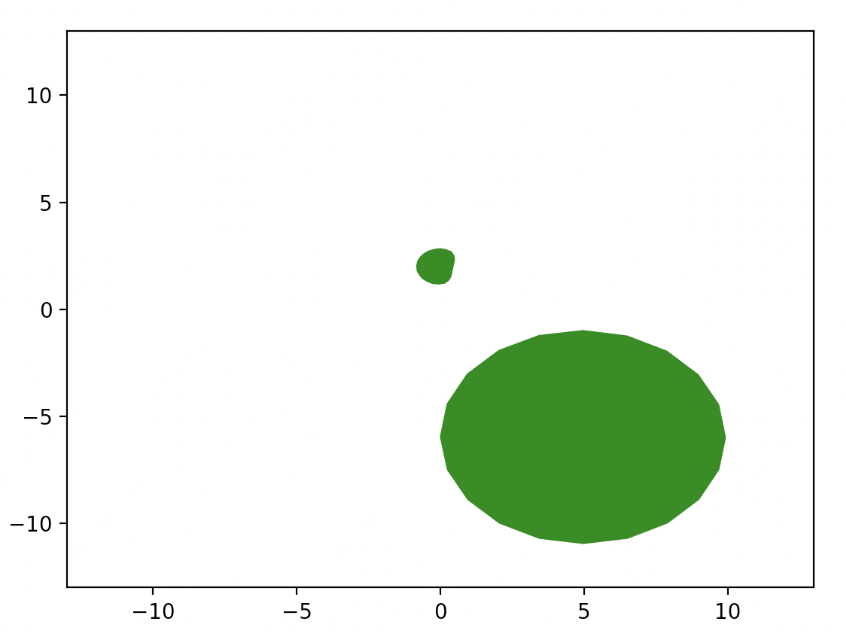}
      \subcaption{step8.}\label{fig:mp_8}
    \end{minipage}
    \begin{minipage}[b]{0.3\linewidth}
      \centering
      \includegraphics[keepaspectratio, scale=0.3]{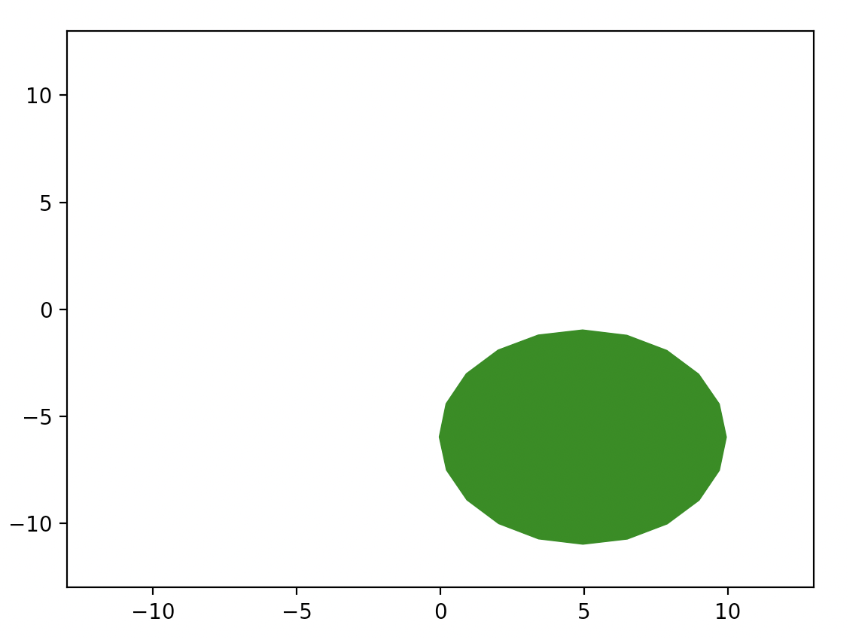}
      \subcaption{step9.}\label{fig:mp_9}
    \end{minipage} \\
    \caption{Evolution of multi particles.}\label{fig:mp}
  \end{figure}
  \begin{figure}[H]
      \centering
      \includegraphics[keepaspectratio, scale=0.5]{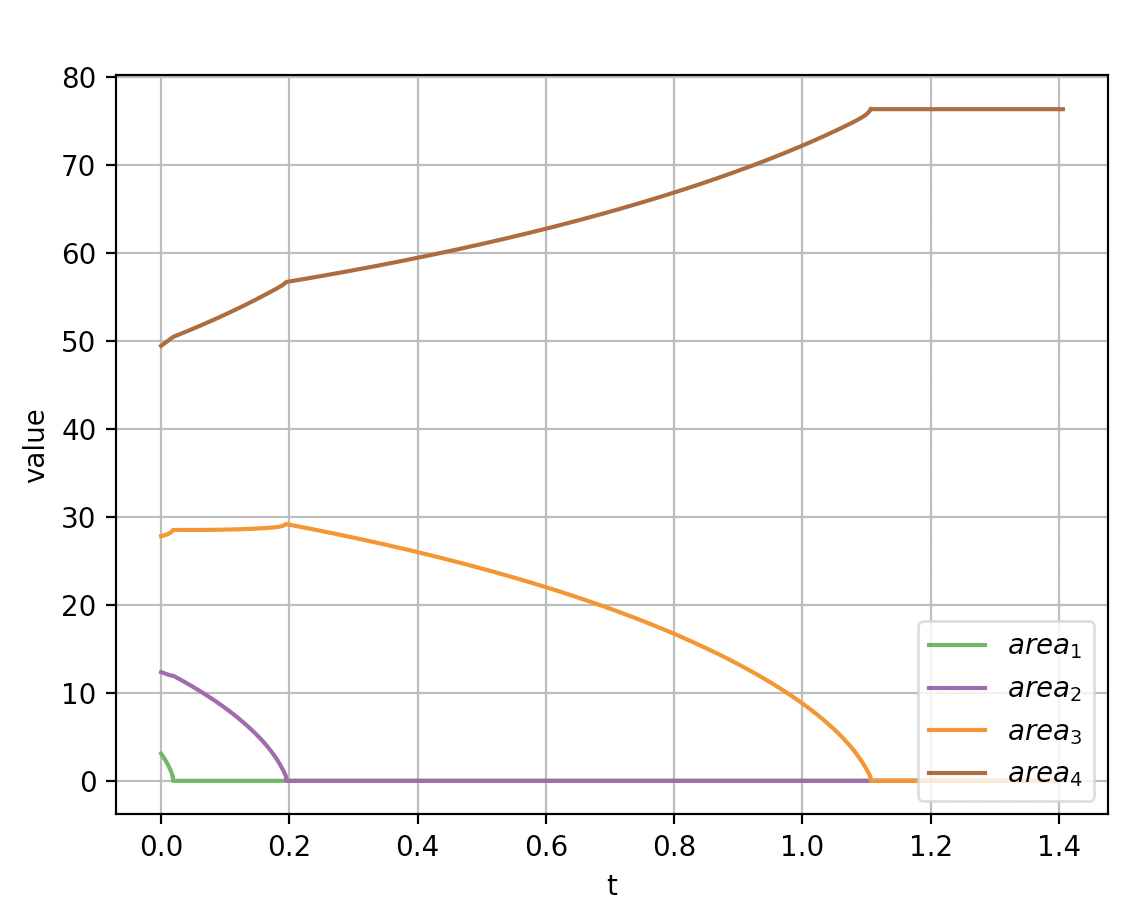}
      \caption{Evolution of areas of each particle}\label{fig:mp_metrics}
  \end{figure}
\subsection{Coalescence of particles}\label{sec:CoP}

Setting the initial datum as two ovals and placing them closely, we can observe the coalescence of particles as shown in Figure 8 \cite{Bates1995ANS}.
This type of phenomenon was rigorously formulated by R\"{o}ger \cite{Rger2005ExistenceOW} with the aid of the notion of varifolds.
He was successful to establish the existence of a weak solution to the Mullins-Sekerka problem without the assumption that
the loss of area never happens. However, there is no uniqueness result of the weak solution.
In this experiment, we check the distance between the collocation points of the ovals step-by-step.
If the distance becomes lower than a prescribed value, then we remove the collocation points and their neighboring collocation points from the ovals.
Second, we artificially connect the ovals and regard them as one polygon.
Each oval has $42$ collocation points, that is, $N = 42*2 = 84$. 
For the numerical result, see Figure \ref{fig:oval}.

\begin{figure}[H]
    \begin{minipage}[b]{0.5\linewidth}
      \centering
      \includegraphics[keepaspectratio, scale=0.25]{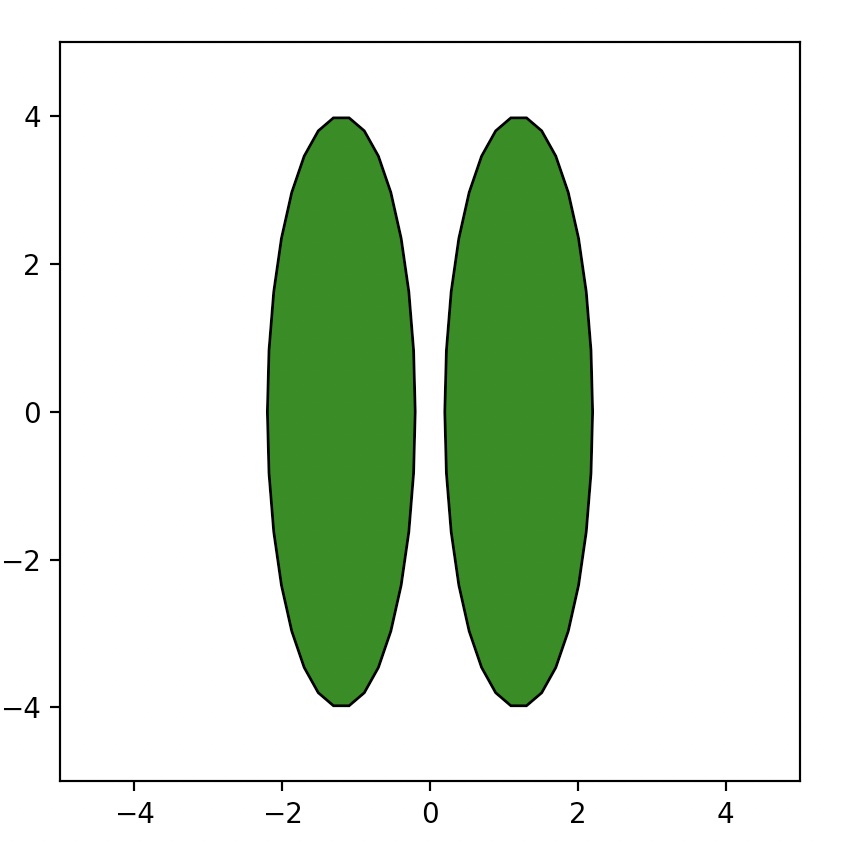}
      \subcaption{\Erase{step1.}\Add{$t = 0.0$}}\label{fig:oval_1}
    \end{minipage}
    \begin{minipage}[b]{0.5\linewidth}
      \centering
      \includegraphics[keepaspectratio, scale=0.25]{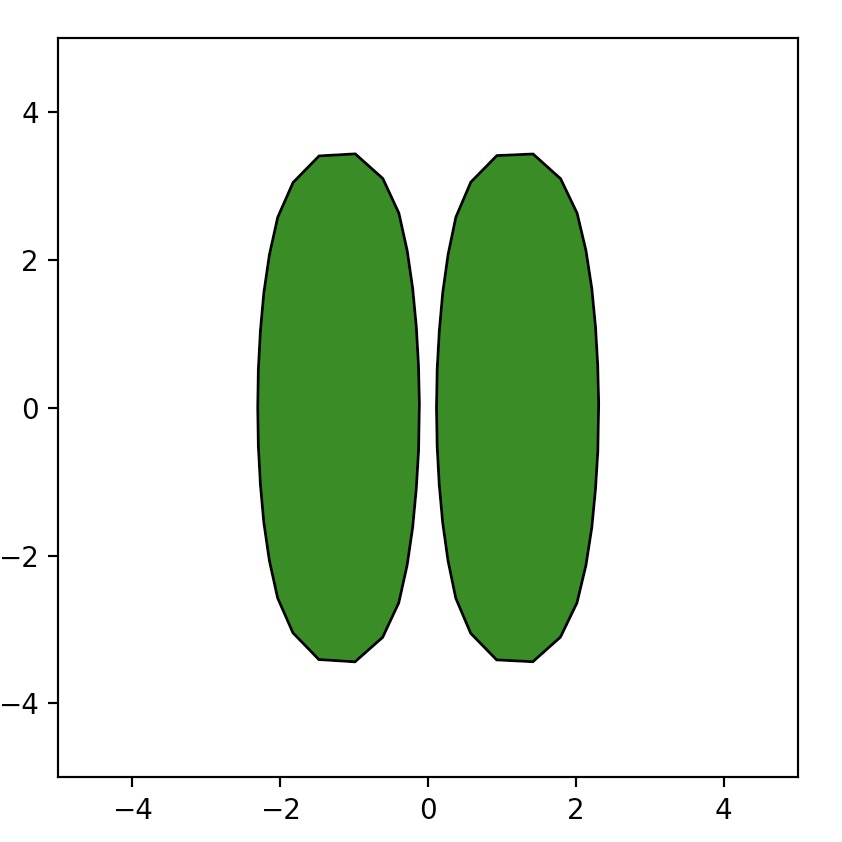}
      \subcaption{\Erase{step2.}\Add{$t = 0.01$}}\label{fig:oval_2}
    \end{minipage} \\
    \begin{minipage}[b]{0.5\linewidth}
      \centering
      \includegraphics[keepaspectratio, scale=0.25]{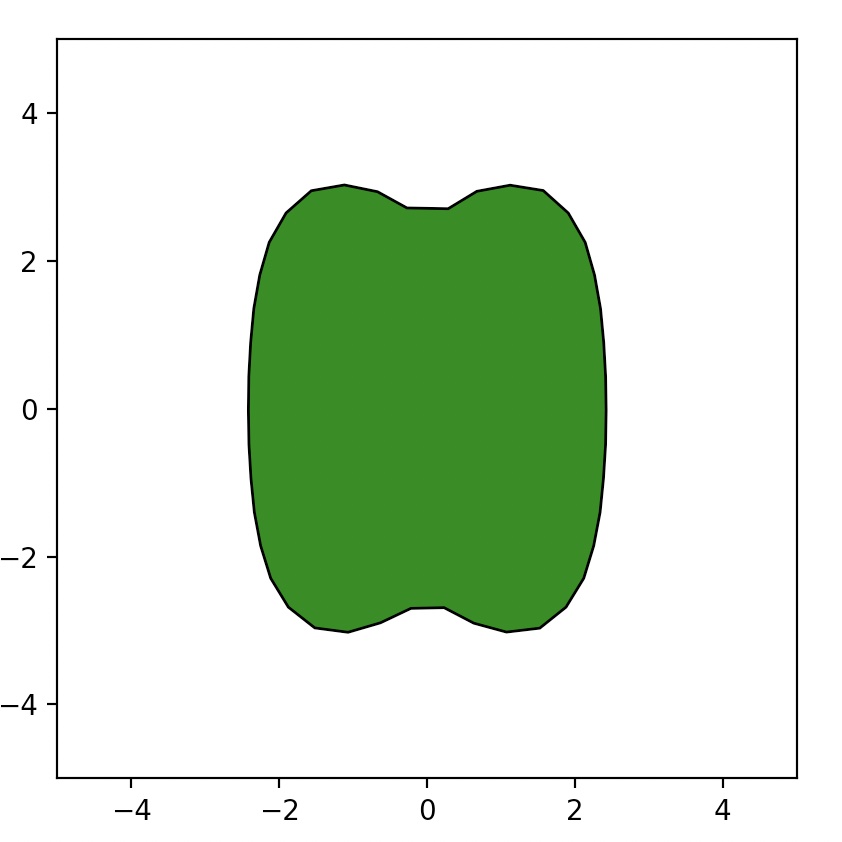}
      \subcaption{\Erase{step4.}\Add{$t = 0.03$}}\label{fig:oval_4}
    \end{minipage}
    \begin{minipage}[b]{0.5\linewidth}
      \centering
      \includegraphics[keepaspectratio, scale=0.25]{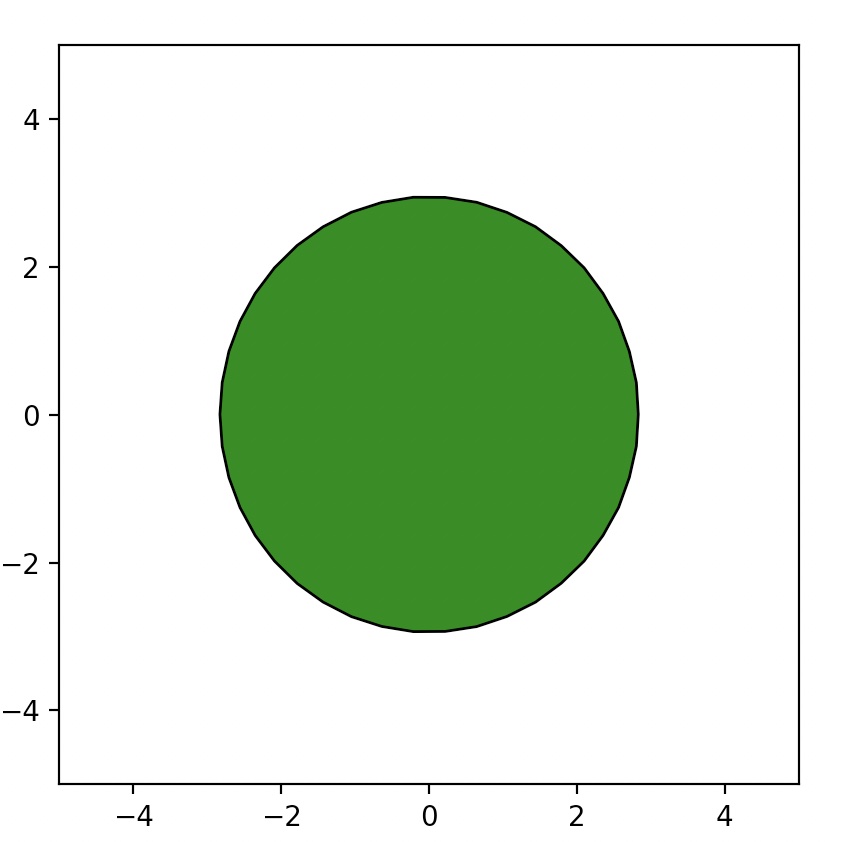}
      \subcaption{\Erase{step5.}\Add{$t = 0.18$}}\label{fig:oval_5}
    \end{minipage}
    \caption{Coalescence of particles.}\label{fig:oval}
  \end{figure}
\subsection{Dumbbell}
When the initial data is a dumbbell, which is the sum set of two circles connected by a narrow bar,
a pinch off phenomenon may occur. For example, let $r,l$ and $b$ be the circle's radius ,the bar's length,
and the thickness, respectively, with $r := 30, l := 5, b := 0.25$.
Then, the thickness of the dumbbell decreases in time and eventually pinches off. 
In this case, we remove the vertices of the polygon when the bar thickness becomes smaller than the prescribed value. 
See Figure \ref{fig:pinchoff}.

\begin{figure}[H]
    \begin{minipage}[b]{0.3\linewidth}
      \centering
      \includegraphics[keepaspectratio, scale=0.3]{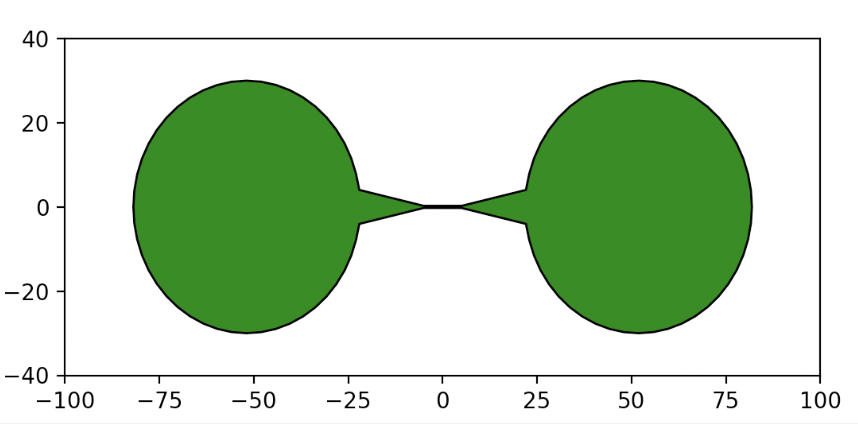}
      \subcaption{step1.}\label{fig:pinchoff_1}
    \end{minipage}
    \begin{minipage}[b]{0.3\linewidth}
      \centering
      \includegraphics[keepaspectratio, scale=0.3]{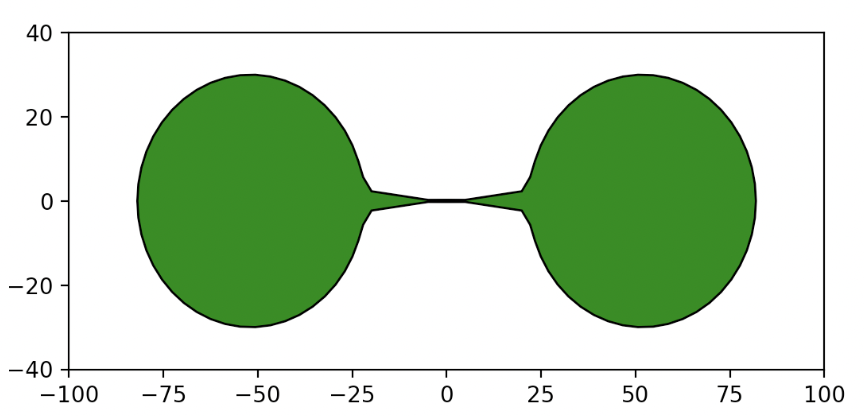}
      \subcaption{step2.}\label{fig:pinchoff_2}
    \end{minipage}
    \begin{minipage}[b]{0.3\linewidth}
      \centering
      \includegraphics[keepaspectratio, scale=0.3]{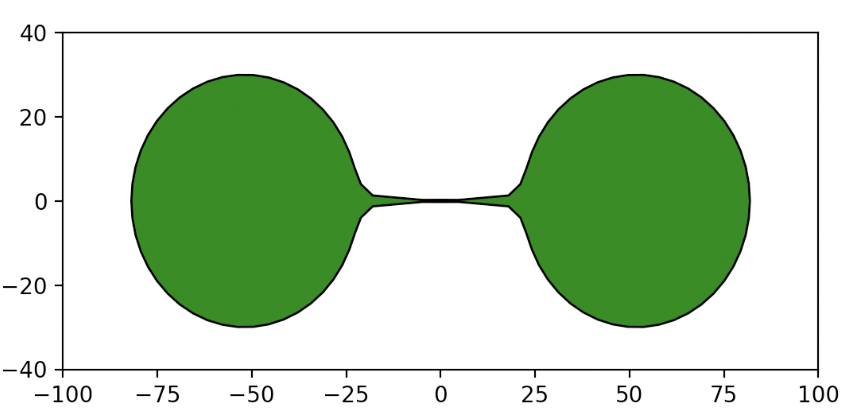}
      \subcaption{step3.}\label{fig:pinchoff_3}
    \end{minipage} \\
    \begin{minipage}[b]{0.3\linewidth}
      \centering
      \includegraphics[keepaspectratio, scale=0.3]{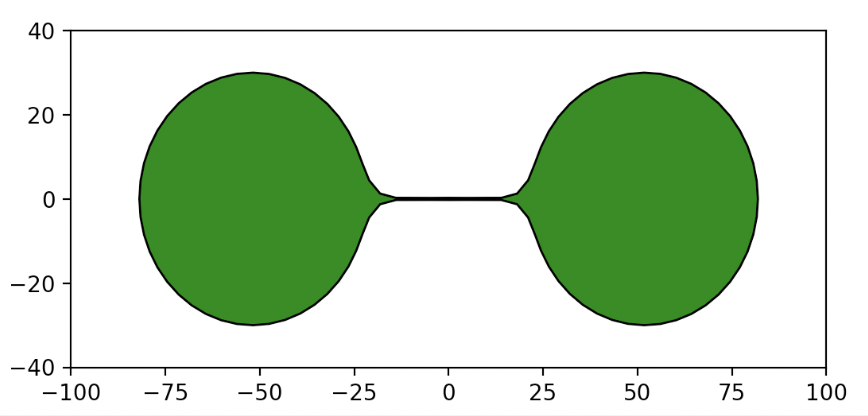}
      \subcaption{step4.}\label{fig:pinchoff_4}
    \end{minipage}
    \begin{minipage}[b]{0.3\linewidth}
      \centering
      \includegraphics[keepaspectratio, scale=0.3]{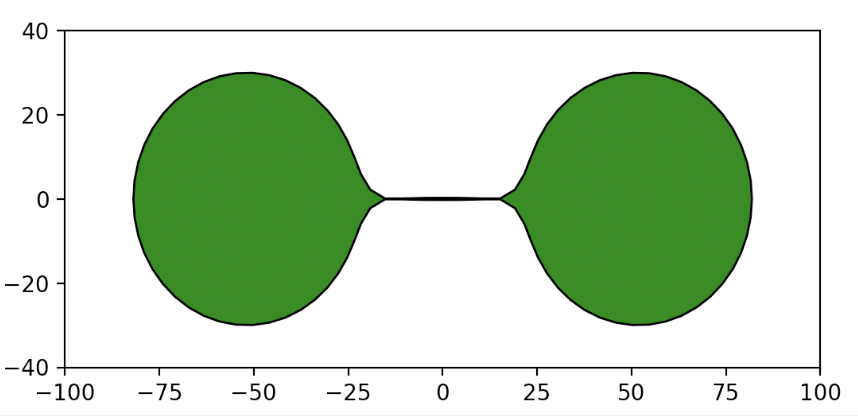}
      \subcaption{step5.}\label{fig:pinchoff_5}
    \end{minipage}
    \begin{minipage}[b]{0.3\linewidth}
      \centering
      \includegraphics[keepaspectratio, scale=0.3]{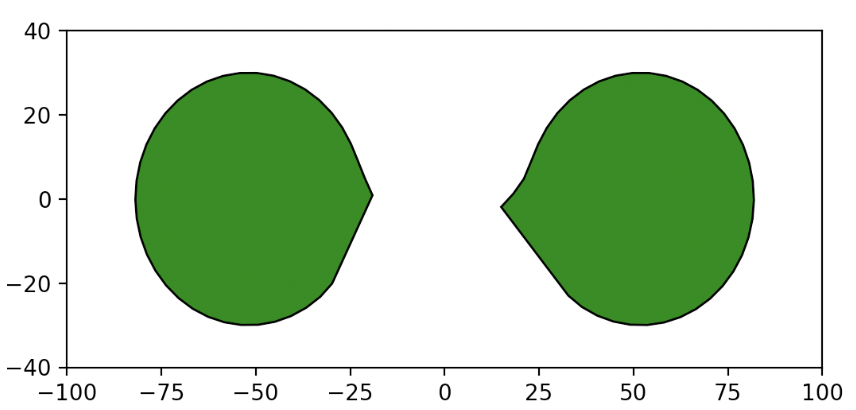}
      \subcaption{step6.}\label{fig:pinchoff_6}
    \end{minipage}
    \caption{Pinch off of a dumbbell.}\label{fig:pinchoff}
  \end{figure}

To avoid this phenomenon, 
we adopt a more general form of the Mullins-Sekerka equation.
We consider the following problem instead of $\eqref{eq:Mullins_Sekerka_intro}$.
\begin{equation}\label{eq:Mullins_Sekerka_coefficient}
    \begin{cases}
        \Delta u = 0\ \ \mbox{in}\ \ \mathbb{R}^2\backslash\Gamma_t, \\
        u = \kappa\ \ \mbox{on}\ \ \Gamma_t, \\
        \nabla u = O\left(\frac{1}{|x|^2}\right)\ \ \mbox{as}\ \ |x|\longrightarrow\infty, \\
        V = \sigma_e\frac{\partial u^e}{\partial\mathbf{n}} - \sigma_i\frac{\partial u^i}{\partial\mathbf{n}}\ \ \mbox{on}\ \ \Gamma_t.
    \end{cases}
\end{equation}

Here, $\sigma_e$ and $\sigma_i$ are positive constants that are generally different.
These quantities designate the heat-diffusion efficiency of the two-phases.
Observe that the area $\mathfrak{A}_t$ is conserved, and the length $\mathfrak{L}_t$ is decreasing in time
even if $\sigma_e\neq \sigma_i$.
For this type of presentation of the Mullins-Sekerka equation, see P.558 \cite{PrussAndSimonett} for instance.
We set $\sigma_e := 1000, \sigma_i := 1$ to prevent the dumbbell from pinching off. See Figure \ref{fig:nopinchoff}.

\begin{figure}[H]
    \begin{minipage}[b]{0.3\linewidth}
      \centering
      \includegraphics[keepaspectratio, scale=0.3]{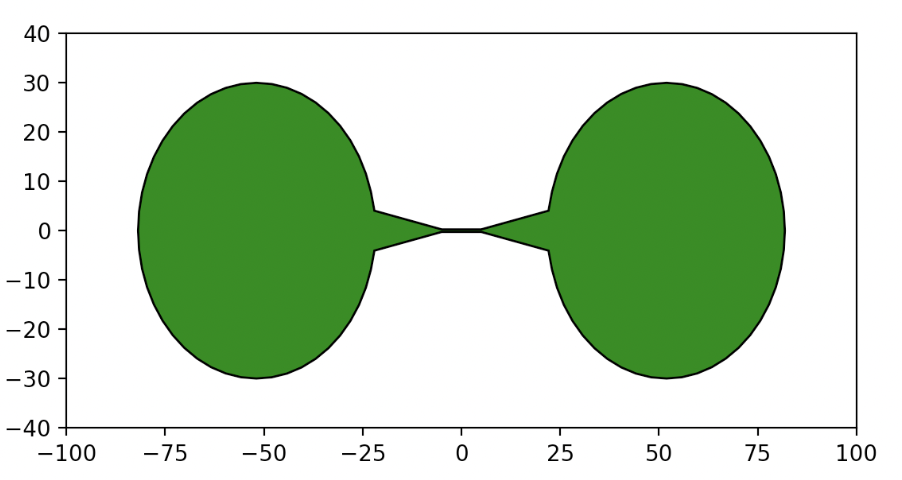}
      \subcaption{step1.}\label{fig:nopf_1}
    \end{minipage}
    \begin{minipage}[b]{0.3\linewidth}
      \centering
      \includegraphics[keepaspectratio, scale=0.3]{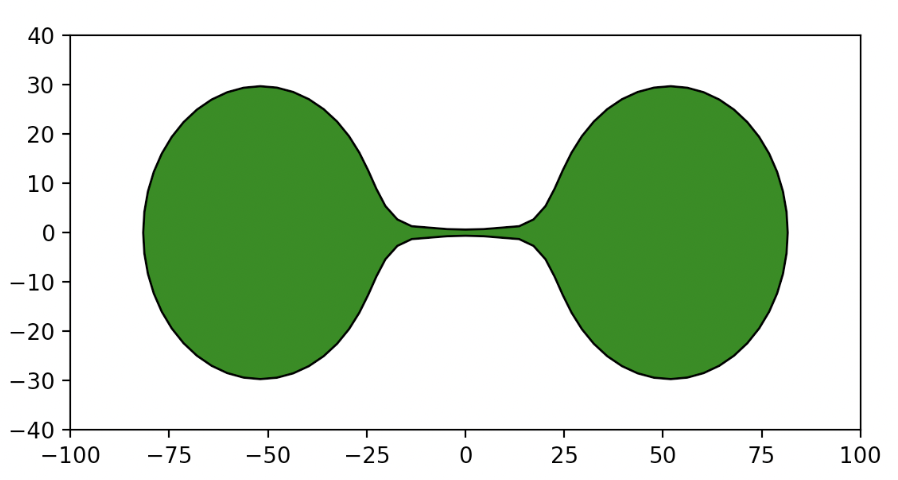}
      \subcaption{step2.}\label{fig:nopf_2}
    \end{minipage}
    \begin{minipage}[b]{0.3\linewidth}
      \centering
      \includegraphics[keepaspectratio, scale=0.3]{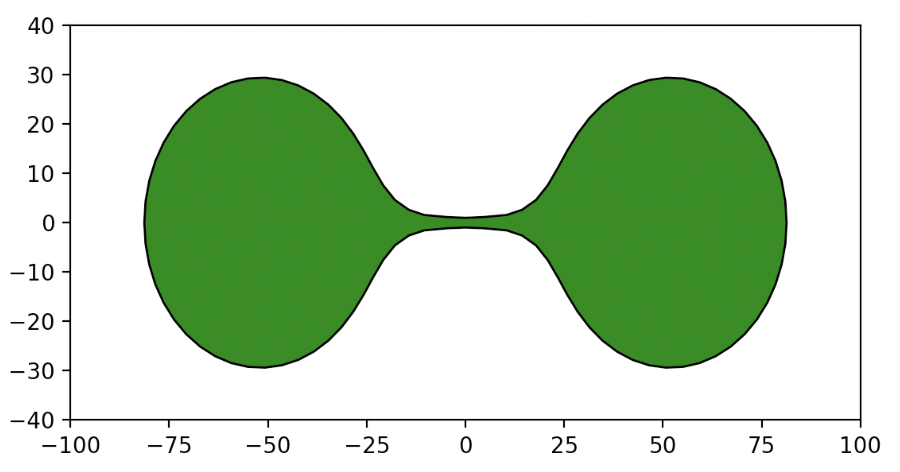}
      \subcaption{step3.}\label{fig:nopf_3}
    \end{minipage} \\
    \begin{minipage}[b]{0.3\linewidth}
      \centering
      \includegraphics[keepaspectratio, scale=0.3]{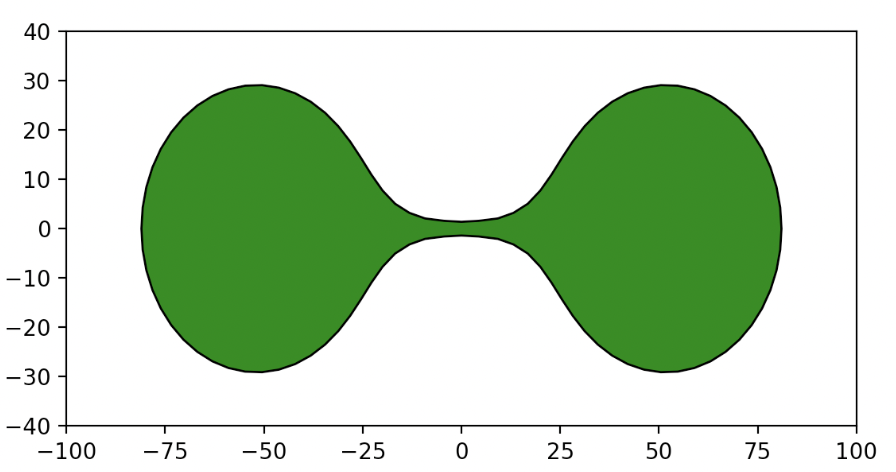}
      \subcaption{step4.}\label{fig:nopf_4}
    \end{minipage}
    \begin{minipage}[b]{0.3\linewidth}
      \centering
      \includegraphics[keepaspectratio, scale=0.3]{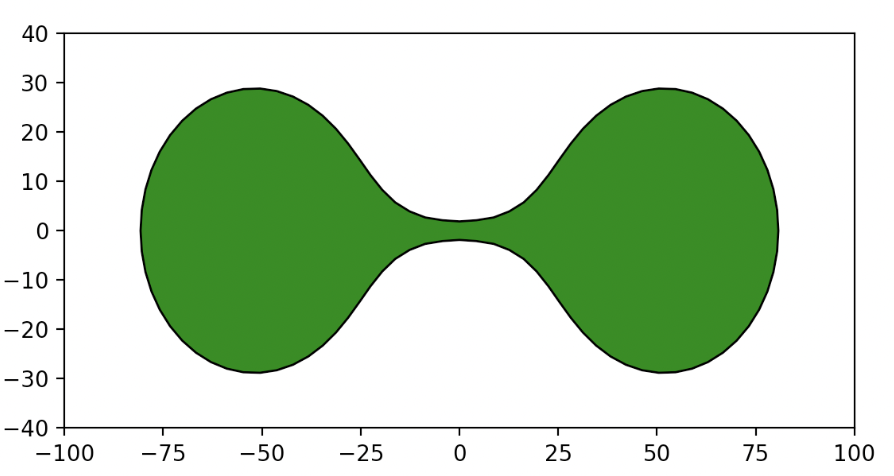}
      \subcaption{step5.}\label{fig:nopf_5}
    \end{minipage}
    \begin{minipage}[b]{0.3\linewidth}
      \centering
      \includegraphics[keepaspectratio, scale=0.3]{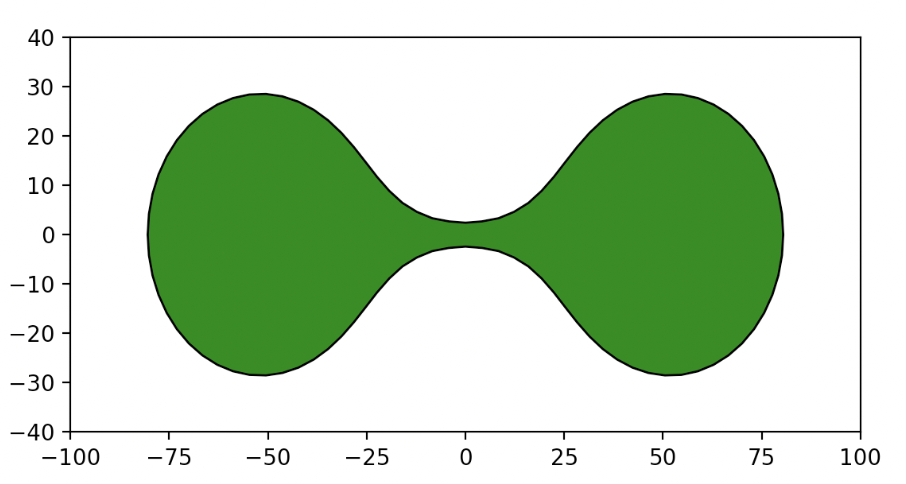}
      \subcaption{step6.}\label{fig:nopf_6}
    \end{minipage}
    \caption{Dumbbell does not pinch off.}\label{fig:nopinchoff}
  \end{figure}

\section{Extension to boundary contact cases}\label{sec:ExtensionToBoundaryContactCases}

Let us consider the Mullins-Sekerka flow with a homogeneous Neumann boundary condition in the half plane $\mathbb{R}^2_+:=\{(x_1,x_2)\in\mathbb{R}^2 \Add{\mid} x_2> 0\}$.
Formally, the problem we are concerned with is as follows:
\begin{equation}\label{eq:boundary_MS}
  \begin{cases}
    \Delta u = 0\ \ \mbox{in}\ \ \mathbb{R}^2_+, \\
    u = \kappa\ \ \mbox{on}\ \ \Gamma_t, \\
    \frac{\partial u}{\partial\mathbf{n}} = 0\ \ \mbox{on}\ \ \partial\mathbb{R}^2_+, \\
    \nabla u = O\left(\frac{1}{|x|^2}\right) \ \ \mbox{as}\ \ |x|\rightarrow\infty, \\
    V = -\left[\frac{\partial u}{\partial\nu}\right] \ \ \mbox{on}\ \ \Gamma_t.
  \end{cases}
\end{equation}
Here, $\Gamma_t$ is an open curve in $\mathbb{R}^2_+$ whose endpoints are bonded on $\partial\mathbb{R}^2_+$ for every $t\geq 0$.
\begin{description}
  \item[Alternative fundamental solutions.]
To solve this problem numerically, we have to use another approximate solution different from $\eqref{eq:def_of_U}$ and $\eqref{eq:def_of_U_bar}$.
Let us explain how to construct approximate solutions to $\eqref{eq:boundary_MS}$ step-by-step.
Similar to the case where $\Gamma_t$ does not touch the boundary, suppose that $\Gamma_t^N$ is an $N$ polygon whose vertices are $\mathbf{X}_i(1\leq i\leq N)$.
The indices are numbered counterclockwise so that $\mathbf{X}_1$ and $\mathbf{X}_N$ are the endpoints of $\Gamma_t^N$.
$\mathbf{X}_1$ and $\mathbf{X}_N$ are allowed to move in time only on $\partial\mathbb{R}^2_+$, that is the $x_1$-axis.
Singular points $y_i^\pm,z_i^\pm(2\leq i\leq N)$ are set as $y_i^\pm := \mathbf{X}_i^* \pm d\mathbf{n}_i, \Add{z_i^+ := 1000 * y^+_i}, z_i^- := \mathbf{X}_i^* - \frac{d}{2}\mathbf{n}_i$.
Now, we define the approximate solutions $U^+$ and $U^-$ as follows:
\begin{equation}\label{eq:def_of_U_boundary}
  U^+(x) := Q_0^+ + \sum_{j=1}^N Q_j^+\{E(x - y_j^+) - E(x - z_j^+) + E(x - \overline{y}_j^+) - E(x - \overline{z}_j^+)\},
\end{equation}
\begin{equation}\label{eq:def_of_U_bar_boundary}
  U^-(x) := Q_0^- + \sum_{j=1}^N Q_j^-\{E(x - y_j^-) - E(x - z_j^-) + E(x - \overline{y}_j^-) - E(x - \overline{z}_j^-)\}
\end{equation}
where we have used the notation $\overline{y} := (y',-y_n)$ when $y = (y',y_n)\in\mathbb{R}^{N-1}\times\mathbb{R}$.
Observe that the functions $U^+$ and $U^-$ are symmetric across the $x_1$-axis and fulfill the homogeneous Neumann boundary conditions owing to their structures.
The real numbers $Q_i^\pm(0\leq i\leq N)$ are determined with the aid of the Dirichlet boundary condition and the AP property. 
We should note that $\kappa_2$ and $\kappa_N$ are affected by the contact angle between $\Gamma_t^N$ and $\partial\mathbb{R}^2_+$ which is not originally
geometry of $\Gamma_t^N$. Thus, we replace the definition of $\kappa_2$ and $\kappa_N$ with
\begin{equation*}
  \kappa_2 := \frac{2\tan_2}{r_2},\ \kappa_N := \frac{2\tan_N}{r_N}.
\end{equation*}
  \item[Modified normal velocity at the end points.]
  We allowed the endpoints $\mathbf{X}_1$ and $\mathbf{X}_N$ to move in time but bonded them to the $x_1$-axis.
  Thus, it is necessary to modify the normal velocity at the points what we have defined in $\eqref{eq:normal_velocity}$.
  Suppose that the contact angle between $\Gamma_t^N$ and the $x_1$-axis at $\mathbf{X}_1$ equals $\theta$.
  When the edge $[\mathbf{X}_1,\mathbf{X}_2]$ moves at the speed $V$ in the direction of $\mathbf{n}_2$,
  the movement of the endpoint $\mathbf{X}_1$ should equal $\frac{V}{\sin{\theta}}$. Taking this observation into account,
  let us define the new normal velocity vectors at the endpoints as follows:
  \begin{equation*}
    \dot{\mathbf{X}}_1(t) := \left(\frac{V^+_1 + V^-_1}{\sin{(\arccos{(\mathbf{t}_1\cdot\mathbf{e}_1)})}}\right)\mathbf{e}_1,\ \dot{\mathbf{X}}_N(t) := \left(\frac{V^+_{N} + V^-_{N}}{\sin{(\arccos{(\mathbf{t}_{N}\cdot\mathbf{e}_1)})}}\right)\mathbf{e}_1.
  \end{equation*}
  Here, we have used the notation that $\mathbf{e}_1:=(1,0)\in\mathbb{R}^2$.

  \item[Modified UDM.]
We change the way to implement the method of uniform distribution for the vertices $\mathbf{X}_i$ because
$\mathbf{X}_1$ and $\mathbf{X}_N$ never move along the tangential vector of $\Gamma_t$, and
the edge $[\mathbf{X}_N,\mathbf{X}_1]$ is not included in $\Gamma_t^N$. Tangential velocity $W_i(1\leq i\leq N)$ are required to satisfy the following linear system:
\begin{equation*}
  \begin{cases}
  L = \sum_{i=2}^Nr_i,\ \dot{L} = \sum_{i=2}^N\kappa_iv_ir_i, \\
  -\cos_iW_i + \cos_{i+1}W_{i+1} = \frac{\dot{L}}{N-1} + \left(\frac{L}{N-1} - r_i\right)\omega - V_{i+1}\sin_{i+1} - V_i\sin_i\ \ \mbox{for}\ \ 2\leq i\leq N, \\
  W_1 = W_N = 0.
  \end{cases}
\end{equation*}
\end{description}
By these changes from the case where $\Gamma_t$ is a simply closed curve, we can determine the velocity vectors $\dot{\mathbf{X}}_i$
for each $1\leq i\leq N$. We exhibit the numerical result of a simulation for the Mullins-Sekerka flow by means of the way proposed above.
An initial open curve is an L-shaped and contacts the $x_1$-axis at $\ang{90}$ angle at two endpoints.
The collocation points are equidistantly placed and at least at the corners of the character L.
Though the contact angles become different from $\ang{90}$ as soon as we start the program,
they eventually converge to $\ang{90}$, and the shape of the curve tends to a semicircle in finite time.
\begin{figure}[H]
    \begin{minipage}[b]{0.5\linewidth}
      \centering
      \includegraphics[keepaspectratio, scale=0.25]{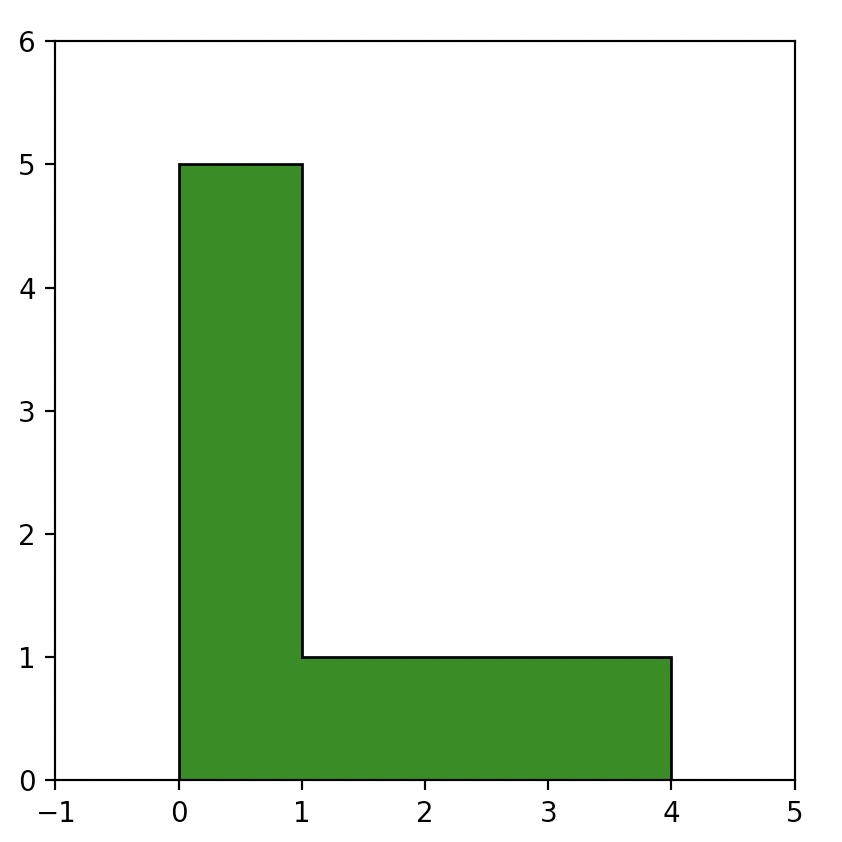}
      \subcaption{\Erase{t=0}\Add{$t = 0.0$}}\label{fig:l_0}
    \end{minipage}
    \begin{minipage}[b]{0.5\linewidth}
      \centering
      \includegraphics[keepaspectratio, scale=0.25]{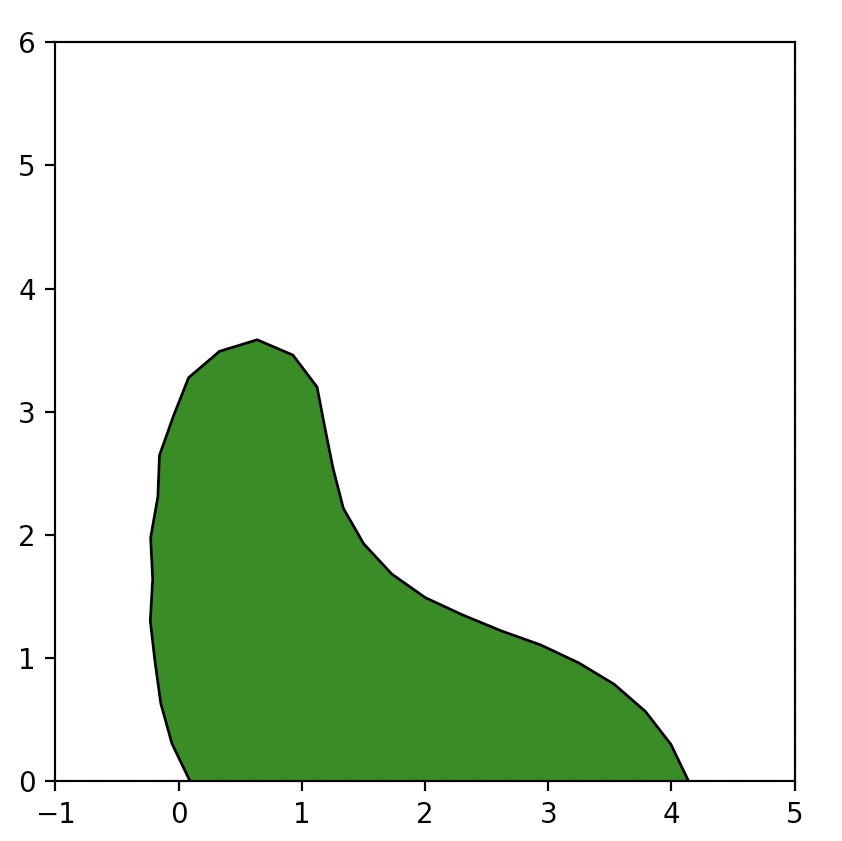}
      \subcaption{\Erase{t=0.059}\Add{$t = 0.08$}}\label{fig:l_2000}
    \end{minipage} \\
    \begin{minipage}[b]{0.5\linewidth}
      \centering
      \includegraphics[keepaspectratio, scale=0.25]{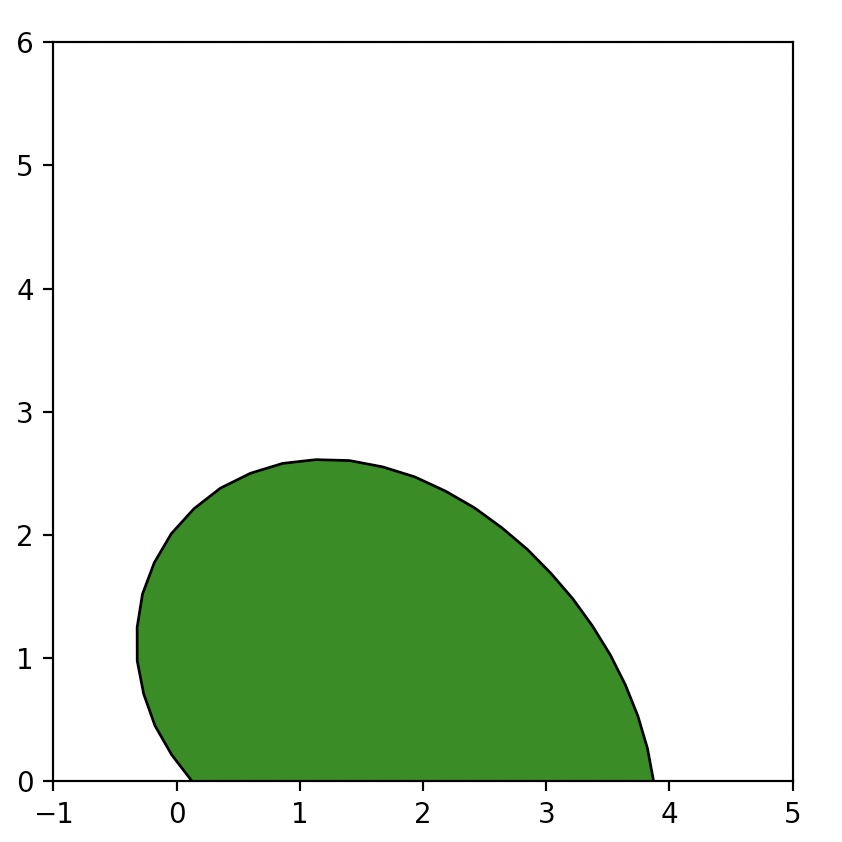}
      \subcaption{\Erase{t=0.59}\Add{$t = 0.49$}}\label{fig:l_20000}
    \end{minipage}
    \begin{minipage}[b]{0.5\linewidth}
      \centering
      \includegraphics[keepaspectratio, scale=0.25]{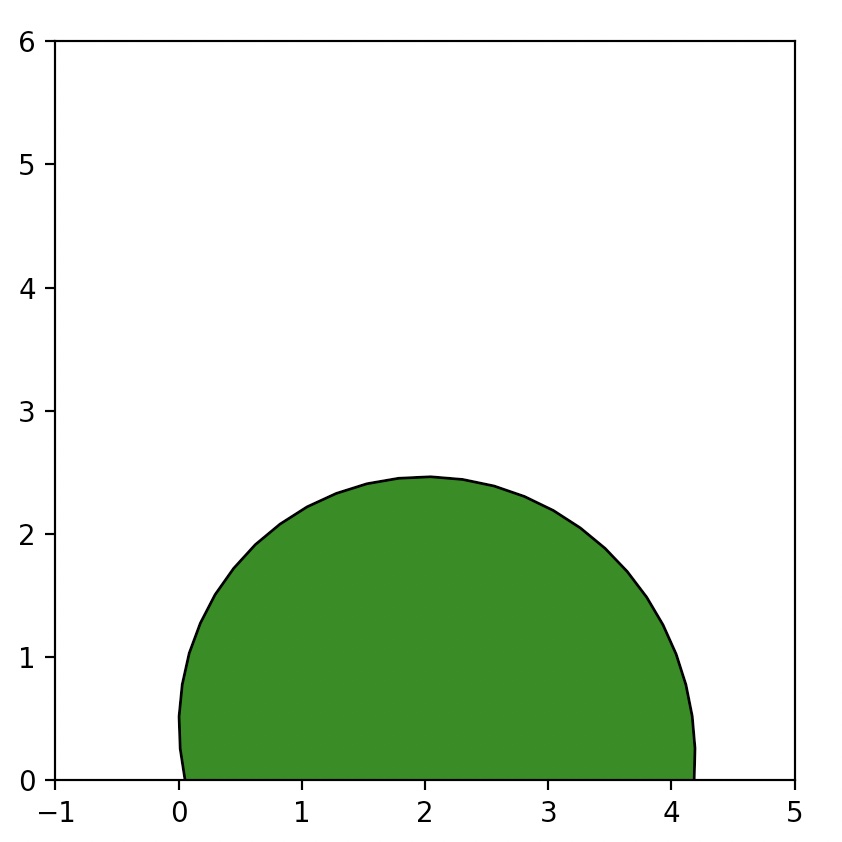}
      \subcaption{\Erase{t=2.37}\Add{$t = 2.23$}}\label{fig:l_80000}
    \end{minipage} \\
    \caption{Evolution of L-shaped open curve.}\label{fig:l}
  \end{figure}
\begin{rem}
  As mentioned in Section \ref{sec:intro}, local well-posedness of the problem $\eqref{eq:boundary_MS}$ has been shown in \cite{AbelsMaxWilke}
  whenever $\Gamma_t$ is perpendicular to the boundary $\partial\mathbb{R}^2_+$, that is to say, $\angle(\Gamma_t,\partial\mathbb{R}^2_+) = \ang{90}$.
  We should note that our scheme yields contact angles that are not equal to $\ang{90}$, and this result is different from that of \cite{AbelsMaxWilke}.
  \Add{We predict that the L-shaped curve cannot be presented by a height function over a fixed reference surface $\Sigma$, which also intersects the boundary of the container at $\ang{90}$.
  Thus, we could not detect $\ang{90}$ numerically because of this extremely short time existence.}
  \Erase{This outcome implies that there might be a solution to $\eqref{eq:boundary_MS}$ with a general contact angle condition.}
\end{rem}

\section{Acknowledgements}
The author has learned how to express curves on the plane and 
background of the method of fundamental solutions from Yazaki's monograph. 
The author is grateful to Professor Norikazu Saito for his suggestion to apply the charge simulation method to this problem.
He reviewed the manuscript and gave me a lot of helpful comments to improve the paper.
Professor Yoshikazu Giga encouraged the author to consider the case where the heat diffusion coefficients are different and examine it. 
\section{Sequence of the proofs}\label{sec:SequenceOfTheProofs}
In this section, we list the proofs of the results presented in Section \ref{sec:PropertiesOfTheScheme} that have been postponed. 
\begin{lem}\label{lem:estimate_H_G}
  Let the initial curve $\Gamma_0$ be a Jordan curve in $\mathbb{R}^2$ with at least $C^2$ regularity.
  Assume that dummy singular points $z_i$ are taken as $z_i = \Add{\mathbf{X}_i^* +} \frac{d}{2}\mathbf{n}_i$ and
  collocation points $\mathbf{X}_i$ are equidistantly placed, that is to say $r_i = \frac{L}{N}$.
  Then, for each $1\leq i,j\leq N$, 
  \begin{equation*}
    |G_{i,j}| = O(1),\  |\mathbf{H}_{i,j}| = O(\sqrt{N}),\ |H_j| = O(\sqrt{N})\ \ \mbox{as}\ \ N\longrightarrow\infty.
  \end{equation*}
  \Add{
    More strongly, it holds that
    \begin{equation*}
      \sup_{N\in\mathbb{N}}\max_{\sigma\in S(N)}\prod_{i = 1}^N|G_{i,\sigma(i)}| < \infty,
    \end{equation*}
    where $S(N)$ denotes the set of all bijections between $\{1,\cdots,N\}$.
  }
\end{lem}
\begin{proof}
  Since $\Gamma$ does not cross itself, we can choose an $N$ so large that \Erase{either} $|\mathbf{X}_i^* - y_j| \geq |\mathbf{X}_i^* - y_i|$
  \Erase{or} and $|\mathbf{X}^*_i - z_j| \geq |\mathbf{X}_i^* - $\Add{$z_i$}\Erase{$y_{i-1}|$} hold for every $1\leq i,j\leq N$. 
  Indeed, viewing the triangle $\triangle \mathbf{X}_i^*y_{i+1}\mathbf{X}_{i+1}^*$, we deduce from the cosine theorem that
  \begin{eqnarray*}
    |\mathbf{X}_i^* - y_{i+1}|^2 &=& d^2 + r^2\cos_i^2 - 2dr\cos_i\cos\left(\frac{\varphi_i}{2} + \frac{\pi}{2}\right) \\
    &=& d^2 + r^2\cos_i^2 + 2dr\cos_i\sin_i \\
    &=& \frac{1}{N} + O_{\geq 0}\left(\frac{1}{N^2}\right) + O\left(\frac{1}{N^\frac{5}{2}}\right) = \frac{1}{N} + O_{\geq 0}\left(\frac{1}{N^2}\right)
  \end{eqnarray*}
  \Add{
  for sufficiently large $N\in\mathbb{N}$. Note that $\sin_i = O(\frac{1}{N})$ since $\sin_i \approx \frac{\varphi_i}{2}$ and $\varphi_i = \hat{\kappa_i}\cdot\frac{L}{N}$ for $N\in\mathbb{N}$ large enough.
  Thus, we have $|\mathbf{X}_i^* - y_{i+1}|\geq \frac{1}{N} = |\mathbf{X}_i^* - y_i|$.
  Moreover, we also take a large $N\in\mathbb{N}$ to obtain the last equality if necessary.
  By a similar argument, we obtain
  \begin{equation*}
    |\mathbf{X}_i^* - z_{i+1}| = \frac{1}{4N} + O_{\geq 0}\left(\frac{1}{N^2}\right).
  \end{equation*}
  }
  \Erase{Assume that the former condition is valid.}
  Fix $1\leq i,j\leq N$ with $i < j$.
  Applying the Taylor expansion to 
  \Erase{$x\mapsto\log{|x|}$ around $\mathbf{X}_i^* - y_j$ and evaluating at $x = \mathbf{X}_i^* - z_j$} 
  \Add{$g(t) := \log{|(\mathbf{X}_i^* - z_j)t + (\mathbf{X}_i^* - y_j)(1-t)|} $ around $t = 0$ and evaluating at $t = 1$} show\Erase{s}
  \Add{
    \begin{multline}\label{eq:TaylorExpansionGij}
      \log{|\mathbf{X}_i^* - z_j|} = \log{|\mathbf{X}_i^* - y_j|} + \frac{\mathbf{X}_i^*-y_j}{|\mathbf{X}_i^* - y_j|^2}\cdot (y_j - z_j) \\
      + \frac{1}{2}\cdot\frac{|y_j-z_j|^2|\mathbf{X}_i^*-y_j + (y_j - z_j)s|^2 - 2\{(\mathbf{X}_i^*-y_j + (y_j - z_j)s)\cdot(y_j-z_j)\}^2}{|\mathbf{X}_i^* - y_j + (y_j - z_j)s|^4}
    \end{multline}
    for some $s\in(0,1)$. 
  }
  It readily follows that $|y_j - z_j| = \frac{d}{2} = \frac{1}{2\sqrt{N}}$. \Erase{$= O(\frac{1}{\sqrt{N}})$}
  \Add{
    For the second term on the right-hand side of $\eqref{eq:TaylorExpansionGij}$, we compute:
    \begin{equation}\label{eq:FirstOrder}
      \left|\frac{\mathbf{X}_i^* - y_j}{|\mathbf{X}_i^* - y_j|^2}\cdot(y_j - z_j)\right|\leq \frac{|y_j - z_j|}{|\mathbf{X}_i^* - y_i|} = \frac{\frac{1}{2\sqrt{N}}}{\frac{1}{\sqrt{N}}} = \frac{1}{2}.
    \end{equation}
    Let us estimate the third term. To this end, note that
    \begin{equation*}
      |\mathbf{X}_i^* - y_j + (y_j - z_j)s| \geq |\mathbf{X}_i^* - y_j|\land|\mathbf{X}_i^* - z_j|
    \end{equation*}
    holds due to an elementary geometric observation.
    Suppose that $|\mathbf{X}_i^* - y_j + (y_j - z_j)s| \geq |\mathbf{X}_i^* - y_j|$.
    Then, the third term is bounded by
    \begin{equation*}
      \frac{1}{2}\left(\frac{|y_j-z_j|^2}{|\mathbf{X}_i^* - y_j|^2} + \frac{2|y_j-z_j|^2}{|\mathbf{X}_i^* - y_j|^2}\right)\leq \frac{3}{2}\cdot\frac{|y_j-z_j|^2}{|\mathbf{X}_i^* - y_i|^2} \leq \frac{3}{2}\cdot\frac{\frac{1}{4N}}{\frac{1}{N}} = \frac{3}{8}.    
    \end{equation*}
    On the other hand, if $|\mathbf{X}_i^* - y_j + (y_j - z_j)s| \geq |\mathbf{X}_i^* - z_j|$, then the third term on the right-hand side of $\eqref{eq:TaylorExpansionGij}$ is bounded by
    \begin{equation*}
      \frac{1}{2}\left(\frac{|y_j-z_j|^2}{|\mathbf{X}_i^* - z_j|^2} + \frac{2|y_j-z_j|^2}{|\mathbf{X}_i^* - z_j|^2}\right)\leq \frac{3}{2}\cdot\frac{|y_j-z_j|^2}{|\mathbf{X}_i^* - z_{i}|^2} \leq \frac{3}{2}\cdot\frac{\frac{1}{4N}}{\frac{1}{4N}} = \frac{3}{2}.    
    \end{equation*}
    By these observation, we see that
    \begin{equation*}
      |G_{i,j}| = \frac{1}{2\pi}\log{\frac{|\mathbf{X}_i^* - y_j|}{|\mathbf{X}_i^* - z_j|}} \leq \frac{1}{2\pi}\cdot\left\{\frac{1}{2}+\left(\frac{3}{8}\lor\frac{3}{2}\right)\right\} = \frac{1}{\pi}.
    \end{equation*}
  }
  Next, we consider the function $x\mapsto\frac{x}{|x|^2}$ and again utilize the Taylor expansion to obtain
  \begin{equation*}
    \frac{\mathbf{X}_i^* - z_j}{|\mathbf{X}_i^* - z_j|^2} = \frac{\mathbf{X}_i^*-y_j}{|\mathbf{X}_i^* - y_j|^2} + \frac{1}{|\mathbf{X}_i^* - y_j|^2}\left(I - \frac{2(\mathbf{X}_i^* - y_j)\otimes(\mathbf{X}_i^* - y_j)}{|\mathbf{X}_i^* - y_j|^2}\right)(y_j-z_j) + O(|y_j-z_j|^2).
  \end{equation*}
  Thus, we have
  \begin{eqnarray*}
    |\mathbf{H}_{i,j}| &=& \frac{1}{2\pi}\left|\frac{1}{|\mathbf{X}_i^* - y_j|^2}\left(I - \frac{2(\mathbf{X}_i^* - y_j)\otimes(\mathbf{X}_i^* - y_j)}{|\mathbf{X}_i^* - y_j|^2}\right)(y_j-z_j) + O(|y_j-z_j|^2)\right| \\
    &\leq & \frac{1}{2\pi}\frac{1}{|\mathbf{X}_i^* - y_{i+1}|^2}\left|I - \frac{2(\mathbf{X}_i^* - y_j)\otimes(\mathbf{X}_i^* - y_j)}{|\mathbf{X}_i^* - y_j|^2}\right|_2|y_j- z_j| + O(|y_j-z_j|^2) \\
    &=& O(N) \times O\left(\frac{1}{\sqrt{N}}\right) + O\left(\frac{1}{N}\right) = O(\sqrt{N})\ \ \mbox{as}\ \ N\rightarrow\infty.
  \end{eqnarray*}
  Finally, let us estimate $H_j$. By the definition,
  \begin{equation*}
    |H_j| = \left|-\sum_{i=1}^N \mathbf{H}_{i,j}\cdot\mathbf{n}_ir_i\right| \leq \sum_{i=1}^N O(\sqrt{N})\times\frac{L}{N} = O(\sqrt{N}).
  \end{equation*}
\end{proof}
Before proving the main theorems, we prepare some lemmas.

\begin{lem}\label{lem:estimate_S_i}
    For each $1\leq i\leq N$, let $S_i$ be defined by
    \begin{equation*}
        S_i(x) := U^+(x)\nabla U^+(x)\cdot\mathbf{n}_i - U^-(x)\nabla U^-(x)\cdot\mathbf{n}_i.
    \end{equation*}
    Moreover, suppose that $y^\pm_i = \mathbf{X}_i^* \pm d\mathbf{n}_i, z^\pm_i = \mathbf{X}_i^* \pm \tilde{d}\mathbf{n}_i$ for some $d,\tilde{d} > 0$.
    Then,
    \begin{equation*}\label{eq:nabla_Si_boundedness}
        \sup_{1\leq i\leq N}{\|\nabla S_i\|_{L^\infty(\Gamma_i)}} \leq C
    \end{equation*}
    holds where 
    \begin{equation*}
      C := \frac{2L^2}{\pi^2d^2}\left(\sum_{i=0}^N|Q_j|\right)^2\left(1+\sqrt{3}\pi\log{\frac{L+\tilde{d}}{d}}\right).
    \end{equation*}
\end{lem}

\begin{proof}
    Direct differentiation shows that
    \begin{equation}\label{eq:S_i}
        \nabla S_i = \nabla U^+\nabla U^+\cdot\mathbf{n}_i + U^+\nabla^2U^+\mathbf{n}_i - \nabla U^-\nabla U^-\cdot\mathbf{n}_i - U^-\nabla^2U^-\mathbf{n}_i.
    \end{equation}

    Here, the symbol $\nabla^2$ denotes the Hessian. 
    The first two terms of $\eqref{eq:S_i}$ have been estimated in Lemma 1 and the proof of Theorem 2 at \cite{Sakakibara_Yazaki}.
    Since the structure of $U^-$ is the same as $U^+$, the same argument works well to estimate the third and fourth terms.
    We omit the proof.
\end{proof}
\begin{lem}\label{lem:boundedness_nabla_1/x**2}
    Let $f:\mathbb{R}^2\rightarrow\mathbb{R}$ be partially differentiable and satisfy $\nabla f = O\left(\frac{1}{|x|^\alpha}\right)$ as $|x|\rightarrow\infty$
    for some $\alpha > 1$. Then, $f$ is bounded in $\mathbb{R}^2$.
\end{lem}
\begin{proof}
    From the assumption for $\nabla f$, there exists $R>0$ such that for some $C>0$, it holds
    \begin{equation*}
        |x|\geq R \Longrightarrow |\nabla f(x)|\leq \frac{C}{|x|^\alpha}\ \ \forall x\in\mathbb{R}^2.
    \end{equation*}
    Fix any $x\in\mathbb{R}^2\backslash B_R$ and set $x_n := (R+n)\frac{x}{|x|}$ for each $n\in\mathbb{N}\cup\{0\}$.
    We construct a sequence $\{\tilde{x_n}\}_n$ by apply Taylor's expansion around $x_{n-1}$ to find $\tilde{x_n}\in(x_{n-1},x_{n})$ for which 
    \begin{equation*}
        f(x_n) = f(x_{n-1}) + \nabla f(\tilde{x_n})\cdot(x_n - x_{n-1})
    \end{equation*}
    holds. Suppose that $x\in[x_{n_0},x_{n_0+1}]$ for some $n_0\in\mathbb{N}\cup\{0\}$.
    Note that $|x_n - x_{n-1}| = 1$ and $|\nabla f(\tilde{x}_k)|\leq \frac{C}{(k-1 + R)^\alpha}$ are valid from the construction of $\{x_n\}_n$ and $\{\tilde{x}_n\}_n$.
    Moreover, we again use the Taylor expansion around $x = x_{n_0}$ to obtain $\tilde{x}\in(x_{n_0},x)$ such that 
    $f(x) = f(x_{n_0}) + \nabla f(\tilde{x})\cdot (x - x_{n_0})$.
    Then, we compute
    \begin{eqnarray*}
        |f(x)| &\leq & |f(x) - f(x_0)| + |f(x_0)| \\
        &\leq & |f(x) - f(x_{n_0})| + \sum_{k = 1}^{n_0}|f(x_k) - f(x_{k-1})| + |f(x_0)| \\
        &\leq & \frac{C}{(n_0 + R)^\alpha} + \sum_{k = 1}^{n_0}|\nabla f(\tilde{x}_k)||x_k - x_{k-1}| + |f(x_0)| \\
        &\leq & \sum_{k = 0}^{n_0}\frac{C}{(k+R)^\alpha} + \max_{\overline{B}_R}|f| \leq \sum_{k=0}^\infty \frac{C}{(k+R)^\alpha }+ \max_{\overline{B}_R}|f|.
    \end{eqnarray*}

    Note that the right-hand side of the above inequality is finite and independent of the choice of $x\in\mathbb{R}^2\backslash B_R$.
    Therefore, we can conclude that $f$ is bounded in $\mathbb{R}^2$.
  
\end{proof}

\begin{lem}\label{lem:U_minus_is_bounded}
    The approximate solution $U^-$ defined by $\eqref{eq:def_of_U_bar}$ satisfies $\nabla U^-(x) = O\left(\frac{1}{|x|^2}\right)$ as $|x|\longrightarrow\infty$.
    Moreover, $U^-$ is bounded in $\mathbb{R}^2\backslash \Omega$.
\end{lem}
\begin{proof}
    The first assertion is straightforward from the following calculation:
    \begin{eqnarray}
        \nabla {U}^-(x) &=& \sum_{i=1}^N {Q}^-_i(\nabla E(x - y_i) - \nabla E(x - z_i)) = \sum_{i=1}^N {Q}^-_i\left(\frac{x-y_i}{|x-y_i|^2} - \frac{x-z_i}{|x-z_i|^2}\right) \\
        &=& \sum_{i=1}^N{Q}^-_i\left(\frac{x-y_i}{|x-y_i|^2} - \frac{x}{|x|^2} - \frac{x-z_i}{|x-z_i|^2} + \frac{x}{|x|^2}\right)\label{eq:plus_minus_teq} \\ 
        &=& \sum_{i=1}^N{Q}^-_i\left(\frac{|x|^2-|x-y_i|^2}{|x-y_i|^2|x|^2}x - \frac{y_i}{|x-y_i|^2} + \frac{|x-z_i|^2-|x|^2}{|x-z_i|^2|x|^2}x + \frac{z_i}{|x-z_i|^2}\right)\label{eq:plus_minus_teq_2} \\
        &=& O\left(\frac{1}{|x|^2}\right)\ \ \mbox{as}\ \ |x|\rightarrow\infty.
    \end{eqnarray}
    The second assertion follows immediately by setting $\alpha := 2$ in Lemma \ref{lem:boundedness_nabla_1/x**2}.
\end{proof}
\begin{rem}
    An argument similar to $\eqref{eq:plus_minus_teq}$ can be found in \cite{Bates1995ANS} in which the boundary integral method was applied to construct
    an approximate solution to $\eqref{eq:Mullins_Sekerka_intro}$. To show decay of the derivative of solutions as $|x|\rightarrow\infty$, they used the area-preserving restriction, namely
    the mean value of the jump of normal derivatives across the phase interface equals zero. Due to the definition $\eqref{eq:def_of_U_bar}$, 
    we do not have to use the second equality of $\eqref{eq:external_problem}$.
\end{rem}
\begin{proof}[Proof of Theorem 1]
  Take an $N\in\mathbb{N}$ so large that $\mathbb{A}_N$ is regular. To confirm that the first value is finite, let us estimate $\det{\mathbb{A}_N}$ and $\widetilde{\mathbb{A}}_N$ separately.
  The cofactor expansion of $\mathbb{A}_N$ \Erase{and Lemma \ref{lem:estimate_H_G}} gives
  \begin{eqnarray*}
    \det{\mathbb{A}_N} &=& \sum_{i=1}^N(-1)^{i+1}\sum_{\sigma\in S(N)}(-1)^{t_\sigma}H_{\sigma(i)}G_{1,\sigma(1)}\cdots G_{i-1,\sigma(i-1)}G_{i+1,\sigma(i+1)}\cdots G_{N,\sigma(N)} \\
    &=& \Add{\sum_{i=1}^N\sum_{\sigma\in S(N)}O(\alpha^N\sqrt{N}) = O(N!N^{\frac{3}{2}}\alpha^N)}
  \end{eqnarray*}
  \Add{where $\alpha > 0$ is a positive constant which is independent of $N$ due to Lemma \ref{lem:estimate_H_G}.}
  We should note that $\sharp S(N) = N!$ to get the last order. 
  By contrast, the cofactors of $\mathbb{A}_N$ can be estimated as follows:
  \begin{equation*}
    (\widetilde{\mathbb{A}}_N)^T_{1,1} = (-1)^{1+1}\det{\mathbb{G}} = O(N!).
  \end{equation*}
  For $2\leq j\leq N+1$,
  \begin{equation*}
    (\widetilde{\mathbb{A}}_N)^T_{1,j} = (-1)^{1+j}\begin{vmatrix}
      1 & G_{1,1} & \cdots & G_{1,j-2} & G_{1,j} & \cdots & G_{1,N} \\
      \vdots &&&&&& \\
      1 & G_{N,1} & \cdots & G_{N,j-2} & G_{N,j} & \cdots & G_{N,N}
    \end{vmatrix} = O(N!).
  \end{equation*}
  For $2\leq i,j \leq N+1$,
  \begin{equation*}
    (\widetilde{\mathbb{A}}_N)^T_{i,j} = (-1)^{i+j}\begin{vmatrix}
      0 & H_1 & \cdots & H_{j-2} & H_{j} & \cdots & H_N \\
      1 & G_{1,1} & \cdots & G_{1,j-2} & G_{1,j} & \cdots & G_{1,N} \\
      \vdots &&&&&& \\
      1 & G_{i-1,1} & \cdots & G_{i-2,j-2} & G_{i-2,j} & \cdots & G_{i-2,N} \\
      1 & G_{i+1,1} & \cdots & G_{i,j-2} & G_{i,j} & \cdots & G_{i,N} \\
      \vdots &&&&&& \\
      1 & G_{N,1} & \cdots & G_{N,j-2} & G_{N,j} & \cdots & G_{N,N}
    \end{vmatrix} = O(N!\sqrt{N}).
  \end{equation*}
  \Erase{By these}\Add{These} estimations tell us that while the absolute sum of the first columns has the order $O(N!N)$,
  that of second columns and beyond have the order $O(N!N^{\frac{3}{2}})$.
  This implies nothing but $\|\widetilde{\mathbb{A}}_N\|_1 = O(N!N^{\frac{3}{2}})$.
  Therefore, we complete the proof of the first value by
  \begin{equation*}
    \|\mathbb{A}_N^{-1}\|_1 = \begin{Vmatrix}
      \frac{1}{\det{\mathbb{A}_N}}\widetilde{\mathbb{A}}_N
    \end{Vmatrix}_1 = \frac{O(N!N^{\frac{3}{2}})}{O(N!N^{\frac{3}{2}})}\times O(N) = O(N).
  \end{equation*}

  Recalling the relationship $\mathbf{A}_N\mathbf{Q}_N = \mathbf{K}_N$, 
  we have $\mathbf{Q}_N = \mathbb{A}^{-1}\mathbf{K}_N$.
  Since $\Gamma$ is smooth, the curvature of $\Gamma$ has a global maximum so that $\kappa_i = O(1)$
  as $N\rightarrow\infty$ for each $1\leq i\leq N$. Then, we compute for each $0\leq i\leq N$,
  \begin{equation*}
    Q_i = \sum_{j = 1}^{N+1}(\mathbb{A}_N^{-1})_{i,j}\kappa_{j-1} = \sum_{j=1}^{N+1}\frac{1}{\det{\mathbb{A}_N}}(\widetilde{\mathbb{A}}_N)_{i,j}\kappa_{j-1} = \sum_{j=1}^{N+1}\frac{O(N!\sqrt{N})}{O(N!N^{\frac{3}{2}})} = O(1).
  \end{equation*}
\end{proof}

\begin{proof}[Proof of \Add{Theorem} \ref{thm:AP}]
    As shown in Proposition \ref{prop:1}, the derivative of $A$ with respect to the time variable is expressed as
    \begin{equation*}
        \dot{A} = \sum_{i=1}^N (v^+_i + v^-_i)r_i + \mbox{errA}.
    \end{equation*}
    Owing to the UDM, it immediately follows that $\mbox{errA} = 0$. On the other hand,
    the second equations of $\eqref{eq:internal_problem}$ and $\eqref{eq:external_problem}$
    require that $\sum_{i=1}^N v^+_ir_i = \sum_{i=1}^N{v}^-_ir_i = 0$. Therefore, we have $\dot{A} = 0$.
\end{proof}

\begin{lem}\label{lem:T_i_and_N_i}
    For each $1\leq i\leq N$,
    \begin{eqnarray*}
        \mathbf{T}_i &=& \cos_i\mathbf{t}_i - \sin_i\mathbf{n}_i = \cos_i\mathbf{t}_{i+1} + \sin_i\mathbf{n}_{i+1}, \\
        \mathbf{N}_i &=& \cos_i\mathbf{n}_i + \sin_i\mathbf{t}_i = \cos_i\mathbf{n}_{i+1} - \sin_i\mathbf{t}_{i+1}.
    \end{eqnarray*}
\end{lem}
\begin{proof}
  This is easily seen if one notes that $\mathbf{T}_i$ is derived by either rotating $\mathbf{t}_i$ by the angle \Erase{$\frac{\phi_i}{2}$}\Add{$\frac{\varphi_i}{2}$} counterclockwise or
  rotating $\mathbf{t}_{i+1}$ with the angle \Erase{$\frac{\phi_i}{2}$}\Add{$\frac{\varphi_i}{2}$} clockwise.
  The second formula immediately follows by rotating the first formula by the angle $\frac{\pi}{2}$.
\end{proof}
We combine the formula $\eqref{eq:evolution_equation}$ with Lemma \ref{lem:T_i_and_N_i} to
get the following formulae:

\begin{lem}
    \begin{eqnarray}
        \dot{\mathbf{X}}_i &=& (V^+_i\cos_i + V^-_i\cos_i - W_i\sin_i)\mathbf{n}_i + (V^+_i\sin_i + V^-_i\sin_i + W_i\cos_i)\mathbf{t}_i \label{eq:lem2_first_formula}\\
        &=& (V^+_i\cos_i + V^-_i\cos_i + W_i\sin_i)\mathbf{n}_{i+1} + (-V^+_i\sin_i - V^-_i\sin_i + W_i\cos_i)\mathbf{t}_{i+1},\label{eq:lem2_second_formula} \\
        \dot{\mathbf{X}}_{i-1} &=& (V^+_{i-1}\cos_{i-1} + V^-_{i-1}\cos_{i-1} + W_{i-1}\sin_{i-1})\mathbf{n}_i \\
        &&+ (-V^+_{i-1}\sin_{i-1} - V^-_{i-1}\sin_{i-1} + W_{i-1}\cos_{i-1})\mathbf{t}_i.
    \end{eqnarray}
\end{lem}
\begin{proof} 
    The formula with respect to $\dot{\mathbf{X}}_i$ is straightforward, by substituting the formulae in Lemma \ref{lem:T_i_and_N_i} to $\eqref{eq:evolution_equation}$.
    On the other hand, the formula with respect to $\dot{\mathbf{X}}_{i-1}$ is obtained by replacing $i$ with $i-1$ in $\eqref{eq:lem2_second_formula}$.
\end{proof}
We are now in the position to prove the results stated in Section \ref{sec:PropertiesOfTheScheme}.

\begin{proof}[Proof of Proposition 1]
    At first, note that 
    \begin{equation*}
        \dot{r}_i = \frac{\mathbf{X}_i - \mathbf{X}_{i-1}}{|\mathbf{X}_i - \mathbf{X}_{i-1}|}\cdot (\dot{\mathbf{X}}_i - \dot{\mathbf{X}}_{i-1}) = \mathbf{t}_i\cdot(\dot{\mathbf{X}}_i - \dot{\mathbf{X}}_{i-1}).
    \end{equation*}
    Then, we can calculate as follows:
    \begin{eqnarray}
        \dot{L} &=& \sum_{i=1}^N \dot{r}_i = \sum_{i=1}^N\mathbf{t}_{i}\cdot(\dot{\mathbf{X}}_i - \dot{\mathbf{X}}_{i-1}) = \sum_{i=1}^N(\mathbf{t}_i - \mathbf{t}_{i+1})\cdot\dot{\mathbf{X}}_i \\
        &=& \sum_{i=1}^N\{(V^+_i\sin_i + V^-_i\sin_i + W_i\cos_i)-(-V^+_i\sin_i-V^-_i\sin_i + W_i\cos_i)\}\label{eq:14} \\
        &=& \sum_{i=1}^N 2(V^+_i+V^-_i)\sin_i = \sum_{i=1}^N 2\cdot\frac{v^+_i+v^+_{i+1} + v^-_i + v^-_{i+1}}{2\cos_i} \cdot \sin_i \\
        &=& \sum_{i=1}^N (v^+_i + v^+_{i+1} + v^-_i + v^-_{i+1})\tan_i = \sum_{i=1}^N (v^+_i + v^-_{i})(\tan_i + \tan_{i-1}) \\
        &=& \sum_{i=1}^N \frac{\tan_i + \tan_{i-1}}{r_i}(v^+_i + v^-_i)r_i = \sum_{i=1}^N \kappa_i(v^+_i + v^-_i)r_i.
    \end{eqnarray}

    Here, to get $\eqref{eq:14}$, we have used $\eqref{eq:lem2_first_formula}$ and $\eqref{eq:lem2_second_formula}$.
    Next, let us confirm $\eqref{eq:prop1-2}$. We deduce
    \begin{eqnarray}
        \dot{A} &=& \sum_{i=1}^N\frac{r_i\mathbf{n}_i + r_{i+1}\mathbf{n}_{i+1}}{2}\cdot \dot{\mathbf{X}}_i \\
        &=& \sum_{i=1}^N \left(\frac{r_i}{2}(V^+_i\cos_i+V^-_i\cos_i - W_i\sin_i) + \frac{r_{i+1}}{2}(V^+_i\cos_i + V^-_{i}\cos_i+ W_i\sin_i)\right)\label{eq:19} \\
        &=& \sum_{i=1}^N \left(\frac{r_i + r_{i+1}}{2}(V^+_i+V^-_i)\cos_i + W_i\sin_i\frac{r_{i+1}-r_i}{2}\right) \\
        &=& \sum_{i=1}^N \left(\frac{(r_i + r_{i+1})(v^+_i + v^+_{i+1} + v^-_i + v^-_{i+1})}{4} + W_i\sin_i\frac{r_{i+1} - r_i}{2}\right) \\
        &=& \sum_{i=1}^N \left(r_i(v^+_i + v^-_i) - \frac{(r_{i+1}-r_i)(v^+_{i+1}+v^-_{i+1} - v^+_i - v^-_i)}{4} + W_i\sin_i\frac{r_{i+1} - r_i}{2}\right) \\
        &=& \sum_{i=1}^N r_i(v^+_i + v^-_i) + \mbox{errA}.
    \end{eqnarray}
    We \Erase{gain}\Add{again} use $\eqref{eq:lem2_first_formula}$ and $\eqref{eq:lem2_second_formula}$ to obtain $\eqref{eq:19}$.
\end{proof}
\begin{proof}[Proof of Theorem 1]
    We follow the proof of Theorem 2 \cite{Sakakibara_Yazaki}.
    Considering the function $\mu\mapsto S_i((1-\mu)x_i^* + \mu x)$, for each $x\in\Gamma_i$, we have 
    \begin{equation}\label{ineq:thm1_4}
        |S_i(x) - S_i(x_i^*)| \leq \int_0^1 |\nabla S_i((1-\mu)x_i^* + \mu x)|d\mu\cdot |x - x_i^*|\leq \|\nabla S_i\|_{L^\infty(\Gamma_i)}\cdot\frac{L}{2N}.
    \end{equation}
    Take $R>0$ so large that $\Omega\subset\subset B_R$.
    Then, we get
    \begin{eqnarray}
        &&\int_{\Omega}|\nabla U^+|^2dx + \int_{B_R\backslash\overline{\Omega}}|\nabla {U}^-|^2dx \\ 
        &=& \sum_{i=1}^N \left(\int_{\Gamma_i} U^+\nabla U^+\cdot\mathbf{n}_i dS - \int_{\Gamma_i}{U}^-\nabla{U}^-\cdot\mathbf{n}_idS\right) + \int_{\partial B_R}{U}^-\nabla{U}^-\cdot\mathbf{n}dS \\
        &=& \sum_{i=1}^N \int_{\Gamma_i} S_i(x) dS + \int_{\partial B_R}{U}^-\nabla{U}^-\cdot\mathbf{n}dS.\label{ineq:thm1_1}
    \end{eqnarray}
    Here, $\mathbf{n}$ denotes the outer unit normal vector field on $\partial B_R$. On the other hand, we have
    \begin{eqnarray}
        \dot{L} &=& \sum_{i=1}^N \kappa_i(v^+_i + {v}^-_i)r_i = \sum_{i=1}^N \left(\int_{\Gamma_i} \kappa_iv^+_i dS + \int_{\Gamma_i} \kappa_i{v}^-_i dS\right) \\
        &=& \sum_{i=1}^N \left(-\int_{\Gamma_i} U^+(\mathbf{X}_i^*)\nabla U^+(\mathbf{X}_i^*)\cdot\mathbf{n}_i dS + \int_{\Gamma_i} {U}^-(\mathbf{X}_i^*)\nabla {U}^-(\mathbf{X}_i^*)\cdot\mathbf{n}_idS \right) \\
        &=& -\sum_{i=1}^N \int_{\Gamma_i} S_i(\mathbf{X}_i^*)dS. \label{ineq:thm1_2}
    \end{eqnarray}
    We estimate the integral part of $\partial B_R$.
    From Lemma \ref{lem:boundedness_nabla_1/x**2}, we obtain
    \begin{equation}\label{ineq:thm1_3}
        \left|\int_{\partial B_R}{U}^-\nabla{U}^-\cdot\mathbf{n}dS\right| \leq C_N\int_{\partial B_R}dS\cdot O\left(\frac{1}{R^2}\right) = C_N\cdot O\left(\frac{1}{R}\right)\ \ \mbox{as}\ \ R\longrightarrow\infty.
    \end{equation}
    Here, we have set $C_N := \sup_{\mathbb{R}^2\backslash(\Omega\cup\Gamma)}{|U^-|}$.
    The identities $\eqref{ineq:thm1_1}, \eqref{ineq:thm1_2}$ and the \Erase{inequality}\Add{inequalities} $\eqref{ineq:thm1_3},\eqref{ineq:thm1_4}$ yield
    \begin{eqnarray*}
        \dot{L} + \int_{\Omega}|\nabla U^+|^2dx + \int_{B_R\backslash\overline{\Omega}}|\nabla {U}^-|^2 dx &=& \left|\sum_{i=1}^N \int_{\Gamma_i}(S_i(x)-S_i(\mathbf{X}_i^*))dS + \int_{\partial B_R}{U}^-\nabla{U}^-\cdot \mathbf{n}dS\right| \\
        &\leq & \sum_{i=1}^N \|\nabla S_i\|_{L^\infty(\Gamma_i)}\cdot \frac{L}{2N}\int_{\Gamma_i}dS + C_N\cdot O\left(\frac{1}{R}\right) \\
        &\leq & \frac{C}{N} + C_N\cdot O\left(\frac{1}{R}\right)\ \ \mbox{as}\ \ R\rightarrow\infty.
    \end{eqnarray*}

    Here, $C > 0$ is determined in Lemma \ref{lem:estimate_S_i}. Since $d,\tilde{d} = O(\frac{1}{\sqrt{N}})$ and $\sum_{i=0}^N|Q_j| = O(N)$ (see Proposition \ref{thm:uniform_boundedness_charges}), we derive $C = O(N^3\log{\sqrt{N}})$.
    We have finished the proof.
\end{proof}
\begin{proof}[Proof of Theorem \ref{thm:discrete_L_formula}]
Noting that
\begin{equation*}
  \mathbf{X}^{n+1}_j = \mathbf{X}^n_j + \Delta t(V^n_j\mathbf{N}^n_j + W^n_j\mathbf{T}^n_j)
\end{equation*}
from the construction of the scheme, together with the Taylor expansion, we get
\begin{eqnarray*}
  && L^{n+1}-L^n = \sum_{j=1}^N(r^{n+1}_j - r^n_j) = \sum_{j=1}^N(|\mathbf{X}^{n+1}_j - \mathbf{X}^{n+1}_{j-1}| - |\mathbf{X}^n_j - \mathbf{X}^n_{j-1}|) \\
  &=& \sum_{j=1}^N\left(\frac{\mathbf{X}^n_j - \mathbf{X}^n_{j-1}}{|\mathbf{X}^n_j - \mathbf{X}^n_{j-1}|}\cdot \{\mathbf{X}^{n+1}_j - \mathbf{X}^{n+1}_{j-1} - (\mathbf{X}^n_j - \mathbf{X}^n_{j-1})\} + O(|\mathbf{X}^{n+1}_j - \mathbf{X}^{n+1}_{j-1} - (\mathbf{X}^n_j - \mathbf{X}^n_{j-1})|^2)\right) \\
  &=& \sum_{j=1}^N\left(\mathbf{t}^n_j\cdot\Delta t(V^n_j\mathbf{N}^n_j + W^n_j\mathbf{T}^n_j - V^n_{j-1}\mathbf{N}^n_{j-1} - W^n_{j-1}\mathbf{T}^n_{j-1}) + \Delta t^2O(|\mathbf{V}^n_j|^2)\right) \\
  &=& \sum_{j=1}^N\left(\Delta t(V^n_j\sin^n_j + V^n_{j-1}\sin^n_{j-1} + W^n_j\cos^n_j - W^n_{j-1}\cos^n_{j-1}) + \Delta t^2O(|\mathbf{V}^n_j|^2)\right).
\end{eqnarray*}
Here, we have utilized the abbreviation $\mathbf{V}^n_j := V^n_j\mathbf{N}^n_j + W^n_j\mathbf{T}^n_j - V^n_{j-1}\mathbf{N}^n_{j-1} - W^n_{j-1}\mathbf{T}^n_{j-1}$ for the remaining term of the Taylor expansion. 
To obtain the fifth equality, one should note that
\begin{equation*}
  \mathbf{t}^n_j\cdot\mathbf{n}^n_{j-1} = \cos{\left(\varphi^n_j-\frac{\pi}{2}\right)}, \mathbf{t}^n_j\cdot\mathbf{n}^n_{j-1} = \cos{\left(\varphi^n_{j-1} + \frac{\pi}{2}\right)}, \mathbf{t}^n_j\cdot\mathbf{t}^n_{j+1} = \cos{\varphi^n_j},\mathbf{t}^n_j\cdot\mathbf{t}^n_{j-1} = \cos{\varphi^n_{j-1}}.
\end{equation*}
Owing to the UDM, it is clear that
\begin{equation}\label{eq:UDM_formula}
  V^n_j\sin^n_j + V^n_{j-1}\sin^n_{j-1} + W^n_j\cos^n_j - W^n_{j-1}\cos^n_{j-1} = \frac{1}{N}\sum_{j=1}^N \kappa^n_jv^n_jr^n_j - \left(r^n_j - \frac{L^n}{N}\right)\cdot 10N
\end{equation}
for each $1\leq j\leq N$. Thus, we have
\begin{eqnarray*}
  L^{n+1} - L^n &=& \Delta t\left(\sum_{j=1}^N\kappa^n_jv^n_jr^n_j - 10N\sum_{j=1}^N\left(r^n_j - \frac{L^n}{N}\right)\right) + \Delta t^2\sum_{j=1}^NO(|\mathbf{V}^n_j|^2) \\
  &=& \Delta t\cdot\sum_{j=1}^N\kappa^n_jv^n_jr^n_j - 10NL^n + 10NL^n +  \Delta t^2\sum_{j=1}^NO(|\mathbf{V}^n_j|^2) \\
  &=& \Delta t\cdot\sum_{j=1}^N\kappa^n_jv^n_jr^n_j + \Delta t^2\sum_{j=1}^NO(|\mathbf{V}^n_j|^2).
\end{eqnarray*}
Second, it is necessary to estimate the term $O(|\mathbf{V}^n_j|^2)$ by means of $N$.
To this end, recall the way to derive the explicit formulae of $W^n_j$ from $\eqref{eq:W_i_explicit_formula}$.
For each $1\leq j\leq N$, $\psi^n_j$ can be estimated as
\begin{equation*}
  \psi^n_j = \frac{1}{N}\sum_{j=1}^N\kappa^n_jv^n_jr^n_j + O(\sqrt{N})\ \ \mbox{as}\ \ N\rightarrow\infty,
\end{equation*}
where the definition of $\psi^n_j$ is
\begin{equation*}
  \psi^n_j := \begin{cases}
    0\ \ \mbox{if}\ \ j = 1, \\
    \frac{1}{N}\sum_{j=1}^N\kappa^n_jv^n_jr^n_j -V^n_{j-1}\sin^n_{j-1} - V^n_j\sin^n_j + \left(\frac{L^n}{N} - r^n_j\right)\omega\ \ \mbox{otherwise}.
  \end{cases}
\end{equation*}
It is easily seen that $\sin^n_j = O(\frac{1}{N})$ as $N\rightarrow\infty$ from $\sin^n_j \simeq \frac{\varphi^n_j}{2}$ for a sufficiently large $N$ and
\begin{equation*}
  O(1) = \hat{\kappa}^n_j = \frac{\varphi^n_j}{\frac{r^n_j + r^n_{j+1}}{2}} = \frac{\varphi^n_j}{O(\frac{1}{N})}.
\end{equation*}
Moreover, we see $V^n_j = O(N^{\frac{3}{2}})$ from Lemma \ref{lem:estimate_H_G} as follows:
\begin{eqnarray*}
V^n_j &=& \Add{\frac{v^n_j + v^n_{j+1}}{2\cos^n_j} = \frac{-\nabla U^+({\mathbf{X}^n_j}^*)\cdot\mathbf{n}^n_j -\nabla U^+({\mathbf{X}^n_{j-1}}^*)\cdot\mathbf{n}^n_{j-1}}{2\cos^n_j}} \\
&=& \Add{\frac{-\sum_{k=1}^N Q^+_k\mathbf{H}^n_{j,k}\cdot\mathbf{n}^n_j-\sum_{k=1}^N Q^+_k\mathbf{H}^n_{j-1,k}\cdot\mathbf{n}^n_{j-1}}{2\cos^n_j} = \sum_{k=1}^N \frac{O(\sqrt{N})}{O(1)} = O(N^{\frac{3}{2}}).}
\end{eqnarray*}
Here, we have used a trivial estimate \Add{$\cos^n_j = O(1)$} because it tends to zero as $N\rightarrow\infty$.
Furthermore, we obtain the order of the constant \Erase{$C^i$}\Add{$C^n$} that appears in $\eqref{eq:W_i_explicit_formula}$
for $N$. The order of the quantity \Erase{$\Psi^i_j$}\Add{$\Psi^n_j$} is calculated as
\begin{eqnarray*}
  \sum_{j=1}^N\Psi^n_j &=&  \sum_{j=1}^N\sum_{k=1}^j\psi^n_k = \sum_{j=1}^N\sum_{k=1}^j\left(\frac{1}{N}\sum_{j=1}^N\kappa^n_jv^n_jr^n_j + O(\sqrt{N}) + O\left(\frac{1}{N}\right)\right) \\
  &=& \frac{N(N+1)}{2}\left(\frac{1}{N}\sum_{j=1}^N\kappa^n_jv^n_jr^n_j + O(\sqrt{N}) + O\left(\frac{1}{N}\right)\right) = O(N)\sum_{j=1}^N\kappa^n_jv^n_jr^n_j + O(N^{\frac{5}{2}}).
\end{eqnarray*}
Thus, we get
\begin{equation*}
  C^n = -\frac{\sum_{j=1}^N\frac{\Psi^n_j}{\cos^n_j}}{\sum_{j=1}^N\frac{1}{\cos^n_j}} = \frac{O(N)\sum_{j=1}^N\kappa^n_jv^n_jr^n_j + O(N^{\frac{5}{2}})}{O(N)} = O(1)\sum_{j=1}^N\kappa^n_jv^n_jr^n_j + O(N^{\frac{3}{2}}).
\end{equation*}
Since $\Psi^n_j$ has a lower order than the quantity above, we deduce that
\begin{equation*}
  W^n_j = \frac{\Psi^n_j + C^n}{\cos^n_j} = O(1)\sum_{j=1}^N\kappa^n_jv^n_jr^n_j + O(N^{\frac{3}{2}}). 
\end{equation*}
Let us return to the discussion on the quantity $O(|\mathbf{V}^n_j|^2)$.
A direct calculation shows 
\begin{eqnarray*}
  |\mathbf{V}^n_j|^2 &=& {V^n_j}^2 + {W^n_j}^2 + {V^n_{j-1}}^2 + {W^n_{j-1}}^2  + 2V^n_jW^n_j\mathbf{N}^n_j\cdot\mathbf{T}^n_j - 2V^n_jV^n_{j-1}\mathbf{N}^n_j\cdot\mathbf{N}^n_{j-1} -2V^n_jW^n_{j-1}\mathbf{N}^n_j\cdot\mathbf{T}^n_{j-1} \\
  &&-2W^n_jV^n_{j-1}\mathbf{T}^n_j\cdot\mathbf{N}^n_{j-1} \Add{- 2W^n_jW^n_{j-1}\mathbf{T}^n_j\cdot\mathbf{T}^n_{j-1}} + 2V^n_{j-1}W^n_{j-1}\mathbf{N}^n_{j-1}\cdot\mathbf{T}^n_{j-1} \\
  &=& {V^n_j}^2 + {W^n_j}^2 + {V^n_{j-1}}^2 + {W^n_{j-1}}^2 -2V^n_jV^n_{j-1}\cos{\left(\frac{\varphi^n_{j-1} + \varphi^n_j}{2}\right)}  \\
  &&-2V^n_jW^n_{j-1}\cos{\left(\frac{\varphi^n_{j-1} + \varphi^n_j}{2} - \frac{\pi}{2}\right)} - 2W^n_jV^n_{j-1}\cos{\left(\frac{\varphi^n_{j-1} + \varphi^n_j}{2} + \frac{\pi}{2}\right)} \\
  &&\Add{- 2W^n_jW^n_{j-1}\cos{\left(\frac{\varphi^n_{j-1} + \varphi^n_j}{2}\right)}}\\
  &=& {V^n_j}^2 + {W^n_j}^2 + {V^n_{j-1}}^2 + {W^n_{j-1}}^2 -2V^n_jV^n_{j-1}\cos{\left(\frac{\varphi^n_{j-1} + \varphi^n_j}{2}\right)}  \\
  &&-2V^n_jW^n_{j-1}\sin{\left(\frac{\varphi^n_{j-1} + \varphi^n_j}{2}\right)} + 2W^n_jV^n_{j-1}\sin{\left(\frac{\varphi^n_{j-1} + \varphi^n_j}{2}\right)} \Add{-2W^n_jW^n_{j-1}\cos{\left(\frac{\varphi^n_{j-1} + \varphi^n_j}{2}\right)}}.
\end{eqnarray*}
Square both sides of the identity $\eqref{eq:UDM_formula}$ to get
\begin{multline}\label{eq:square_UDM_fromula}
  {V^n_j}^2{\sin^n_j}^2 + {V^n_{j-1}}^2{\sin^n_{j-1}}^2 + {W^n_j}^2{\cos^n_j}^2 + {W^n_{j-1}}^2{\cos^n_{j-1}}^2 \\
  + 2V^n_jV^n_{j-1}\sin^n_j\sin^n_{j-1} + 2V^n_jW^n_j\sin^n_j\cos^n_j \Add{- 2V^n_jW^n_{j-1}\sin^n_j\cos^n_{j-1}} \\
  + 2V^n_{j-1}W^n_j\sin^n_{j-1}\cos^n_j - 2V^n_{j-1}W^n_{j-1}\sin^n_{j-1}\cos^n_{j-1} - 2W^n_jW^n_{j-1}\cos^n_j\cos^n_{j-1} \\
  = \left(\frac{1}{N}\sum_{j=1}^N\kappa^n_jv^n_jr^n_j + O(1)\right)^2.
\end{multline}
Adding a proper quantity to both sides of $\eqref{eq:square_UDM_fromula}$ shows that
\begin{multline}\label{eq:estimate_of_diff_between_velocity}
  |\mathbf{V}^n_j|^2 = \left(\frac{1}{N}\sum_{j=1}^N\kappa^n_jv^n_jr^n_j + O(1)\right)^2 + {V^n_j}^2{\cos^n_j}^2 + {V^n_{j-1}}^2{\cos^n_{j-1}}^2 + {W^n_j}^2{\sin^n_j}^2 + {W^n_{j-1}}^2{\sin^n_{j-1}}^2 \\
  - 2V^n_jV^n_{j-1}\cos^n_j\cos^n_{j-1} \Add{- 2V^n_jW^n_{j-1}\sin^n_{j-1}\cos^n_j} + 2V^n_{j-1}W^n_j\sin^n_j\cos^n_{j-1} \\ 
  + 2W^n_jW^n_{j-1}\sin^n_j\sin^n_{j-1} - 2V^n_jW^n_j\sin^n_j\cos^n_j + 2V^n_{j-1}W^n_{j-1}\sin^n_{j-1}\cos^n_{j-1} \\
  = O\left(\frac{1}{N^2}\right)\left(\sum_{j=1}^N\kappa^n_jv^n_jr^n_j\right)^2 + O(\sqrt{N})\sum_{j=1}^N\kappa^n_jv^n_jr^n_j + O(N^3).
\end{multline}
Summing up through $1\leq j\leq N$ in $\eqref{eq:estimate_of_diff_between_velocity}$, we obtain the desired estimate.
\end{proof}
\bibliography{citation_ms_na_thesis}

\begin{thebibliography}{10}

\bibitem{AbelsMaxWilke}
H.~Abels, R.~Maximilian, and M.~Wilke.
\newblock Well-posedness and qualitative behaviour of the {M}ullins-{S}ekerka
  problem with ninety-degree angle boundary contact.
\newblock {\em Math. Ann.}, 381:363--403, 2021.

\bibitem{AlikakosBatesChen1994}
N.D. Alikakos, P.W. Bates, and X.~Chen.
\newblock Convergence of the {C}ahn-{H}illiard equation to the {H}ele-{S}haw
  model.
\newblock {\em Arch. Rational Mech. Anal}, 128:165--205, 1994.

\bibitem{BarrettGarckeRobert}
J.W. Barrett, H.~Garcke, and R.~N{\"u}rnberg.
\newblock On stable parametric finite element methods for the {S}tefan problem
  and the {M}ullins-{S}ekerka problem with applications to dendritic growth.
\newblock {\em J. Comput. Phys.}, 229(18):6270–6299, 2010.

\bibitem{Bates1995ANS}
P.~Bates, X.~Chen, and X.~Deng.
\newblock A numerical scheme for the two phase {M}ullins-{S}ekerka problem.
\newblock {\em Electronic Journal of Differential Equations}, 1995:1--27, 1995.

\bibitem{BatesBrown}
P.W. Bates and S.~Brouwn.
\newblock A numerical scheme for the {M}ullins-{S}ekerka evolution in three
  space dimensions.
\newblock {\em Differential equations and computational simulations (Chengdu,
  1999)}, page 12–26, 2000.

\bibitem{BronsardGarckeStoth1998}
L.~Bronsard, H.~Garcke, and B.~Stoth.
\newblock A multi-phase {M}ullins-{S}ekerka system: matched asymptotic
  expansions and an implicit time discretisation for the geometric evolution
  problem.
\newblock {\em Proc. Roy. Soc. Edinburgh Sect. A}, 128(3):481–506, 1998.

\bibitem{ChenHongYi1996}
X.~Chen, J.~Hong, and F.~Yi.
\newblock Existance uniqueness and regularity of classical solutions of the
  {M}ullins—{S}ekerka problem.
\newblock {\em Communications in Partial Differential Equations},
  21:11-12:1705--1727, 1996.

\bibitem{Epstein}
C.L. Epstein and M.~Gage.
\newblock The {C}urve {S}hortening {F}low, {W}ave {M}otion: {T}heory,
  {M}odelling, and {C}omputation.
\newblock {\em Math. Sci. Res. Inst. Publ., Springer}, 7:15--59, 1987.

\bibitem{EScherSimonett}
J.~Escher and G.~Simonett.
\newblock Classical solutions for {H}ele-{S}haw models with surface tension.
\newblock {\em Adv. Differential Equations}, 2(4):619–642, 1997.

\bibitem{FengProhl2005}
X.~Feng and A.~Prohl.
\newblock Numerical analysis of the {C}ahn-{H}illiard equation and
  approximation of the {H}ele-{S}haw problem.
\newblock {\em Interfaces Free Bound.}, 7(1):1–28, 2005.

\bibitem{Garcke2013CurvatureDI}
H.~Garcke.
\newblock Curvature {D}riven {I}nterface {E}volution.
\newblock {\em Jahresbericht der Deutschen Mathematiker-Vereinigung},
  115:63--100, 2013.

\bibitem{GarckeRauchecker2022}
H.~Garcke and M.~Rauchecker.
\newblock Stability analysis for stationary solutions of the
  {M}ullins-{S}ekerka flow with boundary contact.
\newblock {\em Math. Nachr}, 295(4):683–705, 2022.

\bibitem{HenselStinson202206}
S.~Hensel and K.~Stinson.
\newblock Weak solutions of {M}ullins-{S}ekerka flow as a {H}ilbert space
  gradient flow, 2022.

\bibitem{JulinEtal2022}
V.~Julin, M.~Morini, M.~Ponsiglione, and E.~Spadaro.
\newblock The asymptotics of the area-preserving mean curvature and the
  {M}ullins–{S}ekerka flow in two dimensions.
\newblock {\em Math Annal}, 2022.

\bibitem{KatsuradaOkamotoII}
M.~Katsurada.
\newblock A mathematical study of the charge simulation method ii.
\newblock {\em J. Fac. Sci., Univ. of Tokyo, Sect. IA}, 36:135--162, 1989.

\bibitem{KatsuradaOkamoto}
M.~Katsurada and H.~Okamoto.
\newblock A mathematical study of the charge simulation method i.
\newblock {\em J. Fac. Sci., Univ. of Tokyo, Sect. IA}, 35:507--518, 1988.

\bibitem{LuckhausSturzenhecker}
S.~Luckhaus and T.~Struzenhecker.
\newblock Implicit time discretization for the mean curvature flow equation.
\newblock {\em Calculus of Variations and Partial Differential Equations},
  3:253–271, 1995.

\bibitem{Soni1997TwosidedMF}
Uwe~F. Mayer.
\newblock One-sided {M}ullins-{S}ekerka flow does not preserve convexity.
\newblock {\em Electron. J. Differential Equations.}, 8:1--7, 1993.

\bibitem{Mayer2000}
Uwe~F. Mayer.
\newblock A numerical scheme for moving boundary problems that are gradient
  flows for the area functional.
\newblock {\em European J. Appl. Math.}, 11(1):61–80, 2000.

\bibitem{MurotaKazuo}
K.~Murota.
\newblock Comparison of conventional and "invariant'' schemes of fundamental
  solutions method for annular domains.".
\newblock {\em Jpn J Ind Appl Math}, 1:61--85, December 1995.

\bibitem{Nurunberg202203}
R.~N{\"u}rnberg.
\newblock A structure preserving front tracking finite element method for the
  {M}ullins–{S}ekerka problem.
\newblock {\em Journal of Numer. Math.}, 0(0), 2022.

\bibitem{Pego1989}
R.L. Pego.
\newblock Front migration in the nonlinear {C}ahn-{H}illiard equation.
\newblock {\em Proc. R. Soc. Lond. A}, 422:261--278, 1989.

\bibitem{PrussAndSimonett}
J.~Pr{\"u}ss and G.~Simonett.
\newblock Moving {I}nterfaces and {Q}uasilinear {P}arabolic {E}volution
  {E}quations.
\newblock {\em Birkh{\"a}user}, 2019.

\bibitem{RiceMu1988}
J.~R. Rice and M.~Mu.
\newblock An {E}xperimental {P}erformance {A}nalysis for the {R}ate of
  {C}onvergence of 5-{P}oint {S}tar on {G}eneral {D}omains.
\newblock {\em Purdue University Department of Computer Science Technical
  Reports}, 1988.

\bibitem{Rger2005ExistenceOW}
M.~R{\"o}ger.
\newblock Existence of {W}eak {S}olutions for the {M}ullins-{S}ekerka {F}low.
\newblock {\em SIAM J. Math. Anal.}, 37:291--301, 2005.

\bibitem{Sakakibara_Yazaki}
K.~Sakakibara and S.~Yazaki.
\newblock Structure-preserving numerical scheme for the one-phase {H}ele-{S}haw
  problems by the method of fundamental solutions.
\newblock {\em Computational and Mathematical Methods}, 1(6), November 2019.

\bibitem{Stoth}
Barbara E.~E. Stoth.
\newblock Convergence of the {C}ahn-{H}illiard equation to the
  {M}ullins-{S}ekerka problem in spherical symmetry.
\newblock {\em J. Differential Equations}, 125(1):154--183, 1996.

\bibitem{ZhuChenHou}
J.~Zhu, X.~Chen, and T.Y. Hou.
\newblock An efficient boundary integral method for the {M}ullins–{S}ekerka
  problem.
\newblock {\em J. Comput. Phys.}, 127(2):246–267, 1996.

\end{thebibliography}
\bibliographystyle{plain}
\end{document}